\documentclass[reqno,12pt]{amsart}

\usepackage{graphicx,pdfsync,amssymb, latexsym,amsfonts,amsbsy,color,subcaption,esint,mathtools,enumerate,nicefrac,bm}
\usepackage{amsmath}
\usepackage{hyperref}
\usepackage{booktabs}
\usepackage{multirow}
\usepackage{appendix}
\hypersetup{
colorlinks=true,
linkcolor=blue,
filecolor=blue,
urlcolor=blue,
citecolor=blue,
}

\usepackage{float}
\usepackage{mathrsfs}
\restylefloat{table}

\usepackage{enumitem}
\setitemize{itemsep=0.4em}
\setenumerate{itemsep=0.5em}

\usepackage[utf8]{inputenc}% for special characters
\usepackage[T1]{fontenc}% for special characters
\newtheorem{thm}{Theorem}[section]

\newtheorem{lem}[thm]{Lemma}
\newtheorem{prop}[thm]{Proposition}
\newtheorem{con}{Conjecture}[section]
\theoremstyle{definition}
\newtheorem{defn}{Definition}[section]

\theoremstyle{remark}
\newtheorem{rem}{Remark}[section]
\numberwithin{equation}{section}

\newcommand{\dd}{\mathop{}\!\mathrm{d}}
\newcommand{\supp}{\text{spt}\,}
\newcommand{\MM}{T}
\newcommand{\MMM}{{\bar T}}
\newcommand{\RR}{\mathring R}
\newcommand{\R}{\mathbb R}

\newcommand{\ZZ}{\mathbb Z}
\newcommand{\NN}{\mathbb N}
\newcommand{\vv}[0]{\bar v}
\newcommand{\tb}[0]{\bar \theta}
\newcommand{\ppp}[0]{\bar{\bar p}}
\newcommand{\pex}[0]{\textsf{\textit{p}}}%{p^{\myExact}}
\newcommand{\ootimes}{\mathbin{\mathring{\otimes}}}
\newcommand{\vex}[0]{\textsf{\textit{v}}}%{v^{\myExact}}
\newcommand{\pp}[0]{\bar p}
\newcommand{\RRR}{\mathring{\bar R}}
\newcommand{\wpq}[1][q+1]{w^{\textup{(p)}}_{#1}}
\newcommand{\wcq}[1][q+1]{w^{\textup{(c)}}_{#1}}
\newcommand{\dpq}[1][q+1]{d^{\textup{(p)}}_{#1}}

\newcommand{\Wp}{w^{\textup{(p)}}}
\newcommand{\wc}{w^{\textup{(c)}}}
\newcommand{\Dp}{d^{\textup{(p)}}}
\newcommand{\Rnash}{R_{q+1}^\textup{Nash}}
\newcommand{\Rtransport}{R_{q+1}^\textup{trans}}
\newcommand{\Rosc}{R_{q+1}^\textup{osc}}
\newcommand{\Mtrans}{T_{q+1}^\textup{trans}}
\newcommand{\Mosc}{T_{q+1}^\textup{osc}}
\newcommand{\Mnash}{T_{q+1}^\textup{Nash}}

\newcommand{\de}{d^{(\text{e})}}

\newcommand{\tr}[0]{\textup{Tr}}

\newcommand{\vloc}{u^{{\mathcal{C}}}}
\newcommand{\vnon}{u^{\mathcal{P}}}

 \newcommand{\wloc}{w^{\mathcal{C}}}
\newcommand{\wnon}{w^{\mathcal{P}}}

\newcommand{\Rrem}{R_{q+1}^\textup{rem}}

\newcommand{\tRR}{{\mathring{\widetilde R} }}
\newcommand{\Rtrans}{R_{q+1}^\textup{trans}}

\newcommand{\Prem}{P_{q+1}^\textup{rem} }
\newcommand{\Trem}{T_{q+1}^\textup{rem} }
\newcommand{\matd}[1]{D_{t,#1}}%{(\del_t+#1\cdot\nabla)}
\def\dashint{\,\ThisStyle{\ensurestackMath{%
  \stackinset{c}{.2\LMpt}{c}{.5\LMpt}{\SavedStyle-}{\SavedStyle\phantom{\int}}}%
  \setbox0=\hbox{$\SavedStyle\int\,$}\kern-\wd0}\int}
\def\ddashint{\,\ThisStyle{\ensurestackMath{%
  \stackinset{c}{.2\LMpt}{c}{.5\LMpt+.2\LMex}{\SavedStyle-}{%
    \stackinset{c}{.2\LMpt}{c}{.5\LMpt-.2\LMex}{\SavedStyle-}{%
      \SavedStyle\phantom{\int}}}}\setbox0=\hbox{$\SavedStyle\int\,$}\kern-\wd0}\int}
\usepackage{scalerel,stackengine}
\stackMath
\newcommand\widecheck[1]{%
\savestack{\tmpbox}{\stretchto{%
  \scaleto{%
    \scalerel*[\widthof{\ensuremath{#1}}]{\kern-.6pt\bigwedge\kern-.6pt}%
    {\rule[-\textheight/2]{1ex}{\textheight}}%WIDTH-LIMITED BIG WEDGE
  }{\textheight}%
}{0.5ex}}%
\stackon[1pt]{#1}{\scalebox{-1}{\tmpbox}}%
}

\newcommand{\coloneq}{\mathrel{\mathop:}=}

\newcommand{\ii}{\textup{i}}
\newcommand{\ee}{\textup{e}}
\newcommand{\Div}{\text{div}}
\newcommand{\del}{\partial}

\newcommand{\TT}[0]{\mathsf{T}}
\newcommand{\TTT}[0]{\mathbb{T}}
\newcommand{\uin}[0]{u^{\textup{in}}}
\newcommand{\bin}[0]{b^{\textup{in}}}
\newcommand{\vin}[0]{v^{\textup{in}}}
\newcommand{\tin}[0]{\theta^{\textup{in}}}
\newcommand{\wtq}[1][q+1]{w^{\textup{(t)}}_{#1}}
\newcommand{\dtq}[1][q+1]{d^{\textup{(t)}}_{#1}}
\newcommand{\wt}{w^{\textup{(t)}}}
\newcommand{\dt}{d^{\textup{(t)}}}

\newcommand{\Tosco}{T_{\text{osc},1}}
\newcommand{\Tosct}{T_{\text{osc},2}}
\newcommand{\Rosco}{R_{\text{osc},1}}
\newcommand{\Rosct}{R_{\text{osc},2}}
\usepackage[top=1.1in, bottom=1in, left=1in, right=1in]{geometry}
\newcommand{\thistheoremname}{}

 \newtheorem*{genericthm*}{\thistheoremname}
\newenvironment{namedthm*}[1]
  {\renewcommand{\thistheoremname}{#1}%
   \begin{genericthm*}}
  {\end{genericthm*}}

\usepackage{times}
\usepackage{setspace}
\allowdisplaybreaks[4]
\numberwithin{equation}{section}

\begin{document}
\title{On Onsager-type conjecture for the Els\"{a}sser energies of the ideal MHD equations}

\author{Changxing Miao}
\address[Changxing Miao]{Institute  of Applied Physics and Computational Mathematics, Beijing 100191, P.R. China,}

\email{miao changxing@iapcm.ac.cn}

\author{Yao Nie}

\address[Yao Nie]{School of Mathematical Sciences and LPMC, Nankai University, Tianjin, China.}

 \email{nieyao@nankai.edu.cn}

\author{Weikui Ye}

\address[Weikui Ye]{School of Mathematical Sciences, South China Normal University, Guangzhou,  China}

 \email{904817751@qq.com}
\date{}

\begin{abstract}In this paper, we investigate the ideal magnetohydrodynamics (MHD) equations on tours $\TTT^d$. For $d=3$, we resolve the flexible part of Onsager-type conjecture for Els\"{a}sser energies of the ideal MHD equations. More precisely, for \(\beta < 1/3\), we construct weak solutions \((u, b) \in C^\beta([0,T] \times \mathbb{T}^3)\) with both the total energy dissipation and failure of cross helicity  conservation. The key idea of the proof relies on a symmetry reduction that embeds the ideal MHD system
into a 2$\frac{1}{2}$D Euler flow and the Newton-Nash iteration technique recently developed in \cite{GR}. For $d=2$, we show the non-uniqueness of H\"{o}lder-continuous  weak solutions  with non-trivial magnetic fields. Specifically, for \(\beta < 1/5\), there exist infinitely many solutions \((u, b) \in C^\beta([0,T] \times \mathbb{T}^2)\) with  the same initial data while satisfying the total energy dissipation with non-vanishing velocity and magnetic fields. The new ingredient is developing  a spatial-separation-driven iterative scheme that incorporates the magnetic field  as a controlled perturbation within the convex integration framework for the velocity field, thereby providing sufficient oscillatory freedom for Nash-type perturbations in the 2D setting. As a byproduct, we prove that any H\"{o}lder-continuous Euler solution can be approximated by a sequence of $C^\beta$-weak solutions for the ideal MHD equations in the $L^p$-topology for $1\le p<\infty$.

\end{abstract}
\maketitle
\emph{Keywords}: Ideal MHD system; Onsager-type conjecrure; Energy dissipation; Cross helicity

\emph{Mathematics Subject Classification}: 35Q30,~76D03.

\section{Introduction}
Ideal magnetohydrodynamics (MHD) equations describes the dynamics of electrically conducting fluids such as plasmas \cite{GLL} or liquid metals \cite{ST} in the presence of magnetic fields. This model is widely used in astrophysics to investigate phenomena such as solar flares, stellar winds, and the dynamics of accretion disks. Understanding these intricate physical mechanisms fundamentally relies on the mathematical theories on MHD equations.

 In this paper, we consider the ideal incompressible MHD equations
\begin{equation}
\left\{ \begin{alignedat}{-1}
&\del_t u+u\cdot\nabla  u  +\nabla P   =   b\cdot\nabla  b, &{\rm in}\quad \TTT^d\times(0,T),
 \\
&  \del_t b+u\cdot\nabla b  =b\cdot\nabla u, &{\rm in}\quad \TTT^d\times(0,T),
 \\
&  \nabla \cdot u=\nabla \cdot b  = 0, &{\rm in}\quad \TTT^d\times(0,T),
 %& (v, \theta)\big|_{t=0} = (\vin,  \tin), &{\rm in}\quad \TTT^3,
\end{alignedat}\right.  \label{E}
\end{equation}
associated with initial data $(u, b)\big|_{t=0} = (\uin, \bin)$, where $u: \TTT^d\times [0,T]\to\R^d$ denotes the velocity of the incompressible fluid, $p:\TTT^d\times [0,T]\to\R$ the pressure field, $b:\TTT^d\times [0,T]\to\R^d$  the magnetic field. We are interested in Onsager-type conjecture associated with the ideal MHD equations \eqref{E} and define the
weak solutions:
\begin{defn}[Weak solution]\label{def}Let $T>0$ and let $(\uin, \bin)\in C^{\alpha}(\TTT^d)$ for some $\alpha>0$ be divergence-free in the sense of distributions and have zero mean\footnote{Throughout this paper, ``zero mean'' is ``zero spatial mean''.}. We say that  $(u,b)\in C^0([0, T]\times\TTT^d)$ is a \emph{weak solution}  to the MHD equations~\eqref{E} if \begin{itemize}
     \item [(1)] For a.e. $t\in [0,T]$, $(u(\cdot, t), b(\cdot, t))$ is divergence-free in the sense of distributions;
     \item [(2)]For each divergence-free test function $\phi\in C^\infty_0([0,T)\times \mathbb T^d)$,
     \begin{align}\nonumber
\int_0^T \int_{\mathbb T^d} (\del_t-\Delta)\,\phi u+\nabla \phi : (u\otimes u-b\otimes b) \dd x \dd t=  -\int_{\mathbb T^d} \uin \phi(0,x)\dd x,\\
\int_0^T \int_{\mathbb T^d} (\del_t-\Delta)\,\phi b-\nabla \phi : (u\otimes b-b\otimes u) \dd x \dd t=-\int_{\mathbb T^d}\bin\phi(0,x)\dd x.\nonumber
\end{align}
   \end{itemize}
\end{defn}

\subsection{Onsager-type conjecture}
The symmetrization of the  equations \eqref{E} via the Els\"{a}sser variables $z_{\pm}$ streamlines the derivation of  two globally conserved invariants:  the \textit{total energy} $\mathcal{E}$ and the \textit{cross helicity} $
 \mathcal{H}_{c}$. Specifically, under the  transformation $z_{\pm}=u\pm b$, the governing equations \eqref{E} become
\begin{equation}\label{E-MHD}
\left\{ \begin{alignedat}{-1}
&\partial_t z_{\pm}+z_{\mp}\cdot\nabla z_{\pm}+\nabla p=0,
 \\
&  \nabla \cdot z_{\pm}=0,
\end{alignedat}\right.
\end{equation}
where  $p=P+\frac{1}{2}|b|^2$. For $C^1_{t,x}$ solutions $(u,b)$, the corresponding Els\"{a}sser variables $z_{\pm}$  inherit this regularity. Then taking the
$L^2_x$-inner product of the first equation of \eqref{E-MHD} with $z_{\pm}$ demonstrates that
\begin{align*}
\frac{\mathrm{d}}{\mathrm{d}t} \int_{\mathbb{T}^d} |z_+|^2  \mathrm{d}x = \frac{\mathrm{d}}{\mathrm{d}t} \int_{\mathbb{T}^d} |z_-|^2 \mathrm{d}x = 0,
\end{align*}
which immediately implies the conservation of the \emph{Els\"{a}sser energies} for smooth solutions:
\begin{enumerate}
\item[(i)]The total energy $\mathcal{E}(t)=\mathcal{E}(0)$, where
\begin{align}\label{vL2}
\mathcal{E}(t) = \frac{1}{2} \int_{\mathbb{T}^d} \left( |u(x,t)|^2 + |b(x,t)|^2 \right) \mathrm{d}x = \frac{1}{4} \int_{\mathbb{T}^d} \left( |z_+(x,t)|^2 + |z_-(x,t)|^2 \right) \mathrm{d}x.
\end{align}
\item[(ii)]The cross helicity $\mathcal{H}_c(t)=\mathcal{H}_c(0)$, where
\begin{equation}\label{tLp}
\mathcal{H}_c(t) = \int_{\mathbb{T}^d} u \cdot b\, \mathrm{d}x = \frac{1}{4} \int_{\mathbb{T}^d} \left( |z_+(x,t)|^2 - |z_-(x,t)|^2 \right) \mathrm{d}x.
\end{equation}

\end{enumerate}
Another crucial topological invariant is the \emph{magnetic helicity}, defined as
\begin{align*}
\mathcal{H}_m(t)=\int_{\mathbb{T}^d}a\cdot b\dd x,
\end{align*}
where $a$ is a vector potential of $b$, i.e. $b =\nabla\times a$. The total energy $\mathcal{E}$ and  cross helicity $\mathcal{H}_c$  are well-defined provided that $(u,b)\in L^\infty L^2$,  whereas the magnetic helicity remains well-defined under weaker condition $b\in L^\infty \dot H^{-1/2}$. This hierarchical regularity structure intuitively aligns with Taylor's conjecture that magnetic helicity  persists as an invariant for small resistivity (\cite{Tay74}, \cite{Tay86}). The mathematical formulation of the  Taylor's  conjecture was rigorously established  in \cite{FL}.

The MHD equations \eqref{E} reduces to the Euler equations when the magnetic field $b$  vanishes. For the Euler equations, Onsager \cite{Onsa} conjectured that weak solutions conserve energy if their H\"{o}lder regularity $\alpha>1/3$ while they violate the energy conservation law when $\alpha\le 1/3$. The positive/rigid part of the conjecture in 3D  was established by Constantin, E., and Titi  \cite{CWT},  and their method is dimension-independent and valid in any spatial dimensions. The proof of  the negative/flexible side required significantly greater effort, involving a series of pivotal works such as  \cite{B15, BDIS15, BDS, DS09, DS14, Ise17, N20, Sh00}.  Isett \cite{Ise18} and Giri-Radu \cite{GR} ultimately resolved  the flexible part of  Onsager's conjecture in 3D and 2D respectively,  by  constructing weak solutions in $C_t(C^{1/3-})$ with  compact temporal support, leaving the endpoint $\alpha=1/3$ open.

  %Roughly in parallel with the above developments, the non-uniqueness of weak solutions to the Euler equations was also investigated, which is related to a version of Onsager's conjecture for the Cauchy problem.

Such a regularity-induced dichotomy in energy behavior for the Euler equations  naturally motivates us to investigate the critical threshold below which the total energy, cross helicity, and magnetic helicity are no longer preserved for weak solutions of ideal MHD system \eqref{E}. Regarding the Onsager-type conjecture for magnetic helicity, some significant progress has been made in both its rigidity and flexibility aspects. For the rigid part, Caflisch et al. \cite{CKS} established the conservation result for solutions  $(u,b)\in C_tB^{\alpha}_{3,\infty}$ with $\alpha>0$. Subsequently, the authors in \cite{Alu09, KL07} extended it to
$L^3_{t,x}$-space. For the flexible part, Beekie et al. \cite{BBV}  showed weak solutions that do not conserve magnetic helicity in $C_t H^{\beta}(\TTT^3)$ for $0<\beta\ll1$. Ultimately, Daniel et al. \cite{DSS} proved the sharpness of the $L^3$ integrability criterion for the magnetic helicity conservation, and  resolved the flexible part.

On the total energy and the cross helicity, Buckmaster and Vicol \cite{BV21} proposed the following conjecture.
\begin{con}[Onsager-type conjecture for the Els\"{a}sser energies \cite{BV21}]\label{conjec}\,
\begin{enumerate}
\item[(a)] Any weak solution $(u, b)$ of the ideal MHD system \eqref{E} belonging to
$C_{x,t}^\alpha$ or $L_t^3 B_{3,\infty}^\alpha$ for $\alpha > 1/3$ conserves the total energy $\mathcal{E}$ and the cross helicity $\mathcal{H}_{c}$.
\item[(b)] For any $\alpha < 1/3$, there exist weak solutions $(u, b) \in C_{x,t}^\alpha$ or $L_t^3 B_{3,\infty}^\alpha$ which dissipate the total energy $\mathcal{E}$, and for which the cross helicity $\mathcal{H}_{c}$ is not a constant function of time.
\end{enumerate}
\end{con}
\noindent The rigid  part $(a)$ of Conjecture \ref{conjec} has been resolved by the standard arguments in \cite{CKS,CWT}. Concerning the flexible part $(b)$ of Conjecture \ref{conjec},
the work \cite{BDSV} related to Onsager's conjecture for the 3D Euler equations implies  that there exist weak solutions $(u,b)=(u,0)\in C^{\alpha}_{t,x}$  for $\alpha<\frac{1}{3}$
  with energy dissipation, yet the
cross helicity is identically zero. Faraco et al. \cite{FLS}  constructed  some compactly supported $L^\infty$ weak solutions  with non-trivial $u$ and $b$, where the total energy and the cross helicity are not conserved over time, while the magnetic helicity vanishes identically in $\R^3$. Beekie et al. \cite{BBV} provided the first example of a non-conservative weak solution  $(u,b)\in C_tL^p(\TTT^3)$ for some $p>2$, which are of non-trivial total energy, the cross helicity and the magenetic helicity. However, the flexibility assertion part $(b)$ of Conjecture \ref{conjec} remains open.
%However, to date, the constructed ideal MHD weak solutions that fail to conserve both energy and cross-helicity have not reached the expected Onsager critical exponent of $1/3$.

In this paper, we are focused on the flexible aspects of the Onsager-type conjecture for the Els\"{a}sser energies in the context of ideal magnetohydrodynamics (MHD). Our approach builds on the method of convex integration, which has been applied to study the classical MHD system \cite{LZZ, MY}, and other fluid dynamics models, for instance, the Hall MHD \cite{Dai}, the stationary Navier-Stokes equations \cite{Luo},  the transport equations
\cite{BCD, CL21, CL22, MoS, MS}, the Boussinesq equations \cite{MNY, TZ17, TZ18}.
\subsection{The main results}
Our first goal is to construct non-unique weak solutions in $C^{\beta}_{t,x}(\beta<1/3)$ for the 3D ideal MHD equations, which  violate both the total energy conservation and cross-helicity preservation. More precisely,
\begin{thm}\label{main0}For any $0<\beta<\frac{1}{3}$ and $T>0$, there are infinitely many
weak solutions  $(u,b)\in C^{\beta}(\TTT^3\times[0,T])$ of the system~\eqref{E} sharing the same initial data such that neither
total energy $\mathcal{E}$ nor cross helicity $ \mathcal{H}_{c}$ is conserved in time.
\end{thm}
The solutions in Theorem \ref{main0}  violate energy conservation. Whether they exhibit energy dissipation remains unknown. In Theorem \ref{t:main-energy1}, we  further show the weak solutions  $(u,b)\in C^\beta_{t,x}$ with  $\beta<1/3$ that dissipate energy and exhibit non-conserved cross helicity.
\begin{thm}\label{t:main-energy1} For any  $0<\beta<\frac{1}{3}$ and $T>0$, there exist energy dissipative weak solutions $(u,b)\in C^{\beta}(\TTT^3\times[0,T])$ of the system~\eqref{E} that  do not conserve the cross helicity.
\end{thm}
\begin{rem}Theorem \ref{t:main-energy1} establishes the flexible side of Conjecture~\ref{conjec} on $\TTT^3$. This result, together with the proof of rigid part \cite{CKS,CWT} shows that the threshold regularity exponent for the conservation of the total energy and the cross helicity is  $1/3$.
\end{rem}
For the 2D ideal MHD system,  the magnetic helicity is replaced by the mean-square magnetic potential $\int_{\TTT^2}|\psi|^2\dd x$,
where $\psi$ is the magnetic potential such that  $ b=\nabla^\perp\psi$.  This quantity is  conserved for {$C_wL^2$} weak solutions as in \cite{FLS}, exhibiting a stronger type of rigidity compared to its 3D counterpart. However, the flexibility of the Onager conjecture for Els\"{a}sser energies in the context of 2D ideal MHD equations with nontrivial magnetic fields remains unexplored.  This  motivates us to study the existence of energy-dissipative weak solutions  of the 2D ideal MHD equations. Our result is stated as follows:
\begin{thm}\label{main00}For any  $0<\beta<\frac{1}{5}$ and $T>0$, there exist infinely many solutions $(u,b)\in C^{\beta}([0,T]\times \TTT^2)$ of the system~\eqref{E} sharing the same initial data with both $u$,$b$ non-trivial and  the total energy dissipation.
\end{thm}

\begin{rem}Recently, Giri and Radu  in  \cite{GR} presented a nontrivial weak solution to the 2D Euler equations with temporal compact support, and  resolved the flexibility aspect of the Onsager conjecture. For the velocity field $u$ constructed as in \cite{GR}, the  pair $(u,0)\in C^{\beta}$ $(\beta<1/3)$ is a weak solution of the system \eqref{E}  that fails to conserve energy but does not exhibit energy dissipation.
\end{rem}

\begin{rem}
As discussed in  \cite{BV21}, it is nontrivial to construct weak solutions that conserve the magnetic helicity while do not conserve the Els\"{a}sser energies due to the capability of  convex integration schemes to ``break'' all quadratic conservation laws  which are well-defined at the regularity level of the weak solutions constructed. Specifically, the weak solutions constructed in Theorem~\ref{t:main-energy1} and Theorem \ref{main00} conserve magnetic helicity and magnetic potential,  respectively.
\end{rem}

 Our proofs rely on the convex integration scheme introduced by De Lellis and Sz\'{e}kelyhidi~\cite{DS2013}.  For the three-dimensional case, the main diffculity  is that the Reynolds-magnetic
tensor errors in the relaxed ideal MHD system restrict the oscillation freedom.  To overcome this diffculity, we consider a special class of weak solutions $(v, \theta)$ to 2D incompressible Euler equations coupled with a passive tracer equation (see \eqref{e:B}). The corresponding relaxation  gives sufficient oscillatory freedom to construct perturbations.
 By making use of
 the asymmetry of the errors, we construct energy correctors
decoupled from the Nash perturbations of $v$, and improve the Newton-Nash iteration scheme established in \cite{GR},
thereby obtaining the energy-dissipative solutions that reach the threshold regularity exponent $1/3$.

In the two-dimensional case, we propose a\textit{ spatial-separation-driven} iterative scheme. By splitting the velocity field \( u_q = u^{\mathcal{C}}_q + u^{\mathcal{P}}_q \), we enforce spatial disjointness between the supports of \( u^{\mathcal{C}}_q \) and the magnetic field \( b_q \). The spatial decoupling reduces the relaxed system to one that only contains  Reynolds stress error. By introducing the MHD perturbation flows together with finite propagation speed of the magnetic field, we embed the magnetic field as a controlled perturbation within the convex integration framework for the velocity field.

\subsection{Main ideas}We prove  Theorem \ref{main0} by constructing a special class of weak solutions of \eqref{E}, which originates from the example of Bronzi-Lopes Filho-Nussenzveig Lopes \cite{BLN}. They imposed a symmetry assumption that reduces the 3D ideal MHD system to a 2$\frac{1}{2}$D  Euler flow. Specifically, if $u = (v_1, v_2, v_3)(x_1, x_2)$ is a weak solution of the 3D Euler equations independent of $x_3$, then defining the velocity field as $(v_1, v_2,0)$ and the magnetic field as $(0,0,v_3)$  yields a weak solution to the 3D ideal MHD system. As shown in  \cite{Ise18}, such weak solution $(u,b)$ may violate energy conservation, while the cross helicity $\mathcal{H}_c=0$ due to the trivial third velocity component. To show weak solution with nontrivial cross helicity, this symmetry assumption inspires us to consider a special class of solutions $u=(v_1,v_2,\theta )(x_1, x_2)$ and $b=(0,0,\theta)(x_1,x_2)$, where $(v,\theta)$ solves the following system:
\begin{equation}
\left\{ \begin{alignedat}{-1}
&\del_t v+(v\cdot\nabla) v  +\nabla P   = 0, &{\rm in}\quad \TTT^2\times(0,T),
 \\
&  \del_t \theta+v\cdot\nabla\theta  =0, &{\rm in}\quad \TTT^2\times(0,T),
 \\
&  \nabla \cdot v  = 0, &{\rm in}\quad \TTT^2\times(0,T),
\end{alignedat}\right.  \label{e:B}
\end{equation}
which is the 2D incompressible Euler equations  coupled with  a passive tracer equation. Therefore,  Theorem \ref{main0} can be reduced to  the non-uniqueness of the system \eqref{e:B}.  Different from the 3D ideal MHD system \eqref{E},
the scalar $\theta$ is advected purely by the velocity field  $v$ in the system \eqref{e:B}. This avoids the use of the geometric lemma on skew-symmetric matrices. Nevertheless, within a two-dimensional framework, the fact that any two non-parallel lines necessarily intersect implies that Mikado flows cannot achieve non-interaction. Inspired by the approach of exploiting the extra dimension of time introduced in \cite{CL, GR}, we employ the gluing technique introduced in \cite{Ise18} to construct  approximation solutions such that the Reynolds errors support in disjoint short time intervals. This procedure not only  facilitates us to use temporal dimension, but also directly obtain estimates of the material derivative. In the convex integration stage, we employ the Nash-Newton iteration method introduced in the seminal work of \cite{GR}. We define the Newton perturbations to be the solutions to the Newtonian linearization of the Euler equations coupled with a passive tracer, where  the forcing term is augmented by temporal highly oscillatory functions. Based on the  vector error in the tracer equation and the geometric lemma for symmetric matrices, we construct Nash perturbations by selecting shear flows with different oscillatory directions. Finally, by employing the idea as in \cite{BCV}, we glue the weak solution constructed by convex integration with two smooth  solutions on initial and final temporal segments, thereby generating infinitely many weak solutions with same initial data that fail to conserve total energy and the cross helicity.

We prove Theorem \ref{t:main-energy1} by adapting the convex integration scheme developed for Theorem~\ref{main0}, preserving the strategy for constructing weak solutions
$(v,\theta)$ to the system \eqref{e:B} with $v=(v_1,v_2)$. To obtain weak solutions with energy dissipation, we impose a precised energy profile $e(t)$  in the iterative scheme.  Utilizing the asymmetric cross errors  in the tracer equation, we design energy correctors as  a part of the perturbation for the passive tracer $\theta$.  These correctors   are compactly supported in square domains  which are  disjoint with the supports of the Nash perturbations of the velocity field. This spatial separation  eliminates the interactions between the correctors and velocity perturbations such that the resulting error terms are negligible. For the special solutions $u=(v_1, v_2,\theta), b=(0,0,\theta)$, the cross helicity depends only on $\theta$
and  the perturbations of $(v_1,v_2)$ vanishes at time $t=0,T$. Combining this fact with  the selection of an energy profile $e(t)$ and  an initial solution $u_1$ such that $e(0)\neq e(T)$ and $u_1(0)\neq u_1(T)$, we establish  $\mathcal{H}_c(0)\neq \mathcal{H}_c(T)$, and show non-conservation of cross helicity.

To prove Theorem \ref{main00}, we develop different arguments  from those in  Theorems~\ref{main0} and \ref{t:main-energy1}, as symmetry reduction techniques  are  not applicable in the 2D framework. As we know, for any smooth solution $u$ of the incompressible 2D Euler equations, there exists a sequence of ideal MHD smooth solutions $(u_{\epsilon}, b_{\epsilon})$ such that
$\|(u_{\epsilon}, b_{\epsilon})-(u,0)\|_{H^{2+}}\to 0$ as $\epsilon\to 0$. A natural question is:
\emph{Given a weak solution $u$ of the Euler equations, does there exist a sequence of MHD weak solutions
$(u_{\epsilon}, b_{\epsilon})$ that converges to
$(u, 0)$ in some weak topology?}
If affirmed, we could expect the non-uniqueness  of weak solutions to the 2D ideal MHD equations, following from the non-uniqueness results on 2D Euler equations. Actually, this problem is non-trivial due to the low regularity of the  velocity field. A key insight arises from the perturbation effect of the magnetic field  in the process of approximating smooth Euler solutions by smooth MHD solutions. More precisely, we develop a spatial-segregation-
driven  convex integration scheme in which the approximate solutions $(u_q, b_q,p_q,\RR_q)$ possess more refined properties. More precisely,
\[u_q=u^{\mathcal{C}}_{q}+u^{\mathcal{P}}_q, \quad \supp_x u^{\mathcal{C}}_{q}\cap \supp_x b_q=\emptyset, \quad \supp_x \RR_{q}\cap \supp_x b_q=\emptyset\]
where $u^{\mathcal{P}}_{q}$  is smooth and small.  The disjoint support condition  ensures that $b_q$ is the Lie transport of the smooth component   $u^{\mathcal{P}}_q$, and thus allowing us to consider a relaxed ideal MHD system that contains only the Reynolds stress error, with its support  consistent with \( u^{\mathcal{C}}_q \). To implement the iterative procedure, we construct different perturbations to different components of $(u_q, b_q)$.
\begin{enumerate}
    \item For  $u^{\mathcal{C}}_{q}$, we construct a divergence-free perturbation $w^{\mathcal{C}}_{q+1}$, where the main perturbation part $\wpq$ is constructed by high-frequency oscillations in phase space to reduce the size of Reynolds stress error, and the correction \( \wcq \) ensures the divergence-free condition.
    \vskip 2mm
     \item For  $(u^{\mathcal{P}}_{q},b_q)$, we construct   the MHD perturbation flows $(d_{q+1}, \theta_{q+1})$ to cancel the errors that interference with the support of Reynolds stress at the next iteration. Specifically, $(\theta_{q+1},d_{q+1})$ is defined as the solution to the 2D ideal MHD equations, where these errors are as external forcing terms. Actually,  these errors can be small by leveraging the high-frequency oscillations of the perturbation $w^{\mathcal{C}}_{q+1}$
 ,  and  thus  $(\theta_{q+1},d_{q+1})$ remains small as desired.
\end{enumerate}
The smallness condition on \( u^{\mathcal{P}}_q \) and the finite propagation speed  to hyperbolic systems ensure that the compact supports of \( u^{\mathcal{C}}_{q+1} \) and \( b_{q+1} \) remain non-overlapping, and thus maintaining the self-consistency of the iterative scheme.

%For each given H\"{o}lder-continuous Euler solution $u_E$, we construct an  initial solution $(u_1, b_1)$ through mollifying $u_E$ and solving the 2D ideal MHD equations with specific forcing terms. By employing the above iterative scheme, we obtain a weak solution $(u_{\epsilon}, b_{\epsilon})$ of the ideal MHD equations that is close to $(u_E, 0)$, thereby  proving Theorem  \ref{app}.
From the proof of Theorem \ref{main00}, we demonstrate that H\"{o}lder-continuous solutions of the Euler equations  can be uniformly approximated in $L^p$
  by weak solutions of the ideal MHD equations.
\begin{thm} \label{app}
Let \( 1 \leq p < \infty \), \( \gamma > 0 \), and \( \beta < \frac{1}{5} \). Assume that \( u_E \in C^\gamma([0, T] \times \mathbb{T}^2) \) is a weak solution of the incompressible Euler equations. Then there exists a sequence of weak solutions \( (u_\epsilon, b_\epsilon) \in C^\beta([0, T] \times \mathbb{T}^2) \) to the system \eqref{E} such that
\[
(u_\epsilon, b_\epsilon) \to (u_E, 0) \quad \text{in}\quad L^p, \quad \text{as}\quad\epsilon \to 0.
\]
 \end{thm}
\vskip 2mm
\noindent{\textbf{Organization}:}
In Section \ref{sec2}, we prove Theorems \ref{main0}  and \ref{t:main-energy1}; Section \ref{sec3} is devoted to the proofs of Theorems \ref{main00}  and \ref{app}. Finally, the appendix contains essential technical tools including geometric lemmas used for perturbation constructions, estimates for the inverse divergence iteration step, and other crucial analytical instruments.

\section{Proofs of Theorem \ref{main0} and Theorem \ref{t:main-energy1}}\label{sec2}
\subsection{Proof of Theorem \ref{main0}}
As previously discussed,  if the vector field $v(x_1, x_2)=(v_1, v_2)(x_1, x_2)$ and the scalar function $\theta(x_1, x_2)$ solve the  system \eqref{e:B} in weak sense, then
\[u(x_1, x_2)=(v_1, v_2, \theta)(x_1, x_2),\quad b(x_1, x_2)=(0,0,\theta)(x_1, x_2),\]
is a weak solution of the 3D ideal MHD system \eqref{E} independent of $x_3$ since $u\cdot \nabla b=b\cdot \nabla u=0$ in weak sense.  Hence, the proof of Theorem \ref{main0} can be reduced to demonstrating the non-uniqueness of weak solutions to the system \eqref{e:B}. More precisely, we show the following main proposition.
\begin{prop}\label{t:main}
Let  $(v^{(1)}, \theta^{(1)})\in C([0, T_1]; C^\infty(\TTT^2))$  and $(v^{(2)}, \theta^{(2)})\in C([0, T_2]; C^\infty(\TTT^2))$ solve the system \eqref{e:B} with mean-free initial data $(v^{(1)}(0, x)$ $ \theta^{(1)}(0,x))$, $(v^{(2)}(0, x), \theta^{(2)}(0, x))\in C^\infty(\TTT^2)$, respectively.  Fixed $ \widetilde{T}\le \tfrac{1}{4}\min\{T_1, T_2\}$ and $0<\beta<\frac{1}{3}$, there exists a weak solution $(v, \theta)\in C^{\beta}_{t,x}$ of the Cauchy problem for the system \eqref{e:B}  satisfying
\begin{align}\label{TW}
(v, \theta)\equiv(v^{(1)}, \theta^{(1)})\,\,\text{\rm on}\,\,[0, 2 \widetilde{T}],\quad \text{ and}\quad(v, \theta)\equiv(v^{(2)}, \theta^{(2)})\,\,\text{\rm on}\,\,[3\widetilde{T}, T_2].
\end{align}
\end{prop}
\noindent \textbf{Proposition \ref{t:main} implies Theorem \ref{main0}.} For each smooth initial data \((v^{(1)}(0, x), \theta^{(1)}(0, x))\), there exists a unique smooth solution to system \eqref{e:B} on the time interval \([0, T_1]\). Let \(A > 0\) be a constant and \(f: \mathbb{T} \to \mathbb{R}\) be a smooth periodic function. We define
\[
v^{(2)} = \big(0, A f(x_1)\big), \quad \theta^{(2)} = A f(x_1),
\]
which constitute a global smooth solution to system \eqref{e:B}. By Proposition \ref{t:main}, for any \(T > 0\), there exists a weak solution \((v, \theta) \in C^\beta_{t,x}\) to the Cauchy problem for system \eqref{e:B} that satisfies the gluing conditions:
\[
(v, \theta) \equiv (v^{(1)}, \theta^{(1)}) \quad \text{on } [0, T_1/2], \quad \text{and} \quad (v, \theta) \equiv (v^{(2)}, \theta^{(2)}) \quad \text{on } [2T_1, T].
\]
By changing the amplitude \(A\) and profile \(f\), we  can construct  infinitely many distinct weak solutions \((v, \theta)\) of the system \eqref{e:B}. Hence,  we obtain infinitely many weak solutions \((u, b)\) to the 3D ideal MHD system \eqref{E} by defining
\[
u = (v, \theta)(x_1, x_2), \quad b = (0, 0, \theta)(x_1, x_2).
\]
Furthermore, selecting \(A\) sufficiently large ensures that for \(t \geq 2T_1\), we conclude that the energy \(\mathcal{E}(t)\) and cross-helicity \(\mathcal{H}_c(t)\)
\[
\mathcal{E}(t) = \int_{\mathbb{T}^2} |v|^2 + 2|\theta|^2 \, \dd x = 2A \int_{\mathbb{T}} |f(x_1)|^2 \, \dd x > \int_{\mathbb{T}^2} |v^{(1)}(0)|^2 + 2|\theta^{(1)}(0)|^2 \, \dd x,
\]
\[
\mathcal{H}_c(t) = \int_{\mathbb{T}^2} |\theta|^2 \, \dd x = A^2 \int_{\mathbb{T}} |f(x_1)|^2 \, \dd x > \int_{\mathbb{T}^2} v^{(1)}(0) \cdot \theta^{(1)}(0) \, \dd x.
\]
 are non-conserved. This completes the proof of Theorem \ref{main0}.

Now, it suffices to show Proposition \ref{t:main}.  Without loss of generality, we assume
$T_1=T_2$ in  Proposition \ref{t:main}.  We reduce  Proposition \ref{t:main} to  an iterative proposition. Before stating this  proposition, we explicitly define all parameters and denotes.
\vskip 2mm
\noindent{\textit{Parameters}.} For all $q\ge 1$ and given  $0<\beta<1/3$, we define
\begin{align}\label{con-b}
 b_0:={\min\big\{1+\tfrac{1-3\beta}{12\beta},  \tfrac{3}{2}\big\}}.
\end{align}
For any $a\gg 1$ and  $b\in(1, b_0)$, we give
\begin{align}\label{lambdaq}
    \lambda_q \coloneq  \left\lceil a^{b^q}\right\rceil,\qquad   \delta_q \coloneq \lambda_2^{3\beta}\lambda_q^{-2\beta},
\end{align}
here $\lceil \cdot\rceil$ denotes the ceiling function. Given $b$ and $\beta$, %there exists a $\alpha_0>0$ depending on $b, \beta$ such that for any $0<\alpha<\alpha_0(b, \beta)$,
%\begin{align}\label{alpha0}
 %   \frac{\delta^{1/2}_{q+1}\delta^{1/2}_q\lambda_q}{\lambda_{q+1}}\lesssim\frac{\delta_{q+2}}{\lambda^{8\alpha}_{q+1}}.
%\end{align}
%we define  $100\le N_0(b,\beta)\in \ZZ_{+}$ such that
%\begin{equation}\label{lambdaN}
 %\frac{1}{\lambda_{q+1}^{N_0-\alpha} \ell_q^{N_0+\alpha}} \leq \frac{1}{\lambda_{q+1}^{1-\alpha}}, \quad \forall q\ge 1.
 %\end{equation}
 %Then we give
 we define the positive integer $L$ such that
 \[L\coloneq {\max\Big\{\Big\lceil \tfrac{100\beta b^2}{-(2\beta b^2-b(3\beta+\frac{1}{3})+ \beta+\frac{1}{3})}\Big\rceil , \Big\lceil \tfrac{100\beta b}{\frac{1}{3}-\beta }\Big\rceil, N_0\Big\}+1.}\]We  require  $\alpha$  to satisfy
\begin{align}\label{e:params0}
   &\alpha =\min\{\tfrac{1}{100}, -\tfrac{1}{100L}(2\beta b^2-b(\beta+\tfrac{1}{3})+ \beta+\tfrac{1}{3}), \tfrac{1}{100L}(b-1)\beta(\tfrac{1}{3}-\beta)\}.
\end{align}
 Let $0<\widetilde{T}\le\tfrac{1}{4}$, $M$ and $r_0$ be two universal constants from geometric Lemma \ref{first S}. In the following, $a\in\mathbb{N}^+$ is  a large number depending on $b,\beta,\alpha$ and the initial data such that
 \[a > \max\big\{50^{\beta/\alpha },  M^{\frac{100}{\alpha}}, {(r_0/2)}^{-\frac{2}{\alpha}}, 3{\widetilde{T}^{-1}}\big\}.\]
Let $\ell_q$, $\tau_q$ and $\mu_{q+1}$ be given by
\begin{align} \label{e:ell}
    \ell_q \coloneq\lambda^{-1+\frac{1}{10}(2\beta b^2-b(3\beta+\frac{1}{3})+ \beta+\frac{1}{3})}_q,\quad
     \tau_q \coloneq \delta_q^{-\frac{1}{2}} {\lambda^{-1-5\alpha}_q},\quad\mu_{q+1}\coloneq \lambda^{\frac{1}{3}}_{q+1}\lambda^{\frac{2}{3}+5\alpha}_q\delta^{\frac{1}{2}}_{q+1},
\end{align}
and
\[\delta_{q+1,n}\coloneq (\tau_q\mu_{q+1}\ell^{10\alpha}_q)^{-n}\delta_{q+1 }, \quad 0\le n\le L.\]
The relation between $L$ and $\alpha$ leads to
\begin{equation}\label{lambdaN}
\lambda^L_q\ell_q^{L}+\delta_{q+1,L}
+\frac{\lambda_q}{\lambda_{q+1}}\delta^{\frac{1}{2}}_{q+1}\delta^{\frac{1}{2}}_q
+\ell^{-2}_q\lambda^{-2}_q\Big(\frac{\lambda_q}{\lambda_{q+1}}\Big)^{\frac{1}{3}}
 \Big(\frac{\delta_{q+1}}{\delta_{q}}\Big)^{\frac{1}{2}}\delta_{q+1}\leq  \delta_{q+2 }\lambda_{q+1}^{-4\alpha}, \quad \forall q\ge 1.
\end{equation}
We consider  the  following relaxation  of  the system \eqref{e:B}
 \begin{equation}
\left\{ \begin{alignedat}{-1}
&\del_t v_q-\Div (v_q\otimes v_q)  +\nabla p_q   = \Div \RR_q,
 \\
 &\del_t \theta_q+v_q\cdot\nabla \theta_q     = \Div \MM_q,
 \\
  &\nabla \cdot v_q = 0,
  \\ &(v_q, \theta_q) |_{t=0}=(\vin,\tin):=(v^{(1)}(0, x), \theta^{(1)}(0, x)),
\end{alignedat}\right.  \label{e:subsol-B}
\end{equation}
where
\begin{equation}\label{p_q}
\int_{\TTT^2}p_q\dd x=0,
\end{equation}
$\RR_q$  is a symmetric trace-free $2\times2$  matrix and $\MM_q$ is a vector field. For the system \eqref{e:subsol-B}, we shall establish the following iterative proposition.
\begin{prop}\label{p:main-prop}
Let $b\in(1,b_0)$, $\alpha$ satisfy  \eqref{e:params0} and $0<\widetilde{T}\le 1/4$.
If $(v_q, \theta_q, p_q,\RR_q, \MM_q)$ obeys the equations \eqref{e:subsol-B} and \eqref{p_q} with
\begin{align}
     &\|(v_q, \theta_q)\|_{0} \le M\sum_{i=1}^q\delta^{1/2}_i ,
    \label{e:vq-C03}
    \\
    &\|(v_q, \theta_q)\|_{N} \le M\delta_{q}^{1/2} \lambda^N_q ,\quad 1\leq N\leq L,
    \label{e:vq-C1}
    \\
    &\|(\RR_q, \MM_q)\|_{0} \le \delta_{q+1}\lambda_q^{-4\alpha} ,
    \label{e:RR_q-C01}
\\
&(v_q, \theta_q)=(v^{(1)}, \theta^{(1)}) \,\,\text{\rm on}\,\, [0, 2\widetilde{T}+\tau_q],\label{e:initial1}\\
 &(v_q, \theta_q)= (v^{(2)}, \theta^{(2)})\,\,\text{\rm on}\,\,[3\widetilde{T}-2\tau_q],\label{e:initial}
\end{align}
where $(v^{(1)}, \theta^{(1)})$ and $ (v^{(2)}, \theta^{(2)})$ are consistent with those in Proposition \ref{t:main}, then there exist a smooth solution $(v_{q+1}, \theta_{q+1}, p_{q+1}, \RR_{q+1}, \MM_{q+1})$ satisfying \eqref{e:subsol-B}--\eqref{e:initial}
with $q$ replaced by $q+1$ and
\begin{align}
        \big\|(v_{q+1}, \theta_{q+1}) - (v_q, \theta_q)\big\|_{0} +\frac1{\lambda_{q+1}} \big\|(v_{q+1}, \theta_{q+1})-(v_q, \theta_q)\big\|_1 &\le   M\delta_{q+1}^{1/2}.
        \label{e:velocity-diff1}
\end{align}
\end{prop}
\noindent \textbf{Proposition \ref{p:main-prop} implies Proposition \ref{t:main}.}\quad Firstly, we construct $(v_1, \theta_1)$ by gluing $(v^{(1)}, \theta^{(1)})$ and $(v^{(2)}, \theta^{(2)})$. One can show that $(v_1, \theta_1, \RR_1, \MM_1)$ satisfies \eqref{e:vq-C03}--\eqref{e:initial} by choosing $a$ large enough. Using Proposition  \ref{p:main-prop} inductively, we construct a sequence of solutions  $\{(v_q, \theta_q, p_q,\RR_q,\MM_q)\}$ to the system \eqref{e:subsol-B} satisfying \eqref{e:vq-C03}--\eqref{e:initial}. Thanks to \eqref{e:velocity-diff1},    the limit function $(v,\theta)\in C^{\beta''}(\TTT^3\times[0,1])$ ($\beta''<\beta<\frac{1}{3}$) solves the system \eqref{e:B} by virtue of \eqref{e:velocity-diff1}. Readers can refer to \cite[Section 3.1]{MNY} for the detailed proof and we omit it here for brevity.

We now turn our attention to  prove Proposition \ref{p:main-prop}, which proceeds in five  steps: mollification, gluing procedure, construction of the perturbations, estimates for the perturbations and the tensor errors.

\noindent\textit{Step 1: Mollification.}\,\,By making using of the spatial mollifier $\psi_{\ell_q}$ defined in \eqref{e:defn-mollifier-x},
we define  $(v_{\ell_q}, \theta_{\ell_q}, p_{\ell_q}, \RR_{\ell_q}, \MM_{\ell_q})$ by
\begin{align}
   & v_{\ell_q} \coloneq v_q * \psi_{\ell_q}, \quad\theta_{\ell_q} \coloneq \theta_q * \psi_{\ell_q},  \quad  p_{\ell_q}\coloneq p_q *\psi_{\ell_q} -|v_q|^2 + |v_{\ell_q}|^2,\label{moll-v}\\
 &   \RR_{\ell_q} \coloneq \RR_q * \psi_{\ell_q}  - (v_q \ootimes v_q) * \psi_{\ell_q}  + v_{\ell_q} \ootimes v_{\ell_q} , \label{moll-R}\\
 &\MM_{\ell_q} \coloneq \MM_q * \psi_{\ell_q}  - (v_q \theta_q) * \psi_{\ell_q}  + v_{\ell_q}\theta_{\ell_q},\label{moll-T}
\end{align}
which solves the following equations
\begin{equation}
\left\{ \begin{alignedat}{-1}
&\del_t v_{\ell_q} +\Div (v_{\ell_q}\otimes v_{\ell_q})  +\nabla p_{\ell_q}   =  \Div \RR_{\ell_q}  ,
\\
&\del_t\theta_{\ell_q}+v_{\ell_q}\cdot\nabla\theta_{\ell_q}=\Div\MM_{\ell_q},\\
 & \nabla \cdot v_{\ell_q} = 0,
  \\
  &( v_{\ell_q}, \theta_{\ell_q})\big|_{t=0}= (\vin*\psi_{\ell_{q}}, \tin*\psi_{\ell_{q}}).
\end{alignedat}\right.  \label{e:mollified-euler}
\end{equation}
By the standard method as in \cite[Propistion 2.2]{BDSV}, using $\lambda^{L}_q\ell^{L}_q\lesssim\delta_{q+2}\lambda^{-6\alpha}_{q+1}$ and $\lambda^{-3\alpha}_q\le \ell^{\frac{9}{4}\alpha}_q$,  we easily conclude the following bounds.
\begin{prop}[Estimates for mollified functions \cite{MNY}]\label{p:estimates-for-mollified1}
\begin{align}
&\|(v_{\ell_q}-v_{q}, \theta_{\ell_q}-\theta_{q})\|_{0} \lesssim \delta_{q+2} \lambda_{q+1}^{-5\alpha}, \label{e:v_ell-vq1}
\\
&\|(v_{\ell_q}, \theta_{\ell_q})\|_{N+1} \lesssim M\delta_{q}^{1/2} \lambda_{q} \ell_q^{-N}, && \forall N \geq 0, \label{e:v_ell-CN+1}
\\
&\|(\RR_{\ell_q}, \MM_{\ell_q})\|_{N+\alpha} \lesssim  \delta_{q+1} \ell_q^{-N+2\alpha}, && \forall N \geq 0 .\label{e:R_ell}
\end{align}
\end{prop}
\noindent{\textit{Step 2: Gluing procedure.}} Let $\{t_i\}$ be a time sequence given by
\begin{align*}
&t_0=2\widetilde{T}, \quad t_i\coloneq t_0+ i\tau_q, i\in\NN,\\
&i_{\max}=\sup\{i\ge 1, \tau_q\le3\widetilde{T}-t_i<2\tau_q\}+1.
\end{align*}
and the classical exact flows $\{(\vex_i, {\bm{\theta}_i}, \pex_i)\}_{0\le i\le i_{\max}}$
defined as
\begin{equation}\nonumber
(\vex_0, {\bm{\theta}_0}, \pex_0)=(\vex_{i_{\max}}, {\bm\theta_{i_{\max}}}, \pex_{i_{\max}}):=(v_q, \theta_q, p_q).
 \end{equation}
  For the following  system
 \begin{equation}
 \left\{ \begin{alignedat}{-1}
&\del_t \vex +(\vex\cdot\nabla) \vex  +\nabla \pex   = 0,
\\
&\del_t  \bm\theta +(\vex\cdot\nabla)  \bm\theta  =0,
\\
&  \nabla \cdot \vex  = 0,
\end{alignedat}\right.  \label{e:exact-B}
\end{equation}
the local well-posedness theory ensures that $\{(\vex, {\bm{\theta}}, \pex)\}$ is a unique smooth solution of the system \eqref{e:exact-B} with initial data  $(v_{\ell_q}(t_i), \theta_{\ell_q}(t_i))$ on $[t_i-\tau_q, t_i+\tau_q]$. Hence, for $1\le i\le i_{\max-1}$, we define $(\vex_i, \bm{\theta}_i,\pex_i)=(\vex,\bm\theta,\pex)$ that solves the system. Moreover, one obtains the following stability estimates:
\begin{prop}[Stability\cite{MNY}]
\label{p:stability}For $i\ge 0$, $|t-t_i|\le \tau_q$, and $0\le N\le L+1$, we have
\begin{align}
    \|(\vex_i -v_{\ell_q}, \bm\theta_i -\theta_{\ell_q}) \|_{{N+\alpha}} & \lesssim \tau_q \delta_{q+1} \ell_q^{-N-1+\frac{5}{4}\alpha}\label{e:stability-v},\\
    \|\nabla \pex_i -\nabla p_{\ell_q} \|_{{N+\alpha}} &\lesssim\delta_{q+1} \ell_q^{-N-1+\frac{5}{4}\alpha},\label{e:stability-p}\\
    \|(\matd{v_{\ell_q}}(\vex_i - v_{\ell_q}),  \matd{v_{\ell_q}}(\bm\theta_i - \theta_{\ell_q}))\|_{{N+\alpha}} & \lesssim  \delta_{q+1} \ell_q^{-N-1+\frac{5}{4}\alpha},\label{e:stability-matd}\\
  \|(\mathcal{R}(\vex_i -v_{\ell_q}), \mathcal{R}_{\vex}(\bm\theta_i -\theta_{\ell_q})) \|_{{\alpha}} &\lesssim \tau_q \delta_{q+1} \ell_q^{\frac{5}{4}\alpha}\label{e:stability-v-1},
\end{align}
where the inverse of the divergence operators $\mathcal{R}$ and $\mathcal{R}_{\vex}$ which are given in Definition \ref{def.R}.
\end{prop}
Based on $\{(\vex_i, {\bm{\theta}_i}, \pex_i)\}_{0\le i\le i_{\max}}$, we define the glued solution $(\vv_q, \tb_q, \ppp_q)$ by
\begin{align}
    \vv_q(x,t) \coloneq \sum_{i=0}^{i_{\textup{max}}} \eta_i(t) \vex_i(x,t),\,\,
     \tb_q(x,t) \coloneq \sum_{i=0}^{i_{\textup{max}}} \eta_i(t) \bm\theta_i(x,t),\,\,
    \ppp_q(x,t) \coloneq \sum_{i=0}^{i_{\textup{max}}} \eta_i(t) \pex_i(x,t),
    \label{e:pp_q}
\end{align}
where $\{ \eta_i\}_{i=0}^{i_{\textup{max}}}$ is a partition of unity on $[0,1]$ such that for all $N\ge 0$,
\begin{align*}
       &\supp \eta_0=[-1, t_0+\tfrac{2\tau_q}{3}], \qquad \,\,\,\eta_0 |_{J_0} =1, \qquad\|  \del_t^N \eta_0\|_{0} \lesssim \tau_q^{-N},\\
           & \supp\eta_i=I_{i-1} \cup J_i \cup I_{i},  \qquad\,\,\,\,\eta_i |_{J_i} =1, \qquad\, \|  \del_t^N \eta_i\|_{0} \lesssim \tau_q^{-N},\ \   1\le i\le i_{\max}-1,\\
       &\supp \eta_{i_{\max}}= [t_{i_{\max}-1}+\tfrac{\tau_q}{3}, 2],  \,\,\,\,\eta_i |_{J_{i_{\max}}} =1, \quad\|  \del_t^N \eta_{i_{\max}}\|_{0} \lesssim \tau_q^{-N}.
\end{align*}
Here the time intervals $\{I_i\}$ and $\{J_i\}$ are given by
\begin{align*}
 &I_i \coloneq [ t_i + \tfrac{\tau_q}3,\ t_i+\tfrac{2\tau _q}3], \,\, 0\le i\le i_{\max}-1,\\
  &J_0:=[0, t_0+\tfrac{\tau_q}3),\,\,J_{i_{\max}}:=(t_{i_{\max}-1}+\tfrac{2\tau _q}3, 3\widetilde{T}],\\
 &J_i \coloneq (t_i - \tfrac{\tau_q}3,\ t_i+\tfrac{\tau _q}3),\,\, 1\le i\le i_{\max}-1.
\end{align*}
It follows from the gluing procedure that  $(\vv_q, \tb_q, \pp_q)$ solves
\begin{equation}
\left\{ \begin{alignedat}{-1}
&\del_t \vv_q+\Div (\vv_q\otimes \vv_q)  +\nabla \pp_q   =  \Div \RRR_q,
 \\
 &\del_t \tb_q+\vv_q\cdot\nabla \tb_q     = \Div \MMM_q,
 \\
  &\nabla \cdot \vv_q = 0,
\end{alignedat}\right.  \label{e:glu-B}
\end{equation}
on $[0,1]$, where
\begin{align}
    \RRR_q &\coloneq
        \sum_{i=0}^{i_{\max}}\del_t \eta_i \mathcal R(\vex_i-\vex_{i+1} ) - \sum_{i=0}^{i_{\textup{max}}} \eta_i(1-\eta_i)(\vex_i-\vex_{i+1} )\ootimes (\vex_i-\vex_{i+1} )\notag,
%        \label{e:def-RRR_q}
\\
 \MMM_q &\coloneq
        \sum_{i=0}^{i_{\textup{max}}} \del_t \eta_i \mathcal R_{\vex}(\bm\theta_i-\bm\theta_{i+1} ) - \sum_{i=0}^{i_{\textup{max}}} \eta_i(1-\eta_i)(\vex_i-\vex_{i+1} ) (\bm\theta_i-\bm\theta_{i+1} )\notag,
%        \label{e:def-RRR_q}
\\
    \pp_q &  \coloneq \ppp_q -\int_{\TTT^3}\ppp_q \dd x- \sum_{i=0}^{i_{\textup{max}}} \eta_i(1-\eta_i)\big( |\vex_i - \vex_{i+1}|^2  - \int_{\mathbb T^3} |\vex_i - \vex_{i+1}|^2 \dd x\big).\label{ppp}
%    \label{e:def:pp_q}
\end{align}
Following the proofs of  \cite[Proposition 4.4]{BDSV} and
 \cite[Proposition 4.6]{MNY}, we have
\begin{prop}[Estimates for $\RRR_q$ and $\MMM_q$]\label{p:estimate-RRRq}For all $0\le N\le L+1$, we have
    \begin{align}
        \|(\RRR_q, \MMM_q)\|_{N+\alpha}
        &\lesssim \delta_{q+1} \ell_q^{-N+ \alpha}, \label{e:RRR_q-N+alpha-bd}
        \\
        \| (\matd{\vv_q} \RRR_q, \matd{\vv_q}\MMM_q)\|_{N+\alpha }
        &\lesssim  \delta_{q+1} \delta_{q}^{1/2}\lambda_q \ell_q^{-N-\frac{3}{4}\alpha}.\label{e:matd-RRR_q}
    \end{align}
\end{prop}
\begin{prop}[Estimates for $\vv_q$ and $\tb_q$]\label{est-vvq}For  all $0\le N\le L+1$,
\begin{align}
&\|(\vv_{q}-v_{\ell_q}, \tb_{q}-\theta_{\ell_q})\|_{\alpha}  \le \delta_{q+1}^{1/2} \ell_q^\alpha,\label{e:stability-vv_q} \\
&\|(\vv_{q}, \tb_{q})\|_{0}  \le M\sum_{i=1}^q\delta^{1/2}_i+\delta_{q+1}^{1/2} \ell_q^\alpha, \label{e:stability-vv_q-N} \\
&\|(\vv_{q}, \tb_q)\|_{1+N}  \lesssim M\delta_{q}^{1/2} \lambda_{q} \ell_q^{-N}.\label{e:vv_q-bound}
\end{align}
Furthermore, we also have
\begin{align}\label{vvqtbq}
(\vv_{q}, \tb_q)\equiv (v^{(1)}, b^{(1)}) \,\,{\rm on}\,\,[0, 2\widetilde{T}+\tfrac{\tau_q}{3}],\quad (\vv_{q}, \tb_q)\equiv (v^{(2)}, b^{(2)}) \,\,\text{\rm on}\,\,[3\widetilde{T}-\tfrac{\tau_q}{3}, 1].
\end{align}
\end{prop}
Let $\Phi_i$ be the so-called inverse flow map, defined as the solution to the glued vector transport equations
    \begin{align}
         (\partial_t+\vv_q \cdot \nabla) \Phi_i = 0, \quad \Phi_i\big|_{t=t_i} = x. \label{e:phi_i-defn}
    \end{align}
   From Proposition \ref{est-vvq}, one establishes the following estimates for $\Phi_i$.
    \begin{prop}[Estimates for $\Phi_i$\cite{KMY}]\label{p:estimates-for-inverse-flow-map1} For $a\gg 1$, $0\le N\le L+1$ and every $t\in {\rm supp} \,\,\eta_i$,
\begin{align}
 \|\nabla\Phi_i-{\rm Id}_{3\times3}\|_0 &\le\frac1{10},\label{e:nabla-phi-i-minus-I3x3}
 \\
 \| (\nabla \Phi_i)^{-1}\|_N+ \|  \nabla \Phi_i\|_N &\le \ell_q^{-N}, \label{e:nabla-phi-i-CN1}
\\
\|\matd{\vv_q} \nabla \Phi_i\|_N &\lesssim \delta_q^{1/2} \lambda_q \ell_q^{-N}. \label{e:nabla-phi-i-matd1}
\end{align}
\end{prop}
\noindent{\textit{Step 3: Construction of the perturbation}.} We introduce highly oscillatory temporal functions to exploit additional temporal dimensions. This technique achieves the non-intersection property of Nash perturbations, while it simultaneously gives rise to new errors.  Inspired by \cite{GR}, we construct Newtonian perturbations to counteract these errors. \\
\noindent {\textbf{Newton perturbations.}}\,\,
For $0\leq i\leq i_{\max}-1,~n\in\NN$, we define the time cut-offs $\{\chi_{i,n}\}$ by
\begin{align*}
    &\supp\chi_{i,n}\subset \big[t_i+\tfrac{1}{3}\tau_q-\tfrac{2n+1}{12L}\tau_q,\quad
t_i+\tfrac{2}{3}\tau_q+\tfrac{2n+1}{12L}\tau_q\big],\\
&\chi_{i,n}(t)\equiv 1,\quad \forall t\in
    \big[t_i+\tfrac{1}{3}\tau_q-\tfrac{2n+1}{12L}\tau_q+\tfrac{1}{24L}\tau_q,\quad
t_i+\tfrac{2}{3}\tau_q+\tfrac{2n+1}{12L}\tau_q-\tfrac{1}{24L}\tau_q\big],
\end{align*}
and
\begin{align*}
    |\partial^N_t\chi_{i,n}|\lesssim\tau^{-N}_q,\quad \forall\,\, N\in \NN.
\end{align*}
From the definition of $\{\chi_{i,n}\}$, we easily deduce that
\begin{align*}
    \chi_{i,n}\chi_{i',n'}=0,\quad \forall\,\, 0\le i\neq i'\le i_{\max}-1,\,\, n, n'\in \NN.
\end{align*}
For $0\leq i\leq i_{\max}-1, n\in\NN$, we give another time cut-offs $\{\bar{\chi}_{i,n}\}$ by
\begin{align*}
    &\supp\bar{\chi}_{i,n}\subset \big[t_i+\tfrac{1}{3}\tau_q-\tfrac{2n}{12L}\tau_q,\quad
t_i+\tfrac{2}{3}\tau_q+\tfrac{2n}{12L}\tau_q\big],\\
&\bar{\chi}_{i,n}(t)\equiv 1,\quad \forall t\in
    \big[t_i+\tfrac{1}{3}\tau_q-\tfrac{2n}{12L}\tau_q+\tfrac{1}{24L}\tau_q,\quad
t_i+\tfrac{2}{3}\tau_q+\tfrac{2n}{12L}\tau_q-\tfrac{1}{24L}\tau_q\big],
\end{align*}
and
\begin{align*}
    |\partial^N_t\bar{\chi}_{i,n}|\lesssim\tau^{-N}_q,\quad \forall\,\, N\in \NN.
\end{align*}
We infer from the definitions of  $\{\chi_{i,n}\}$ and  $\{\bar{\chi}_{i,n}\}$ that, for any $n\in\NN$ and $0\le i\le i_{\max-1}$,
\begin{align*}
 {{\chi}_{i,n}\bar{\chi}_{i,n }=\bar{\chi}_{i,n},\,\,{\chi}_{i,n}\bar{\chi}_{i,n+1}= {\chi}_{i,n},\,\,\partial_t{\chi}_{i,n}\bar{\chi}_{i,n}=0,\,\,
 {\chi}_{i,n}\partial_t\bar{\chi}_{i,n+1}=0.}
 \end{align*}
We denote ${\Lambda_{\theta}=\{e_1, e_2 \}}$ and choose two sets $\Lambda_1$ and $\Lambda_2$ in Lemma \ref{first S} such that $\Lambda_1\cap\Lambda_2\cap \Lambda_{\theta}=\emptyset.$  For simplicity, we denote
\begin{align*}
\Lambda_1\cup\Lambda_2=:\Lambda_v, \quad     \Lambda_v\cup\Lambda_{\theta}=:\Lambda.
\end{align*}
Next, we give the  temporal functions  to achieve non-interaction for different Nash perturbations.
\begin{prop}\label{time-g}
 Let $n\in\{0,1,2,...,L-1\}$. For each $k\in\Lambda$, there exist $L$ many smooth 1-periodic time functions $g_{k, n}:\mathbb{R}\rightarrow\mathbb{R}$ such that
 \begin{align*}
\int_0^1g^2_{k, n}(t)\dd t=1,
 \end{align*}
and $\forall k,k'\in \Lambda$ and $ n,n'\in\{0,1,2,...,L-1\}$,
\begin{align*}
\supp g_{k, n}\cap\supp g_{k' ,n'}=\emptyset, \quad \forall (k,n)\neq (k',n').
 \end{align*}
Moreover, $\forall k,k'\in\Lambda$, $0\leq i,i'\leq i_{\max}-1$, and $n,n'\in \{0,1,2,...,L-1\}$, we have $$\chi_{i,n} g_{k ,n}\cdot\chi_{i',n'}  g_{k' ,n'}=0,\quad \forall (k,i,n)\neq (k',i',n').$$
\end{prop}
Using the fact that ${\Lambda_{\theta}=\{e_1, e_2 \}}$, one immediately gives  the vector function $\nabla\Phi_i\MMM_q\chi_{i,n}(t)$ in components form as follows.
\begin{equation}\label{Tqik}
-\nabla\Phi_i\MMM_q\chi_{i,n}(t)={\sum_{j=1}^2}\big(-\nabla\Phi_i\MMM_q\big)^j \chi_{i,n}(t)e_j =:\sum_{k\in\Lambda_{\theta}}  a^{\theta}_{k}(\nabla\Phi_i\MMM_q)\chi_{i,n}(t) k.
\end{equation}
%For $0\le n\le L-1$, we define $\{({a}^{\theta}_{q,k,i,n}, {a}^{v}_{q,k,i,n},{A}^{\theta}_{q,k,i,n},A^{v_1}_{q,k,i,n},A^{v_2}_{q,k,i,n})\}$ by $\{(\MMM_{q,n}, \RRR_{q,i,n})\}$.
Firstly, for  $n=0$, we define
\begin{align*}
 \MMM_{q,0}:=\MMM_{q},\quad \RRR_{q,i,0}:=\RRR_{q}+\sum_{k\in\Lambda_{\theta}}\delta^{-1 }_{q+1,n}   \big(a^{\theta}_{k}  (\nabla\Phi_i\MMM_{q,0})\big)^2\nabla\Phi^{-1}_i(k\otimes k)\nabla\Phi^{-\TT}_i,
\end{align*}
where smooth function $a_k$ stems from Lemma \ref{first S}. For given $\MMM_{q,n}$ and $\RRR_{q,i,n}$,  we define
\begin{align}
 &{a}^{\theta}_{q, k,i,n}:=\delta^{-1/2}_{q+1,n}\chi_{i,n}   a^{\theta}_{k}(\nabla\Phi_i\MMM_{q,n}),\quad {A}^{\theta}_{q,k,i,n}: =   \chi^2_{i,n}
\nabla {\Phi}^{-1}_i a^{\theta}_{k}(\nabla\Phi_i\MMM_q) k , \quad k\in \Lambda_{\theta}{\color{red},}\label{def-a-theta}\\
&a^{v}_{q,k,i,n}:= \delta_{q+1,n}^{1/2}\chi_{i,n} a_{k}\Big(\nabla\Phi_i\Big({\rm Id}-\frac{\RRR_{q,i,n}}{\delta_{q+1,n}}\Big)\nabla\Phi^{\TT}_i\Big), \quad k\in \Lambda_{v},\label{def-av}\\
&A^{v_1}_{q,k,i,n} := (a^{v}_{q,k,i,n})^2
\nabla\Phi^{-1}_i(k\otimes k)\nabla\Phi^{-\TT}_i,\quad k\in \Lambda_{v},\label{def-A-v1}\\
&A^{v_2}_{q,k,i,n} :=  (a^{\theta}_{q,k,i,n})^2
\nabla\Phi^{-1}_i(k\otimes k)\nabla\Phi^{-\TT}_i,\quad k\in \Lambda_{\theta}.\label{def-A-v2}
\end{align}
Then we define $\MMM_{q,n+1}$ and $\RRR_{q,i,n+1}$ to be
\begin{align}\label{next Tq}
\MMM_{q,n+1} :=&\sum_{j}\partial_t\bar{\chi}_{j,n+1}\mathcal{R}_{\vex}\dt_{q+1,j,n+1},\\
\RRR_{q,i,n+1}:= &\sum_{j}\partial_t\bar{\chi}_{j,n+1}\mathcal{R}\wt_{q+1,j,n+1}\notag\\
&+\sum_{  k\in\Lambda_{\theta}}\delta^{-1 }_{q+1,n+1}
\big(a^{\theta}_{k}  (\nabla\Phi_i\MMM_{q,n+1})\big)^2\nabla\Phi^{-1}_i(k\otimes k)\nabla\Phi^{-\TT}_i ,\label{next Rq}
\end{align}
where $(\wt_{q+1,i,n+1}, \dt_{q+1,i,n+1})$ solves the following system
\begin{equation}
\left\{ \begin{alignedat}{-1}
&\del_t \wt_{q+1,i,n+1}+\vv_{q}\cdot\nabla \wt_{q+1,i,n+1}+ \wt_{q+1,i,n+1}\cdot\nabla \vv_{q}  +\nabla p^{(\text{t})}_{q+1,i,n+1}  \\
&\qquad\qquad\qquad\quad=
\sum_{ k\in\Lambda_{v}}f_{k,n+1}(\mu_{q+1}t)\Div  {A}^{v_1}_{q,k,i,n}+\sum_{ k\in\Lambda_{\theta}}f_{k,n+1}(\mu_{q+1}t)\Div  {A}^{v_2}_{q,k,i,n},
 \\
& \del_t \dt_{q+1,i,n+1}+\vv_{q}\cdot\nabla \dt_{q+1,i,n+1}+ \wt_{q+1,i,n+1} \cdot\nabla   \tb_q  =
\sum_{ k\in\Lambda_{\theta}}f_{k,n+1}(\mu_{q+1}t) \Div  {A}^{\theta}_{q,k,i,n},\\
  &\nabla \cdot  \wt_{q+1,i,n+1}  = 0\\
 & (\wt_{q+1,i,n+1}, \dt_{q+1,i,n+1})|_{t=t_i} = 0.
\end{alignedat}\right.  \label{time}
\end{equation}
Here  $f_{k,n+1}:=-\mathbb{P}_{>0}g_{k,n+1}^2=1-g_{k,n+1}^2$ with  $\|\int_{t_i}^t f_{k, n+1}(\mu_{q+1} s)\dd s\|_{L^\infty}\lesssim \mu^{-1}_{q+1}$ and
\begin{align}\label{av-pi}
    \int_{\TTT^2}p^{\textup{(t)}}_{q+1,i,n+1}\dd x=0.
\end{align}
Now, we give the Newton perturbations $(\wt_{q+1}, \dt_{q+1})$ by
\begin{align}\label{def-wt-dt}
  \wtq  :=\sum_{n=0}^{L-1} \sum_{i}\bar{\chi}_{i,n+1}(t)\wt_{q+1,i,n+1},
\quad \dt_{q+1}:=\sum_{n=0}^{L-1} \sum_{i}\bar{\chi}_{i,n+1}(t)\dt_{q+1,i,n+1}.
\end{align}
The introduction of $\{\chi_{i,n}\}$ and $\{\bar{\chi}_{i,n}\}$ ensures that $\wtq(t_i)=\dtq(t_i)=0$.

\noindent{\textbf{Nash perturbations.}}\,\, The Nash perturbations are constructed by shear flows $\phi_{(\lambda_{q+1}k)}(x)k$ given by: Let $\psi:\mathbb{T}\rightarrow\mathbb{R}$ be a smooth mean-free cutoff function supported on the interval $[0,\lambda^{-1}_1]$ satisfying that $\phi=\psi'$ and $\int_{\mathbb{T}}\phi^2(x)\dd x=1$. We define
\begin{align}\label{phi-xi}
\phi_{(\lambda_{q+1}k)}(x):=\phi (\lambda_{q+1} \bar{k}\cdot x),\quad\bar{k}\perp k,~~ k\in \Lambda.
\end{align}
Then we give the the principal correctors $\wpq $ such that
\begin{align}
\wpq= & \sum_{n=0}^{L-1}\Big(\sum_{i;k\in\Lambda_v}g_{k,n+1}(\mu_{q+1}t)\bar{a}^{v}_{q,k,i,n}
\nabla\bar{\Phi}^{-1}_i\phi_{(\lambda_{q+1}k)}(\bar{\Phi}_i(x,t))k\notag\\
&+
\sum_{i;k\in\Lambda_{\theta }}g_{k,n+1}(\mu_{q+1}t)
\bar{a}^{\theta }_{q,k,i,n}
\nabla\bar{\Phi}^{-1}_i\phi_{(\lambda_{q+1}k)}(\bar{\Phi}_i(x,t))k\Big)
=:\sum_{n=0}^{L-1}w^{(\textup{p})}_{q+1,n},\label{defi-wpq}\\
 \dpq =&\sum_{n=0}^{L-1}\sum_{i;k\in\Lambda_{\theta }}\delta^{1/2}_{q+1,n}\chi_{i,n}(t)g_{k,n+1}(\mu_{q+1}t)
\phi_{(\lambda_{q+1}k)}(\bar{\Phi}_i(x,t))=: \sum_{n=0}^{L-1}d^{(\textup{p})}_{q+1,n} ,\label{defi-dpq}
\end{align}
where
\begin{align*}
&\bar{a}^{v}_{q,k,i,n}=\chi_{i,n} \delta^{1/2}_{q+1,n}a_{k}
\Big(\nabla\bar{\Phi}_i\big({\rm Id}-\frac{\RRR_{q,i,n}}{\delta_{q+1,n}}\big)
\nabla\bar{\Phi }^{\TT}_i\Big),\,\,\,\,\bar{a}^{\theta}_{q,k,i,n}=\chi_{i,n} \delta^{-1/2}_{q+1,n}a^{\theta}_{k}
\big(\nabla\bar{\Phi}_i\MMM_{q,n} \big)
\end{align*}
and $\bar{\Phi}_i$ is governed by $\vv_q+ \wtq$:
\begin{equation}
\left\{ \begin{alignedat}{-1}
&\del_t \bar{\Phi}_i+(\vv_q+ \wtq ) \cdot\nabla \bar{\Phi}_i   =0,
 \\
  &\bar{\Phi}_i(t_i,x)= x.
\end{alignedat}\right.  \label{barphi}
\end{equation}
Similar to \eqref{def-a-theta}, \eqref{def-A-v1} and \eqref{def-A-v2}, we define
\begin{align*}
&\bar{A}^{\theta}_{q,k,i,n}:=   \chi^2_{i,n}  a^{\theta}_{i}  (\nabla\bar{\Phi}_i\MMM_{q,n})
\nabla \bar{\Phi}^{-1}_ik ,\,\,   \bar{A}^{v_2}_{k,i,n}:=    (\bar{a}^{\theta}_{q,k,i,n})^2
\nabla\bar{\Phi}^{-1}_i(k \otimes k)\nabla\bar{\Phi}^{-\TT}_i,\quad \forall  k\in\Lambda_{\theta},\\
&\bar{A}^{v_1}_{q,k,i,n}  =  (\bar{a}^{v}_{q,k,i,n})^2
\nabla\bar{\Phi}^{-1}_i(k\otimes k)\nabla\bar{\Phi}^{-\TT}_i,\quad \forall k\in \Lambda_{v}.
 \end{align*}
Denote $\nabla^{\perp}=(-\partial_{x_2}, \partial_{x_1})$, one immediately has
\begin{align*}
\phi_{(\lambda_{q+1}k)}(x){k}=\lambda^{-1}_{q+1}\nabla^{{\perp}}\Big(\int_0^{\lambda_{q+1}\bar{k}\cdot x}\phi(z)\dd z\Big)=\lambda^{-1}_{q+1}\nabla^{{\perp}}\psi(\lambda_{q+1}\bar{k}\cdot x).
\end{align*}
Hence, we infer from $\det\nabla\Phi=1$ that
\begin{align*}
\nabla \Phi^{-1}_i\phi_{(\lambda_{q+1}k)}(\Phi_i(x,t)){k}=\lambda^{-1}_{q+1}\nabla^{{\perp}}\big(\psi(\lambda_{q+1}\bar{k}\cdot \Phi_i(x,t))\big).
\end{align*}
This equality together with \eqref{defi-wpq} yields that
\begin{align*}
\wpq
=&\lambda^{-1}_{q+1}\sum_{n=0}^{L-1}\sum_{i;k\in\Lambda_v}g_{k,n+1}(\mu_{q+1}t)\bar{a}^v_{q,k,i,n}
\nabla^{\perp}\big(\psi(\lambda_{q+1}\bar{k}\cdot\bar{\Phi}_i )\big) \\
   &+  \lambda^{-1}_{q+1}\sum_{n=0}^{L-1}
\sum_{i;k\in\Lambda_{\theta }}g_{k,n+1}(\mu_{q+1}t)
 { \bar{a}^{\theta }_{q,k,i,n} }
\nabla^{\perp}\big(\psi(\lambda_{q+1}\bar{k}\cdot\bar{\Phi}_i )\big)  .
\end{align*}
Since $\Div\wpq\neq 0$, we need to construct divergence-free corrector $\wcq$, defined as
\begin{align}
    \wcq =&
  \sum_{n=0}^{L-1} \lambda^{-1}_{q+1}\Big(\sum_{i;k\in\Lambda_v}g_{k,n+1}(\mu_{q+1}t)\nabla^{\perp}(\bar{a}^v_{q,k,i,n})
    \psi(\lambda_{q+1}\bar{k}\cdot\bar{\Phi}_i )  \notag\\
   &+  
\sum_{i;k\in\Lambda_{\theta }}g_{k,n+1}(\mu_{q+1}t)
 {\nabla^{\perp} (\bar{a}^{\theta }_{q,k,i,n} )}
 \psi(\lambda_{q+1}\bar{k}\cdot\bar{\Phi}_i )  \Big)\notag\\
 =:&\wc_{q+1,n}.\label{e:defn-wcq}
\end{align}
Collecting the above two equalities together leads to \begin{align}\label{wpq-wcq}
\wpq +\wcq\notag
=& \lambda^{-1}_{q+1}\nabla^{\perp} \sum_{n=0}^{L-1}\sum_{i;k\in\Lambda_v}g_{k,n+1}(\mu_{q+1}t) \bar{a}^v_{q,k,i,n}
    \psi(\lambda_{q+1}\bar{k}\cdot\bar{\Phi}_i )\notag\\
   & + \lambda^{-1}_{q+1}\nabla^{\perp}\sum_{n=0}^{L-1}
\sum_{i;k\in\Lambda_{\theta }}g_{k,n+1}(\mu_{q+1}t)
 { \bar{a}^{\theta }_{q,k,i,n} }
  \psi(\lambda_{q+1}\bar{k}\cdot\bar{\Phi}_i ).\end{align}
It follows that $\Div(\wpq+\wcq)=0$.

\noindent{\textit{Step 4: Estimates for the perturbations.}}\,\,We define
\[w_{q+1}=\wpq+\wcq+\wtq, \quad d_{q+1}=\dpq+\dtq.\]
Firstly, we establish the estimates for the Newtonian perturbation $(\wtq, \dtq)$. To this end, we start with estimating  $\wt_{q+1,i,n+1}$ and $\dt_{q+1,i,n+1}$ via the induction on $n$.
 \begin{prop}[Estimates for $(\wt_{q+1,i,n+1},\dt_{q+1,i,n+1})$]\label{est-RM}
 For $0\le n\le L-1$, $0\le N\le L+1$ and $t\in[t_{i-1}, t_{i+1}]$, we have
\begin{align}
\|\wt_{q+1,i,n+1}\|_{N+\alpha}+\|\dt_{q+1,i,n+1}\|_{N+\alpha}&\lesssim_N\mu^{-1}_{q+1}\delta_{q+1,n}\ell^{-N -1-3\alpha}_q,\label{estimate-wt}\\
\|\RRR_{q,i,n}\|_{N+\alpha}+\|\MMM_{q,n}\|_{N+\alpha}&\lesssim_N\delta_{q+1,n}\ell^{-N+\frac{3}{2}\alpha}_q,\label{estimate-Rv}\\
\|\matd {\vv_q}\RRR_{q,i,n}\|_{N+\alpha}
+\|\matd{\vv_q}\MMM_{q,n}\|_{N+\alpha}
&\lesssim_N  \tau^{-1}_q\delta_{q+1,n}\ell^{-N+\frac{3}{2}\alpha}_q. \label{estimate-DRv}
\end{align}
\end{prop}
\begin{proof}Firstly, for $n=0$, we deduce from \eqref{e:RRR_q-N+alpha-bd} and \eqref{e:nabla-phi-i-CN1} that
\begin{align*}
\|\RRR_{q,i,0}\|_{N+\alpha}+\|\MMM_{q,0}\|_{N+\alpha}&\lesssim\delta_{q+1,0}\ell^{-N+\frac{3}{2}\alpha}_q.
\end{align*}
With the aid of \eqref{e:matd-RRR_q} and \eqref{e:nabla-phi-i-matd1}, one gets
\begin{align*}
 \|\matd {\vv_q}\RRR_{q,i,0}\|_{N+\alpha}
+\|\matd{\vv_q}\MMM_{q,0}\|_{N+\alpha}&\lesssim\tau^{-1}_q\delta_{q+1,0}\ell^{-N+\frac{3}{2}\alpha}_q.
\end{align*}
Assuming that \eqref{estimate-Rv} and \eqref{estimate-DRv} hold for $n=m$, we want to prove them for $n=m+1$. Then we obtain from \eqref{def-a-theta}--\eqref{def-av} that
\begin{align}\label{E-a-theta}
\|{a}^{\theta}_{q, k,i,m}\|_{N+\alpha}\lesssim\delta^{1/2}_{q+1,m}\ell^{-N }_q, \quad \|\matd {\vv_q }{a}^{\theta}_{q, k,i,m}\|_{N+\alpha}\lesssim\tau^{-1}_q\delta^{1/2}_{q+1,m}\ell^{-N }_q,
\end{align}
\begin{align}\label{E-a-v}
\|a^{v}_{q,k,i,m}\|_{N+\alpha}\lesssim_N { \delta^{1/2}_{q+1,m}\ell^{-N-3\alpha}_q}, \quad \|\matd{\vv_q} a^{v}_{q,k,i,m}\|_{N+\alpha}
\lesssim_N  \tau^{-1}_q\delta^{1/2}_{q+1,m}\ell^{- N-3\alpha}_q.
\end{align}
Therefore, one has
\begin{align}
&\|{A}^{\theta}_{q, k,i,m}\|_{N+\alpha}\lesssim\delta^{1/2}_{q+1,m}\ell^{-N }_q,& \|\matd{\vv_q}{A}^{\theta}_{q, k,i,m}\|_{N+\alpha}\lesssim \tau^{-1}_q\delta^{1/2}_{q+1,m}\ell^{-N }_q\label{E-Atheta},\\
&\|{A}^{v_1}_{q, k,i,m}\|_{N+\alpha}\lesssim_N  \delta_{q+1,m}\ell^{-N- 3\alpha}_q,
&\|\matd{\vv_q}{A}^{v_1}_{q, k,i,m}\|_{N+\alpha}\lesssim_N  \tau^{-1}_q\delta_{q+1,m}\ell^{-N-3\alpha}_q,\label{E-Av1}\\
&\|{A}^{v_2}_{q, k,i,m}\|_{N+\alpha}\lesssim  \delta_{q+1,m}\ell^{-N-3\alpha}_q, &\|\matd{\vv_q}{A}^{v_2}_{q, k,i,m}\|_{N+\alpha}\lesssim \tau^{-1}_q\delta_{q+1,m}\ell^{-N-3\alpha}_q.\label{E-Av2}
\end{align}
Moreover, with the aid of the commutator estimate, we have
\begin{align*}
 \|[\mathbb{P}_{H}\Div,\vv_q\nabla] {A}^{v_1}_{q,k,i,m}\|_{ N+\alpha}&\lesssim \|\vv_q\|_{1+N+\alpha}\|{A}^{v_1}_{q,k,i,m}\|_{1 +\alpha}+\|\vv_q\|_{1 +\alpha}\|{A}^{v_1}_{q,k,i,m}\|_{1+N+\alpha}\\
 &\lesssim \delta_{q+1,m}\lambda_q\delta^{1/2}_{q}\ell^{ -N-1-4\alpha}_q,
\end{align*}
and
\begin{align*}
 \|[\mathbb{P}_{H}\Div,\vv_q\nabla] {A}^{v_2}_{q,k,i,m}\|_{ N+\alpha}&\lesssim \|\vv_q\|_{1+N+\alpha}\|{A}^{v_2}_{q,k,i,m}\|_{1 +\alpha}+\|\vv_q\|_{1 +\alpha}\|{A}^{v_2}_{q,k,i,m}\|_{1+N+\alpha}\\
 &\lesssim \delta_{q+1,m}\lambda_q\delta^{1/2}_{q}\ell^{-N-1-4\alpha}_q.
\end{align*}
The above two inequalities together with \eqref{E-Av1} and \eqref{E-Av2} yields that
\begin{align}
&\| \matd{\vv_q}\mathbb{P}_{H}\Div{A}^{v_1}_{q, k,i,m} \|_{N+\alpha}\lesssim \tau^{-1}_q\delta_{q+1,m}\ell^{ -N-1-3{\alpha}}_q,\\
&\|\matd{\vv_q}\mathbb{P}_{H}\Div{A}^{v_2}_{q, k,i,m}\|_{N+\alpha}\lesssim \tau^{-1}_q\delta_{q+1,m}\ell^{ -N-1-3{\alpha}}_q.\label{E-pt-Av2}
\end{align}
Owing to $\|\vv_q\|_{1}\lesssim \delta^{1/2}_q\lambda_q$, one easily verifies that the system \eqref{time} possesses a unique solution on $[t_{i-1}, t_{i+1}]$. Using Lagrangian coordinate $y(t,x)$ with $\partial_ty(t,x) =\vv_q(t,y(t,x))$ and $y(t_i,x)=x$,  we have
\begin{align}
  \wt_{q+1,i,m+1}(t,y(t,x))=& -\int_{t_i}^{t}\big(\mathbb{P}_H(\wt_{q+1,i,m+1}\cdot\nabla \vv_q)\big)(s,y(s,x))\dd s\notag\\
  &+\int_{t_i}^{t}\Big(\frac{\nabla\Div}{\Delta}(\vv_q\cdot\nabla \wt_{q+1,i,m+1})\Big)(s,y(s,x))\dd s\notag\\
&+\sum_{i;k\in\Lambda_{v}}\int_{t_i}^{t}f_{k,m+1}(\mu_{q+1}s) (\mathbb{P}_{H}\Div  {A}^{v_1}_{q,k,i,m})(s,y(s,x))\dd s\notag\\
&+\sum_{i;k\in\Lambda_{\theta}}\int_{t_i}^{t}f_{k,m+1}(\mu_{q+1}s) (\mathbb{P}_{H}\Div  {A}^{v_2}_{q,k,i,m})(s,y(s,x))\dd s \label{int-w}
\end{align}
and
\begin{align*}
\dt_{q+1,i,m+1}(t,y(t,x))=&-\int_{t_i}^{t} (\wtq\cdot \nabla   \tb_q) (s,y(s,x))\dd s\\
&+
\sum_{i;k\in\Lambda_{\theta}}\int_{t_i}^{t}f_{k,m+1}(\mu_{q+1}s) (\Div  {A}^{\theta}_{q,k,i,m})(s,y(s,x))\dd s.
\end{align*}
Note that $\|\int_{t_i}^tf_{k,m+1}(\mu_{q+1}s) \dd s\|_{L^\infty}\le \mu^{-1}_{q+1}$, via integration by parts, we have, for $t\in [t_i, t_{i+1}]$,
\begin{align*}
\| \wt_{q+1,i,m+1}(t)\|_{\alpha}\lesssim&\int^t_{t_i}\|\wt_{q+1,i,m+1}(s)\|_{\alpha}\|\vv_q(s)\|_{1+\alpha}\dd s\\
&+\mu^{-1}_{q+1} \Big\|(\mathbb{P}_{H}\Div  {A}^{v_1}_{q,k,i,m},  \mathbb{P}_{H}\Div{A}^{v_2}_{q,k,i,m})\circ y\Big|_{t_i}^{t}\Big\|_{\alpha} \\
&+\mu^{-1}_{q+1}\int_{t_i}^{t}\|\big(  \matd{\vv_q}\mathbb{P}_{H}\Div  {A}^{v_1}_{q,k,i,m}(s), \partial_t\mathbb{P}_{H}\Div{A}^{v_2}_{q,k,i,m}\big)(s,y(s,x)) \|_{\alpha}\dd s\\
\lesssim&  \mu^{-1}_{q+1}\delta_{q+1,m}\ell^{ -1-3{\alpha}}_q,
\end{align*}
where we have used  \eqref{E-Av1}--\eqref{E-pt-Av2} and Gr\"{o}nwall's inequality. Thus, one deduces  that, for $t\in [t_{i-1}, t_{i+1}]$,
\begin{align}
    \| \wt_{q+1,i,m+1}\|_{\alpha}\lesssim \mu^{-1}_{q+1}\delta_{q+1,m}\ell^{ -1-3{\alpha}}_q.\label{e-w-alpha}
\end{align}
In the same way as leading to \eqref{e-w-alpha}, we have
$ \|\dt_{q+1,i,m+1}\|_{\alpha}\lesssim \mu^{-1}_{q+1}\delta_{q+1,m}\ell^{ -1}_q$.

Since $ y(t,x),~y^{-1}(t,x)$ are measure-preserving diffeomorphism mappings for each $t\in [t_{i-1},t_{i+1}]$, we have, for any function $F\in B^{s}_{p,q}$ with $s\in(-1,1), 1\le p,q\le \infty$,
\begin{align*}
    \|F\circ y\|_{B^s_{p,q}}\sim \|F\|_{B^s_{p,q}}.
\end{align*}
Therefore, we obtain from \eqref{int-w} that, for $t\in[t_{i-1}, t_{i+1}]$,
\begin{align}\label{e-R-wi}
\| \wt_{q+1,i,m+1}\|_{B^{-1+\alpha}_{\infty,\infty}}\lesssim&\tau_q\|\wt_{q+1,i,m+1}(s)\|_{\alpha}\|\vv_q(s)\|_{\alpha}\notag\\
&+\mu^{-1}_{q+1}\|( {A}^{v_1}_{q,k,i,m}, {A}^{v_2}_{q,k,i,m}, \partial_t{A}^{v_1}_{q,k,i,m}, \partial_t{A}^{v_2}_{q,k,i,m})\|_{\alpha}\notag\\
\lesssim& \mu^{-1}_{q+1}\delta_{q+1,m}\ell^{-3{\alpha}}_q.
\end{align}
Similarly, we have
\begin{equation}\label{e-R-di}
\| \dt_{q+1,i,m+1}\|_{B^{-1+\alpha}_{\infty,\infty}}\lesssim \mu^{-1}_{q+1}\delta_{q+1,m}.
\end{equation}
From the definitions of $\RRR_{q,i,m+1}$ and $\MMM_{q,m+1}$ in \eqref{next Tq} and \eqref{next Rq}, we conclude that
  \begin{align*}
\|\RRR_{q,i,m+1}\|_{\alpha}+\|\MMM_{q,m+1}\|_{\alpha}&\lesssim\tau^{-1}_q\mu^{-1}_{q+1}\delta_{q+1,m}\ell^{-3{\alpha}}_q\le \delta_{q+1,m+1}\ell^{\frac{3}{2}\alpha}_q.
\end{align*}
Hence we prove \eqref{estimate-Rv} for $N=0$. For each $N\ge 1$, taking the $N$th derivative of $\eqref{time}_1$ yields
\begin{align*}
    &\del_t \sum_{|\gamma|=N }\partial^{\gamma} \wt_{q+1,i,m+1}+\vv_{q}\nabla \sum_{|\gamma|=N }\partial^{\gamma}\wt_{q+1,i,m+1}\\
    =&-\sum_{|\gamma|=N; \gamma_1+ \gamma_2=\gamma,\gamma_1\neq\gamma}\partial^{\gamma_2}\vv_{q}\nabla \partial^{\gamma_1}\wt_{q+1,i,m+1}- \sum_{ |\gamma|=N;\gamma_1+ \gamma_2=\gamma}\partial^{\gamma_2}\wt_{q+1,i,m+1}\nabla \partial^{ \gamma_1}v_q  \\
    &+\nabla \sum_{|\gamma|=N }\partial^{\gamma}p_{q+1,i,m+1}
\sum_{i;k\in\Lambda_{v}}f_{k,m+1}  \sum_{|\gamma|=N }\partial^{\gamma}\Div  {A}^{v_1}_{q,k,i,m}+\sum_{i;k\in\Lambda_{\theta}}f_{k,m+1}\sum_{|\gamma|=N }\partial^{\gamma}\Div  {A}^{v_2}_{q,k,i,m}.
\end{align*}
Applying $\mathbb{P}_H$ and Lagrangian coordinates to the above equation yields an expression similar to \eqref{int-w}, and then using \eqref{E-Av1}\--\eqref{E-pt-Av2} and \eqref{e-w-alpha}, we have
\begin{align}\label{e-w-N}
\| \wt_{q+1,i,m+1}\|_{N+\alpha}\lesssim& \mu^{-1}_{q+1}\delta_{q+1,m}\ell^{ -N-1-3{\alpha}}_q,\,\,\,\forall\, 0\le N\le L+1.
\end{align}
Similarly, we infer from \eqref{E-Atheta} that $\| \dt_{q+1,i,m+1}\|_{N+\alpha}\lesssim_N \mu^{-1}_{q+1}\delta_{q+1,m}\ell^{ -N-1}_q.$ Then one immediately has \eqref{estimate-Rv} for $n=m+1$. Therefore, we prove \eqref{estimate-Rv}  by induction on $n$, and \eqref{estimate-wt} is also valid.

From the definition of $\MMM_{q,m+1}$ in \eqref{next Tq}, we have
\begin{align*}
\matd{\vv_q}\MMM_{q,m+1} =&\sum_{i}\partial^2_t\bar{\chi}_{i,m+1}\mathcal{R}_{\vex}\dt_{q+1,i,m+1}+\sum_{i}\partial_t\bar{\chi}_{i,m+1}\mathcal{R}_{\vex}\matd{\vv_q} \dt_{q+1,i,m+1}\\
&+\sum_{i}\partial_t\bar{\chi}_{i,m+1}[\mathcal{R}_{\vex},\vv_q\cdot\nabla ] \dt_{q+1,i,m+1}\\
=&\sum_{i}\partial^2_t\bar{\chi}_{i,m+1}\mathcal{R}_{\vex}\dt_{q+1,i,m+1}-\sum_{i}\partial_t\bar{\chi}_{i,m+1}\mathcal{R}_{\vex}(w_{q+1,i,m+1}\cdot\nabla\tb_q)\\
&+\sum_{i}\partial_t\bar{\chi}_{i,m+1}[\mathcal{R}_{\vex},\vv_q\cdot\nabla ] \dt_{q+1,i,m+1},
\end{align*}
where we have used the fact that $$\partial_t\bar{\chi}_{i,m+1}\mathcal{R}_{\vex}\matd{\vv_q} \dt_{q+1,i,m+1}=-\partial_t\bar{\chi}_{i,m+1}\mathcal{R}_{\vex}(\wt_{q+1,i,m+1}\cdot\nabla\tb_q)$$ by $\eqref{time}_2$ and $\partial_t\bar{\chi}_{i,n+1}A^{\theta}_{q,k,i,n }=0$. Therefore,  one infers from \eqref{estimate-wt} that
\begin{align}
\|\matd{\vv_q}\MMM_{q,m+1} \|_{N+\alpha}  \lesssim & \tau^{-2}_{q}\|\mathcal{R}_{\vex}\dt_{q+1,i,m+1}\|_{N+\alpha}+ \tau^{-1}_{q}\|\mathcal{R}_{\vex}(\wt_{q+1,i,m+1}\cdot\nabla\tb_q)\|_{N+\alpha}\notag\\
&+\tau^{-1}_q\|[\mathcal{R}_{\vex},\vv_q\cdot\nabla ] \dt_{q+1,i,m+1}\|_{N+\alpha}\notag\\
 \lesssim&\tau^{-2}_q\mu^{-1}_{q+1}\delta_{q+1,m}+\tau^{-1}_q\mu^{-1}_{q+1}\delta_{q+1,m}\ell^{-N-1-\frac{5\alpha}{2}}_q\notag\\
 \lesssim& \tau^{-1}_q\delta_{q+1,m+1}\ell^{-N+\frac{3}{2}\alpha}_q.\label{DMn+1}
\end{align}
It follows from $\eqref{time}_1$ and $\partial_t\bar{\chi}_{i,m+1}A^{v_1}_{q,k,i,m }=\partial_t\bar{\chi}_{i,m+1}A^{v_2}_{q,k,i,m }=0$ that
\begin{align*}
  &\|\matd{\vv_q}\RRR_{q,i,m+1}\|_{N+\alpha} \\
\lesssim&\tau^{-2}_q\|\mathcal{R}\wt_{q+1,i,m+1}\|_{N+\alpha}+ \tau^{-1}_q(\|\mathbb{P}_H\mathcal{R}(\wt_{q+1,i,m+1}\cdot\nabla \vv_q)\|_{N+\alpha}\notag+\|[\mathcal{R},\vv_q\cdot\nabla ] \wt_{q,i,m+1}\|_{N+\alpha})\\
&+\delta^{-1}_{q+1,m+1}(\|\matd{\vv_q}\nabla\Phi_i\|_{N+\alpha}\|\MMM_{q,m+1}\|^2_0+\|\matd{\vv_q}\nabla\Phi_i\|_{0}\|\MMM_{q,m+1}\|_0\|\MMM_{q,m+1}\|_{N+\alpha}\\&+\|\matd{\vv_q}\nabla\Phi_i\|_{0}\|\MMM_{q,m+1}\|^2_0\|\nabla\Phi_i\|_{N+\alpha}+\|\matd{\vv_q}\MMM_{q,m+1}\|_{N+\alpha}\|\MMM_{q,m+1}\|_0\\ &+\|\matd{\vv_q}\MMM_{q,m+1}\|_{0}\|\MMM_{q,m+1}\|_{N+\alpha}+\|\matd{\vv_q}\MMM_{q,m+1}\|_{0}\|\MMM_{q,m+1}\|_{0}\|\nabla\Phi_i\|_{N+\alpha})\\
 \lesssim& \tau^{-1}_q\delta_{q+1,m+1}\ell^{-N+\frac{3}{2}\alpha}_q.
\end{align*}
Thus, we prove \eqref{estimate-DRv} and complete the proof of Proposition \ref{est-RM}.
\end{proof}

Based on the estimates of $(\wt_{q+1,i,n+1},\dt_{q+1,i,n+1})$, we can derive the estimates for  $(\wtq, \dtq)$.
\begin{prop}[Estimates for $(\wtq, \dtq)$]\label{est-wtq}Let $\wtq$ and $\dtq$ be defined in \eqref{def-wt-dt}. We have
\begin{align}
&\|(\wtq, \dtq)\|_{\alpha}+\ell^{ N}_q\|(\wtq, \dtq)\|_{N+\alpha}\le \mu^{-1}_{q+1}\delta_{q+1}\ell^{-1-{2}\alpha}_q , \label{wt-dt-alpha}\\
& \Big\|\matd {\vv_q+\wtq}\sum_n\RRR_{q,i,n}\Big\|_{N+\alpha}
+\Big\|\matd{\vv_q+\wtq}\sum_n\MMM_{q,n}\Big\|_{N+\alpha}
 \lesssim\tau^{-1}_q\delta_{q+1}\ell^{-N}_q .\label{es-DR-v+w}
\end{align}
\end{prop}
\begin{proof}Applying \eqref{estimate-wt} to \eqref{def-wt-dt} , we have, for $0\le N\le L+1$,
\begin{align*}
\|(\wtq, \dtq)\|_{\alpha}+\ell^{N}_q\|(\wtq, \dtq)\|_{N+\alpha}\lesssim\mu^{-1}_{q+1}\delta_{q+1}\ell^{-1-2{\alpha}}_q\le \delta^{1/2}_{q+1,1}.
\end{align*}
Hence, we derive \eqref{wt-dt-alpha}. Note that the upper bound of $\|\wtq\|_{N+\alpha}$ is less than that of $\|\vv_q\|_{N+\alpha}$, combining this equality with the derivation process of \eqref{estimate-DRv}, we obtain \eqref{es-DR-v+w}.
\end{proof}
The estimates for $\wtq$ enables us to readily derive the following estimate of $\bar{\Phi}_i$.
 \begin{prop}[Estimates for $\bar{\Phi}_i$]\label{estimates-for-inverse-bar-flow} For $a\gg 1$,
 $0\le N\le L+1 $ and every $t\in[t_{i-1}, t_{i+1}]$, we have
\begin{align}
 \|\nabla\bar{\Phi}^{\pm1}_i-{\rm Id}_{3\times3}\|_0 &\le\frac1{10},\label{e:nabla-barphi-i-minus-I3x3}
 \\
% |\phi_k-x|&\lesssim \tau_q\|\vv_q\|_0\lesssim \lambda^{-1/{3}}_1, \label{Phi-x}\\
 \| (\nabla \bar{\Phi}_i)^{-1}\|_N+ \|  \nabla \bar{\Phi}_i\|_N &\le \ell_q^{-N}, \label{e:nabla-barphi-i-CN}
\\
\|\matd{\vv_q+\wtq} \nabla \bar{\Phi}^{\pm1}_i\|_N &\lesssim \delta_q^{1/2} \lambda_q \ell_q^{-N}, \label{e:nabla-barphi-i-matd}\\
 \|\nabla\bar{\Phi}^{\pm1}_i- \nabla{\Phi}^{\pm1}_i \|_{\alpha}&\lesssim \ell^{-2}_q\lambda^2_q(\tfrac{\lambda_q}{\lambda_{q+1}})^{1/3}
 (\tfrac{\delta_{q+1}}{\delta_{q}})^{1/2} . \label{e: barphi-phi}
\end{align}
\end{prop}
\begin{proof}Since $\nabla \bar{\Phi}_i$ is the flow of $\vv_q+\wtq$ and the upper bound of $\|\vv_q+\wtq\|_{N}$ is less than that of $\|\vv_q+\wtq\|_{N}$, we immediately show \eqref{e:nabla-barphi-i-minus-I3x3}--\eqref{e:nabla-barphi-i-matd} by Proposition \ref{p:estimates-for-inverse-flow-map1}. Note that
\begin{equation}
\left\{ \begin{alignedat}{-1}
&\del_t (\bar{\Phi}_i- {\Phi}_i)+ \vv_q  \cdot \nabla  (\bar{\Phi}_i- {\Phi}_i)+ \wtq  \cdot\nabla \bar{\Phi}_i  =0,
 \\
  &(\bar{\Phi}_i- {\Phi}_i)(t_i,x)= 0,
\end{alignedat}\right.  %\label{barphi}
\end{equation}
we deduce by Gr\"{o}nwall's inequality and $2\tau_q\|\vv_q\|_{1+\alpha}\le 1$  that, for $t\in[t_{i-1}, t_{i+1}]$,
\begin{align*}
\|\nabla\bar{\Phi}_i- \nabla{\Phi}_i \|_{\alpha}
&\lesssim  \tau_q(\|\wtq \|_{1+\alpha}\|\nabla \bar{\Phi}_i\|_0+\|\wtq \|_0\|\nabla \bar{\Phi}_i\|_{1+\alpha})
\\
&\lesssim\tau_q\mu^{-1}_{q+1}\ell^{-2-\frac{3\alpha}{2}}_q\delta_{q+1}
=\lambda^{-4\alpha}_q\ell^{-\frac{3\alpha}{2}}_q
(\ell_q\lambda_q)^{-2}(\tfrac{\lambda_q}{\lambda_{q+1}})^{1/3} (\tfrac{\delta_{q+1}}{\delta_{q}})^{1/2}.\end{align*}
This equality implies \eqref{e: barphi-phi} and we complete Proposition \ref{estimates-for-inverse-bar-flow}.
\end{proof}

Observing that the difference between ($\bar{a}^{\theta }_{q,k,i,n},~\bar{a}^{v }_{q,k,i,n},~ \bar{A}^{\theta}_{q,k,i,n}, ~\bar{A}^{v_1}_{q,k,i,n} ,~\bar{A}^{v_2}_{q,k,i,n})$ and $({a}^{\theta }_{q,k,i,n}$, ${a}^{v }_{q,k,i,n}, {A}^{\theta}_{q,k,i,n}, {A}^{v_1}_{q,k,i,n},{A}^{v_2}_{q,k,i,n})$ only depend  on   $\bar{\Phi}_i$ and $\Phi_i$, we easily show the following proposition by combining \eqref{E-a-theta}--\eqref{E-Av2} with Proposition \ref{estimates-for-inverse-bar-flow}.
\begin{prop}\label{barA-A}  For $0\le n\le L-1$, $0\le N\le L+1$ and $t\in[t_{i-1}, t_{i+1}]$, we have
 \begin{align*}
 &\|(\bar{a}^{\theta }_{q,k,i,n}, \bar{a}^{v}_{q,k,i,n})\|_{N+\alpha}\lesssim\delta^{1/2}_{q+1,n}\ell^{-(1+\frac{1}{2}\alpha)N}_q,\\
&\|(\matd {\vv_q+\wtq}\bar{a}^{\theta }_{q,k,i,n},\matd {\vv_q+\wtq}\bar{a}^{v }_{q,k,i,n})\|_{N+\alpha}
\lesssim \tau^{-1}_q\delta^{1/2}_{q+1,n}\ell^{-(1+\frac{1}{2}\alpha)N}_q ,\\
&\|({A}^{\theta}_{q,k,i,n}-\bar{A}^{\theta}_{q,k,i,n},{A}^{v_1}_{q,k,i,n}-\bar{A}^{v_1}_{q,k,i,n},{A}^{v_2}_{q,k,i,n}-\bar{A}^{v_2}_{q,k,i,n})\|_{\alpha}
\lesssim \ell^{-2}_q\lambda^2_q\big(\tfrac{\lambda_q}{\lambda_{q+1}}\big)^{1/3}
 \big(\tfrac{\delta_{q+1}}{\delta_{q}}\big)^{1/2}\delta_{q+1}.
\end{align*}
\end{prop}
Based on Propositions \ref{est-wtq} and \ref{barA-A}, we  obtain the following bounds.
\begin{prop}[Estimates for $w_{q+1}$ and $d_{q+1}$]\label{estimate-wq+11} For $0\leq n\le L-1$ and $0\le N\le L$, there exists a universal constant $M$ satisfying
\begin{align}
&\|\Wp_{q+1,n}\|_{\alpha}+\tfrac{1}{\lambda^N_{q+1}}\|\Wp_{q+1,n}\|_N\lesssim \delta^{1/2}_{q+1,n} ,\label{estimate-wp}\\
&\|\wc_{q+1,n}\|_{\alpha}+\tfrac{1}{\lambda^N_{q+1}}\|\wc_{q+1,n}\|_N\lesssim\lambda^{-1}_{q+1}\delta^{1/2}_{q+1,n}\ell^{-1-\frac{\alpha}{2}}_q,\label{estimate-wc}\\
&\|\Dp_{q+1,n}\|_{\alpha}+\tfrac{1}{\lambda^N_{q+1}}\|\Dp_{q+1,n}\|_N\lesssim \delta^{1/2}_{q+1,n}.  \label{estimate-d}\\
&\|(w_{q+1},d_{q+1 })\|_0+\tfrac{1}{\lambda^N_{q+1}}\|(w_{q+1},d_{q+1 })\|_N\le \frac{M}{2}\delta^{1/2}_{q+1}.\label{estimate-w}
\end{align}
\end{prop}
\subsubsection{Step 5: Estimates for $(\RR_{q+1}, T_{q+1})$} Let us define $(v_{q+1},\theta_{q+1}) $ by
$$v_{q+1}=\vv_q+w_{q+1},\,\,\theta_{q+1}= \tb_{q}+d_{q+1}.$$
One verifies that $(v_{q+1}, \theta_{q+1}, P_{q+1}, \RR_{q+1}, \MM_{q+1})$ satisfies \eqref{e:subsol-B} with replacing $q$ by $q+1$,
where
\begin{align*}
\Div\RR_{q+1}
=& \Div\underbrace{\mathcal{R}\big( \matd{\vv_q+\wtq} (\wpq+\wcq)\big)}_{\Rtransport}+\Div\underbrace{\mathcal R \big((\wpq+\wcq)\cdot\nabla (\vv_q+\wtq)  \big)}_{\Rnash}\\
&+\Div   \Big(\wcq\otimes w_{q+1}+w_{q+1}\otimes \wcq-\wcq\otimes \wcq +\wtq\otimes \wtq \notag\\
&+ \mathcal{R}(\partial_t\wtq+ \vv_q  \cdot\nabla\wtq+\wtq \cdot\nabla\vv_q )+\mathcal{R}\nabla P^{(\text{t})}_{q+1}+\mathcal{R}\Div(\wpq\otimes \wpq+\RRR_{q} )\Big)\\
 =:& \Div\Rtransport+\Div\Rnash+\Div\Rosc ,
\end{align*}
and
\begin{align*}
\Div\MM_{q+1}
=& \Div\underbrace{\mathcal R_{\vex}(\matd{\vv_q+\wtq}\dpq )}_{\Mtrans}+\Div\underbrace{ \mathcal R_{\vex}((\wpq+\wcq)\cdot\nabla(\tb_q+\dtq))}_{\Mnash}\\
&+\Div   \Big(\wtq\otimes \dtq+\wcq\otimes\dpq+ \vv_q\otimes \dtq+\wtq \otimes \tb_q \notag\\
&
 +\mathcal{R}_{\vex}\Div\big(\wpq\otimes \dpq+\MMM_{q}\big)+\mathcal{R}_{\vex}\partial_t\dtq \Big)\\
  =:& \Div\Mtrans+\Div\Mnash+\Div\Mosc ,
\end{align*}
and $$P_{q+1}=\pp_q+P^{(\text{t})}_{q+1}, \,\,P^{(\text{t})}_{q+1}=\sum_i\sum_{n=0}^{L-1}p^{(\text{t})}_{q+1,i,n+1}.$$Moreover, we have by \eqref{av-pi} that
\begin{align}\label{av-Pt}
\int_{\TTT^2}P^{(t)}_{q+1}\dd x=0.
\end{align}

\begin{prop}[Estimates for $(\Mtrans,\Rtransport)$]\label{proptrans1}
\begin{equation}\label{Trans}
\begin{aligned}
\|(\Mtrans,\Rtransport)\|_{\alpha}\lesssim\delta_{q+2}\lambda^{-5\alpha}_{q+1}.
\end{aligned}
\end{equation}
\end{prop}
\begin{proof}Note  that $\matd {\vv_q+\wtq} \phi_{(\lambda_{q+1}k)}(\bar{\Phi}_i)=0$, we can express $\Mtrans$ directly from the definition of $\dpq$ in~\eqref{defi-dpq} as
\begin{align*}
\Mtrans=\mathcal{R}_{\vex}\Big(\sum_{n=0}^{L-1}\sum_{i;k\in\Lambda_{\theta }}\delta^{1/2}_{q+1,n}(\chi'_{i,n}g_{k,n}+\chi_{i,n}g'_{k,n})
\phi_{(\lambda_{q+1}k)}(\bar{\Phi}_i(x,t))\Big).
\end{align*}
By Lemma \ref{l:non-stationary-phase}, we have
\begin{align}\label{e-Trans}
\|\Mtrans\|_{\alpha}\lesssim \lambda^{-1+\alpha}_{q+1}\mu_{q+1}\delta^{1/2}_{q+1} =\big(\frac{\lambda_q}{\lambda_{q+1}}\big)^{\frac{2}{3}+2\alpha}\lambda^{3\alpha}_{q+1}\delta_{q+1}.
\end{align}
Thanks to \eqref{wpq-wcq}, one deduces that
\begin{align}
&\Rtransport=\mathcal{R}\Big(\matd {\vv_q+\wtq}(\wpq +\wcq)\Big)\notag\\
=& \lambda^{-1}_{q+1} \mathcal{R}\Big(\sum_{n=0}^{L-1}\sum_{i;k\in\Lambda_v}\matd{\vv_q+\wtq}\nabla^{\perp}\big(g_{k,n+1}(\mu_{q+1}t) \bar{a}^v_{q,k,i,n}
    \psi(\lambda_{q+1}\bar{k}\cdot\bar{\Phi}_i )\big) \notag\\
    &+\sum_{n=0}^{L-1}
\sum_{i;k\in\Lambda_{\theta }}\matd{\vv_q+\wtq}\nabla^{\perp}\big(g_{k,n+1}(\mu_{q+1}t)
 { \bar{a}^{\theta }_{q,k,i,n} }
  \psi(\lambda_{q+1}\bar{k}\cdot\bar{\Phi}_i )\big)\Big)\notag\\
  =&\lambda^{-1}_{q+1} \mathcal{R}\Big(\sum_{n=0}^{L-1}\sum_{i;k\in\Lambda_v}\big(\nabla^{\perp}\matd{\vv_q+\wtq}\big(g_{k,n+1}(\mu_{q+1}t) \bar{a}^v_{q,k,i,n}\big)\big)
    \psi(\lambda_{q+1}\bar{k}\cdot\bar{\Phi}_i )\notag\\
    &-\sum_{n=0}^{L-1}\sum_{i;k\in\Lambda_v}\nabla^{\perp}(\vv_q+\wtq)\cdot\nabla   \big(\bar{a}^v_{q,k,i,n}g_{k,n+1}(\mu_{q+1}t)\psi(\lambda_{q+1}\bar{k}\cdot\bar{\Phi}_i )\big)\notag\\
    &+ \sum_{n=0}^{L-1}\sum_{i;k\in\Lambda_{\theta}}\big(\nabla^{\perp}\matd{\vv_q+\wtq}\big(g_{k,n+1}(\mu_{q+1}t) \bar{a}^{\theta}_{q,k,i,n}\big)\big)
    \psi(\lambda_{q+1}\bar{k}\cdot\bar{\Phi}_i )\notag\\
    &-\sum_{n=0}^{L-1}\sum_{i;k\in\Lambda_{\theta}}\nabla^{\perp}(\vv_q+\wtq)\cdot\nabla   \big(\bar{a}^{\theta}_{q,k,i,n}g_{k,n+1}(\mu_{q+1}t)\psi(\lambda_{q+1}\bar{k}\cdot\bar{\Phi}_i )\big)\Big).\notag\end{align}
Applying Lemma \ref{l:non-stationary-phase} and Proposition \ref{barA-A} to the above equality. we have
\begin{align}\label{tranw}
\|\Rtransport\|_{\alpha}\lesssim \lambda^{-2+\alpha}_{q+1}\mu_{q+1}\ell^{-1-\frac{\alpha}{2}}_q\delta^{1/2}_{q+1}.
\end{align}
This inequality combined with \eqref{e-Trans} yields
\begin{align*}
    \|(\Mtrans,\Rtransport)\|_{\alpha}\lesssim \big(\frac{\lambda_q}{\lambda_{q+1}}\big)^{\frac{2}{3}+2\alpha}\lambda^{3\alpha}_{q+1}\delta_{q+1}\lesssim \delta_{q+2}\lambda^{-5\alpha}_{q+1},
\end{align*}
where we have used the conditions $b<\frac{1}{3\beta}$ and $\alpha<\frac{-(2\beta b^2-(2\beta+\frac{2}{3})b+\frac{2}{3})}{100}$, and this condition can be guaranteed by \eqref{con-b} and \eqref{e:params0}. Hence, we complete the proof of Proposition \ref{proptrans1}.
\end{proof}
\begin{prop}[Estimates for $(\Mnash,\Rnash)$]\label{nash1}
\begin{equation}
\begin{aligned}
\|(\Mnash,\Rnash)\|_{\alpha}\lesssim \delta_{q+2}\lambda^{-7\alpha}_{q+1}.
\end{aligned}
\end{equation}
\end{prop}
\begin{proof}
Combining with Proposition \ref{est-vvq},  Proposition~\ref{barA-A}, Proposition~\ref{est-wtq} and   Lemma \ref{l:non-stationary-phase}, we have
\begin{align*}
\|(\Mnash,\Rnash)\|_{\alpha}
\lesssim&\frac{M\delta^{1/2}_{q+1}\delta^{1/2}_q\lambda_q}{\lambda^{1-\alpha}_{q+1}}+\frac{M\delta^{1/2}_{q+1}\delta^{1/2}_q\lambda_q}{\lambda^{N_0-\alpha}_{q+1}\ell^{N_0+\alpha}_q}\lesssim\delta_{q+2}\lambda^{-7\alpha}_{q+1},
\end{align*}
where we have used \eqref{lambdaN} and\eqref{wpq-wcq}. Hence, we complete the proof of Proposition \ref{nash1}.
\end{proof}
\begin{prop}[Estimates for ($\Mosc$, $\Rosc$)]\label{propMosc}
\begin{align}
\|(\Mosc,\Rosc)\|_{\alpha}\lesssim\delta_{q+2}\lambda^{-5\alpha}_{q+1}.
\end{align}
\end{prop}
\begin{proof}We decompose $\Mosc$ into two parts $\Tosco$ and $\Tosct$, where
\begin{align*}
\Tosco= \mathcal{R}_{\vex}\big(\Div(\MMM_q+ \wpq\otimes\dpq)+\partial_t\dtq+\vv_q\cdot\nabla\dtq+\wtq\cdot\nabla\bar{\theta}_q\big)
\end{align*}
and
\begin{align*}
\Tosct=\wtq\otimes \dtq+\wcq\otimes\dpq.
\end{align*}
For $\Tosco$, since $\chi_{i,n}g_{k,n}\cdot \chi_{i',n'}g_{k',n'}=0$ for $(k,i,n)\neq(k',i',n')$, we deduce from the definitions of $\wpq$ and $\dpq$ that
\begin{align}
&\mathcal{R}_{\vex}\Div\big(\MMM_q+\sum_{n;i;k\in\Lambda_{\theta}}g^2_{k,n+1} \bar{A}^{\theta}_{q,k,i,n}\phi^2_{(\lambda_{q+1}k)}(\bar{\Phi}_i(x,t))\big)\notag\\
=&\mathcal{R}_{\vex}\Div\big(\MMM_q+\sum_{n;i;k\in\Lambda_{\theta}}g^2_{k,n+1}\bar{A}^{\theta}_{q,k,i,n}+\sum_{n;i;k\in\Lambda_{\theta}}g^2_{k,n+1}\bar{A}^{\theta}_{q,k,i,n}\mathbb{P}_{\neq 0}\big(\phi^2_{(\lambda_{q+1}k)}(\bar{\Phi}_i(x,t))\big)\big)\notag\\
=&\mathcal{R}_{\vex}\Div\big(\MMM_q+\sum_{n;i;k\in\Lambda_{\theta}}g^2_{k,n+1}{A}^{\theta}_{q,k,i,n}+
 \sum_{n;i;k\in\Lambda_{\theta}}g^2_{k,n+1} ( \bar{A}^{\theta}_{k,i,n}-{A}^{\theta}_{q,k,i,n})\notag\\
&+\sum_{n;i;k\in\Lambda_{\theta}}g^2_{k,n+1}\bar{A}^{\theta}_{q,k,i,n}\mathbb{P}_{\neq 0}\big(\phi^2_{(\lambda_{q+1}k)}(\bar{\Phi}_i(x,t))\big)\big)\notag.
\end{align}
Using the fact that
\begin{align}
    \sum_{n;i;k\in\Lambda_{\theta}}{A}^{\theta}_{q,k,i,n}&= \sum_{n;i;k\in\Lambda_{\theta}}   \chi^2_{i,n}  a^{\theta}_{i}  (\nabla\Phi_i\MMM_{q,n})
\nabla\Phi^{-1}_ik \notag\\
&= -\sum_i\sum_{n=0}^{L-1 } \chi^2_{i,n} \MMM_{q,n}  = - \MMM_q- \sum_i\sum_{n=0}^{L-2 } \partial_t\bar{\chi}_{i,n+1}\mathcal{R}_{\vex}d_{q+1,i,n+1},\notag
\end{align}
and
\begin{align*}
\mathcal{R}_{\vex}\partial_t\dtq=\sum_{i}\sum_{n=0}^{L-1} \partial_t\bar{\chi}_{i,n+1}(t)\mathcal{R}_{\vex}\dt_{q+1,i,n+1} + \sum_{i}\sum_{n=0}^{L-1}\bar{\chi}_{i,n+1}(t)\mathcal{R}_{\vex}\partial_t\dt_{q+1,i,n+1},
\end{align*}
we have
\begin{align*}
 \Tosco=&\sum_i\partial_t\bar{\chi}_{i,L}(t)\mathcal{R}_{\vex}\dt_{q+1,i,L}
\sum_{n;i;k\in\Lambda_{\theta}}g^2_{k,n+1}\mathcal{R}_{\vex}\Div ( \bar{A}^{\theta}_{q,k,i,n}-{A}^{\theta}_{q,k,i,n})\notag\\
&+\sum_{n;i;k\in\Lambda_{\theta}}\mathcal{R}_{\vex}\big(g^2_{k,n+1}\Div\bar{A}^{\theta}_{q,k,i,n}\mathbb{P}_{\neq 0}\big(\phi^2_{(\lambda_{q+1}k)}(\bar{\Phi}_i(x,t))\big)\big).
\end{align*}
Using this  equality together with \eqref{e-R-di}, Proposition \ref{barA-A} and Lemma \ref{l:non-stationary-phase}, we conclude that
\begin{align*}
\|\Tosco\|_{\alpha}\lesssim&\tau^{-1}_q\mu^{-1}_{q+1}\delta_{q+1,L-1}+\ell^{-2}_q\lambda^2_q\big(\frac{\lambda_q}{\lambda_{q+1}}\big)^{1/3}
 \big(\frac{\delta_{q+1}}{\delta_{q}}\big)^{1/2}\delta_{q+1}+\frac{\delta^{1/2}_{q+1}\ell^{-1+\frac{\alpha}{2}}_q}{\lambda^{1-\alpha}_{q+1}}\\
\lesssim&\ell^{-2\alpha}_q\delta_{q+1,L}+\delta_{q+2}\lambda^{-5\alpha}_{q+1}\lesssim \delta_{q+2}\lambda^{-5\alpha}_{q+1}.
\end{align*}
By Proposition \ref{est-wtq} and Proposition \ref{estimate-wq+11}, one deduces that
\begin{align*}
\|\Tosct\|_{\alpha}\lesssim  \mu^{-2}_{q+1}\delta^2_{q+1}\ell^{-2-\alpha}_q+M\lambda^{-1}_{q+1}\ell^{-1-\frac{\alpha}{2}}_q\delta_{q+1}.
\end{align*}
Hence we obtain
\begin{align*}
\|\Mosc\|_{\alpha}\le \|\Tosco\|_{\alpha}+\|\Tosct\|_{\alpha}\lesssim \delta_{q+2}\lambda^{-5\alpha}_q.
\end{align*}
Let us now proceed to estimate $\Rosc$ by dividing it into two parts, $\Rosco$ and $\Rosct$, where
\begin{align*}
  \Rosco=  \mathcal{R}\big(\Div(\wpq\otimes \wpq+\RRR_{q})+\partial_t\wtq+\vv_q\cdot\nabla\wtq+\wtq\cdot\nabla\vv_q+\nabla P^{(\text{t})}_{q+1}\big)
\end{align*}
and
\begin{align*}
\Rosct=&\wcq\otimes w_{q+1}+w_{q+1}\otimes \wcq-\wcq\otimes \wcq +\wtq\otimes \wtq.
\end{align*}
For $\Rosco$, since $\chi_{i,n}g_{k,n}\cdot \chi_{i',n'}g_{k',n'}=0$ for $(k,i,n)\neq(k',i',n')$ and $\mathbb{P}_{=0}\big(\phi^2_{(\lambda_{q+1}k)}(\bar\Phi_i)\big)=1$, we obtain
\begin{align}\label{Pu}
& \Div(  \RRR_q+ \wpq\otimes\wpq  ) \notag\\
=& \Div\Big(\RRR_q+\sum_{ n;i;k\in\Lambda_{v}}g^2_{k,n} {A}^{v_1}_{q,k,i,n}+\sum_{ n;i;k\in\Lambda_{\theta}} g^2_{k,n}{A}^{v_2}_{q,k,i,n} + \sum_{ n;i;k\in\Lambda_{v}}g^2_{k,n}\mathbb{P}_{\neq0}(\phi^2_{(\lambda_{q+1}k)}(\bar{\Phi}_i))\bar{A}^{v_1}_{q,k,i,n}
\notag\\
&+\sum_{n;i;k\in\Lambda_{\theta}}g^2_{k,n}\mathbb{P}_{\neq0}(\phi^2_{(\lambda_{q+1}k)}(\bar{\Phi}_i))\bar{A}^{v_2}_{q,k,i,n}+
 \sum_{n;i;k\in\Lambda_{v}}g^2_{k,n}(\bar{A}^{v_1}_{q,k,i,n}-{A}^{v_1}_{q,k,i,n})\notag\\
&+\sum_{n;i;k\in\Lambda_{\theta}}g^2_{k,n} ( \bar{A}^{v_2}_{q,k,i,n}-{A}^{v_2}_{q,k,i,n})\Big).
\end{align}
By Geometric lemma \ref{first S}, one gets
\begin{align*}
  & \Div\sum_{n;i;k\in\Lambda_{\theta}}{A}^{v_1}_{q,k,i,n}\\
   =& \Div\sum_{n;i;k\in\Lambda_{\theta}}  (a^{v}_{q,k,i,n})^2
\nabla\Phi^{-1}_i(k\otimes k)\nabla\Phi^{-\TT}_i=\Div\sum_{n;i}\chi^2_{i,n}(\delta_{q+1,n}{\rm Id}-\RRR_{q,i,n})\\
=&-\Div\RRR_q-\sum_i\sum_{k\in\Lambda_{\theta}}\delta^{-1}_{q+1}(a^{\theta}_k(\nabla\Phi_i\MMM_q))^2\nabla \Phi^{-1}_i(k\otimes k)\nabla\Phi^{-\TT}_i)-\Div\sum_i\sum_{n=1}^{L-1}\chi^2_{i,n}\RRR_{q,i,n}\\
=&-\Div\RRR_{q }-\Div\sum_{n=0}^{L-2}\sum_{  i}\partial_t\bar{\chi}_{i,n+1}\mathcal{R}\wt_{q+1,i,n+1}\\
&+\Div\sum_i\sum_{n=0 }^{L-1}\sum_{  k\in\Lambda_{\theta}}\chi^2_{i,n}\delta^{-1 }_{q+1,n}   \big(a^{\theta}_{k}  (\nabla\Phi_i\MMM_{q,n})\big)^2\nabla\Phi^{-1}_i(k\otimes k)\nabla\Phi^{-\TT}_i\\
  =&  -\Div\RRR_{q }-\sum_{n=1 }^{L-1}\sum_{i }\partial_t\bar{\chi}_{i,n}\wt_{q+1,i,n}-\sum_{n=0 }^{L-1}\sum_{ i;k\in\Lambda_{\theta}}\Div{A}^{v_2}_{q,k,i,n}.
\end{align*}
Combining this equality with $\mathbb{P}_{= 0}(g^2_{k,n+1})=1$ and $f_{k,n+1}:=\mathbb{P}_{\neq 0}(g^2_{k,n+1})$, we obtain from \eqref{Pu} that
\begin{align*}
& \Div(  \RRR_q+ \wpq\otimes\wpq  ) \\
=-& \sum_{ n;i;k\in\Lambda_{v}}f_{k,n+1}\Div{A}^{v_1}_{q,k,i,n}-\sum_{ n;i;k\in\Lambda_{\theta}} f_{k,n+1}\Div{A}^{v_2}_{q,k,i,n}-\sum_{n=1 }^{L-1}\sum_{i }\partial_t\bar{\chi}_{i,n}\wt_{q+1,i,n}
\notag\\
& +\sum_{ n;i;k\in\Lambda_{v}}g^2_{k,n+1}\mathbb{P}_{\neq0}(\phi^2_{(\lambda_{q+1}k)}(\bar{\Phi}_i))\Div \bar{A}^{v_1}_{q,k,i,n}+\sum_{n;i;k\in\Lambda_{\theta}}g^2_{k,n+1}\mathbb{P}_{\neq0}(\phi^2_{(\lambda_{q+1}k)}(\bar{\Phi}_i))\Div \bar{A}^{v_2}_{q,k,i,n}\\
&+\sum_{n;i;k\in\Lambda_{v}}g^2_{k,n+1}\Div(\bar{A}^{v_1}_{q,k,i,n}-{A}^{v_1}_{q,k,i,n})+\sum_{n;i;k\in\Lambda_{\theta}}g^2_{k,n+1} \Div( \bar{A}^{v_2}_{q,k,i,n}-{A}^{v_2}_{q,k,i,n}).
\end{align*}
This equality together with $\eqref{time}_1$ shows that
\begin{align*}
    \Rosco=&\mathcal{R}\big(-\sum_{i }\partial_t\bar{\chi}_{i,L}\wt_{q+1,i,L} +\sum_{ n;i;k\in\Lambda_{v}}g^2_{k,n+1}\mathbb{P}_{\neq0}(\phi^2_{(\lambda_{q+1}k)}(\bar{\Phi}_i))\Div \bar{A}^{v_1}_{q,k,i,n}\\
&+\sum_{n;i;k\in\Lambda_{\theta}}g^2_{k,n+1}\mathbb{P}_{\neq0}(\phi^2_{(\lambda_{q+1}k)}(\bar{\Phi}_i))\Div \bar{A}^{v_2}_{q,k,i,n}+\sum_{n;i;k\in\Lambda_{v}}g^2_{k,n+1}\Div(\bar{A}^{v_1}_{q,k,i,n}-{A}^{v_1}_{q,k,i,n})\\
&+\sum_{n;i;k\in\Lambda_{\theta}}g^2_{k,n+1} \Div( \bar{A}^{v_2}_{q,k,i,n}-{A}^{v_2}_{q,k,i,n})\big)
\end{align*}
By \eqref{E-Av1}, \eqref{E-Av2}, \eqref{e-R-wi} and Proposition \ref{barA-A}, we have
\begin{align*}
\|\Rosco\|_{\alpha}\lesssim&\tau^{-1}_q\mu^{-1}_{q+1}\delta_{q+1,L-1}\ell^{-1-{2\alpha}}_q+{\delta_{q+1}\ell^{-\frac{\alpha}{2}}_q}\lambda^{-1+\alpha}_{q+1}+\ell^{-2}_q\lambda^2_q\big(\frac{\lambda_q}{\lambda_{q+1}}\big)^{1/3}
 \big(\frac{\delta_{q+1}}{\delta_{q}}\big)^{1/2}\delta_{q+1}\\
 \lesssim&\delta_{q+2}\lambda^{-5\alpha}_{q+1}.
\end{align*}
It follows from Proposition \ref{estimate-wq+11} that
\begin{align*}
\|\Rosct\|_{\alpha}\lesssim&M\lambda^{-1}_{q+1}\ell^{-1-\frac{\alpha}{2}}_q\delta_{q+1}\lesssim \delta_{q+2}\lambda^{-6\alpha}_{q+1}.
\end{align*}
Thus, we finish the proof of Proposition \ref{propMosc}.
 \end{proof}
  Now we shall demonstrate that  $(v_{q+1}, \theta_{q+1}, \RR_{q+1}, \MM_{q+1})$  \eqref{p_q}--\eqref{e:initial} with $q$ replaced by $q+1$ as outlined below.

First of all, we infer from \eqref{ppp} that $\int_{\TTT^3}\bar{p}_q\dd x=0$. Since $p_{q+1}=\bar{p}_q+P^{{t}}_{q+1}$, we get from \eqref{av-Pt} that $p_{q+1}$ satisfies \eqref{p_q} at $q+1$ level.
Next, we have by \eqref{e:stability-vv_q-N} and \eqref{estimate-w}
\begin{align*}
\|(v_{q+1}, \theta_{q+1})\|_0\le&\|(\vv_{q}, \tb_{q})\|_0+\|(w_{q+1}, d_{q+1})\|_0\\
\le&M\sum_{i=1}^q\delta^{1 / 2}_i+ \delta_{q+1}^{1/2} \ell_q^\alpha +\frac{M}{2}\delta^{1/2}_{q+1}=M\sum_{i=1}^{q+1}\delta^{1 / 2}_i,
\end{align*}
This inequality shows that $(v_{q+1}, \theta_{q+1})$ satisfies \eqref{e:vq-C0} at $q+1$ level. By virtue of \eqref{e:vv_q-bound} and \eqref{estimate-w}, we obtain for $1\leq N\leq L$,
\begin{align*}
\|(v_{q+1}, \theta_{q+1})\|_N\le &\|(\vv_{q}, \tb_{q})\|_N+\|(w_{q+1}, d_{q+1})\|_N\\
\le&  CM\delta_{q}^{1 / 2} \lambda_{q} \ell_q^{-N} +\frac{M}{2} \delta_{q+1}^{1 / 2} \lambda^N_{q+1}\le M \delta_{q+1}^{1 / 2} \lambda^N_{q+1}.
\end{align*}
Collecting Proposition \ref{proptrans1}--Proposition \ref{propMosc} together shows that
\begin{align*}
\|(\RR_{q+1}, \MM_{q+1})\|_{0}\le  \delta_{q+2}\lambda_{q+1}^{-3\alpha}.
\end{align*}
Noting  that
\[\supp w_{q+1},\,\, \supp d_{q+1}\subset \cup_{i=0}^{i_{\max}-1}(I_i+\big(-\tfrac{\tau_q}{6}, \tfrac{\tau_q}{6}\big))\subset\big(2\widetilde{T}+\tfrac{\tau_q}{6}, 3\widetilde{T}-\tfrac{\tau_q}{6}\big),\]
we deduce
$$(v_{q+1}, \theta_{q+1})(t,x)=(\vv_q, \tb_q)(t,x), \quad t\in[0, 2\widetilde{T}+\tfrac{\tau_q}{6}]\cup [3\widetilde{T}-\tfrac{\tau_q}{6}, 1], \,x\in\TTT^2.$$ Combining with \eqref{vvqtbq} and $20\tau_{q+1}<\tau_q$ , we gets
\begin{align*}
(v_{q+1}, \theta_{q+1})\equiv (v^{(1)}, b^{(1)}) \,\,{\rm on}\,\,\big[0, 2\widetilde{T}+\tau_{q+1}\big],\quad(v_{q+1}, \theta_{q+1})\equiv (v^{(2)}, b^{(2)}) \,\,\text{\rm on}\,\,\big[3\widetilde{T}-{\tau_{q+1}}, 1\big].
\end{align*}
This fact shows that \eqref{e:initial1} and \eqref{e:initial} hold at $q+1$ level. Hence, we complete the proof of Proposition~\ref{p:main-prop}.

\subsection{Proof of Theorem  \ref{t:main-energy1}}To prove  Theorem  \ref{t:main-energy1}, we modify the iterative scheme as in the proof of Proposition \ref{p:main-prop} by introducing an energy profile.
\begin{prop}\label{Prope}
Let $b\in(1,b_0)$ and $\alpha$ satisfy \eqref{e:params0}.
If $(v_q, \theta_q, p_q,\RR_q, \MM_q)$ obeys the equations \eqref{e:subsol-B} and \eqref{p_q} with
\begin{align}
     &\|(v_q, \theta_q)\|_{0} \le M\sum_{i=1}^q\delta^{1/2}_i ,
    \label{e:vq-C02}
    \\
    &\|(v_q, \theta_q)\|_{N} \le M\delta_{q}^{1/2} \lambda^N_q ,\quad 1\leq N\leq L,
    \label{e:vq-C1}
    \\
    &\|(\RR_q, \MM_q)\|_{0} \le \delta_{q+1}\lambda_q^{-4\alpha} ,
    \label{e:RR_q-C0}\\
    &
10M^2\delta_{q+1}\leq e(t)-\int_{\mathbb{T}^2}(|v_q|^2+2|\theta_q|^2)\dd x \leq 30M^2\delta_{q+1},\label{v-theta-energy2}
\end{align}
 then there exist smooth functions $(v_{q+1}, \theta_{q+1}, p_{q+1}, \RR_{q+1}, \MM_{q+1})$ satisfying \eqref{e:subsol-B}, \eqref{p_q},
\eqref{e:vq-C02}--\eqref{v-theta-energy2}
with $q$ replaced by $q+1$ and
\begin{align}
        \big\|(v_{q+1}, \theta_{q+1}) - (v_q, \theta_q)\big\|_{0} +\frac1{\lambda_{q+1}} \big\|(v_{q+1}, \theta_{q+1})-(v_q, \theta_q)\big\|_1 &\le   M\delta_{q+1}^{1/2}.
        \label{e:velocity-diff}
\end{align}
\end{prop}
\begin{proof}Similar to the proof of Proposition \ref{p:main-prop}, the construction of the weak solution at the
$q+1$-th step is split into three parts: mollification, gluing, and perturbation. The new ingredient  lies in the perturbation step, where we introduce an additional energy perturbation $d^{(\text{e})}$.

Before mollification, we set the initialization to be
$$\theta_1=0, \quad v_1=\chi(t)\phi(\bar{k}\cdot x){k}\quad \textup{and}\quad \RR_1:=\mathcal{R}\big(\chi'(t)\cdot\phi(\bar{k}\cdot x){k}\big),$$
where $\chi(t)$ is a smooth, monotonically decreasing function with $\chi(0)=1$ and $\chi(T)=0$ and $\phi$ the periodic  function defined in \eqref{phi-xi}.\\
\noindent \textbf{Mollification}: ${(v_q, p_q, \theta_q, \RR_q, \MM_q)\mapsto (v_{\ell_q},  p_{\ell_q}, \theta_{\ell_q}, \RR_{\ell_q}, \MM_{\ell_q})}$.  Here ${(v_{\ell_q},  p_{\ell_q}, \theta_{\ell_q}, \RR_{\ell_q}, \MM_{\ell_q})}$ is consistent with that in \eqref{moll-v}--\eqref{moll-T} and satisfies the estimates in Proposition \ref{p:estimates-for-mollified1}. Moreover, one infers from \eqref{e:v_ell-vq1} that
\[8M^2\delta_{q+1}\le e(t)-\int_{\TTT^2}(|v_{\ell_q}|^2+2|\theta_{\ell_q}|^2)\dd x\le 40M^2\delta_{q+1}.\]
\noindent \textbf{Gluing}: $ (v_{\ell_q},  p_{\ell_q}, \theta_{\ell_q}, \RR_{\ell_q}, \MM_{\ell_q})\mapsto (\vv_q,  \bar{p}_q, \bar{\theta}_q, \RRR_{q}, \MMM_{q})$.
 The definition of $(\vv_q,  \bar{p}_q, \bar{\theta}_q)$ here shares the same manner as in \eqref{e:pp_q}, in which $(\vex_0, \bm{\theta}_0, \pex_0)=(\vex_{i_{\max}}, \bm{\theta}_{i_{\max}}, \pex_{i_{\max}})=(v_1,p_1,\theta_1)$, for $1\le i\le i_{\max}-1$, $(\vex_i, \bm{\theta}_i, \pex_i)$ is the unique exact solution to the system \eqref{e:exact-B} with initial data $(v_{\ell_q}(i\tau_q), \theta_{\ell_q}(i\tau_q))$. \\
 \noindent \textbf{Perturbation}: We define $ h: \mathbb{T}^2 \rightarrow \mathbb{T}^2 $ as a smooth, mean-free cutoff function supported on the interval $[0, \frac{1}{100}]^2$ with $\|h\|_{L^2} = 1$.  We define
\[
h_{\Lambda} := h(\lambda_{q+1}(x - x_{\Lambda}))
\]
such that
\[
h_{\Lambda} \cdot \phi_{(\lambda_{q+1}k)} = 0 \quad \forall k \in \Lambda,
\]
where $ \phi_{(\lambda_{q+1}k)}$ is defined in \eqref{phi-xi}. Furthermore, we define the energy gap $e_q$ by
   \begin{align*}
e_q(t) \coloneq& \frac{1}{\sum_i  \int_{\mathbb{T}^2}{\tilde \eta_i}^2\dd x}\Big(e(t) -\int_{\mathbb T^2}|\vv_q|^2+|\bar{\theta}_{q}|^2+|\wpq|^2+2|\dpq|^2+ {E(t,x)} \dd x- 20\delta_{q+2}\Big),
\end{align*}
where
\begin{align*}
 E(t,x):= &|\wtq|^2+2\wtq\cdot\vv_q
+2(| \dtq|^{2}
+2\dtq\bar{\theta}_{q})
\end{align*}
and the ${\tilde \eta_i(t,x)}_{i\geq0} $ is introduced as in \cite{BDSV,KMY} with the following properties:
    \begin{enumerate}
    \item $\tilde\eta_i \in C^\infty_c(\mathbb T^2 \times (J_{i}\cup I_{i}\cup J_{i+1}) ; [0,1])$ and $\tilde\eta_i(\cdot,t) \equiv 1$ for $t\in I_i$ such that:
        \begin{align}
         \|\del^n_t  \tilde\eta_i\|_{L^\infty_tC^m_x    }  \lesssim_{n,m} \tau_q^{-n},~~~n,m\ge 0. \notag
    \end{align}
    \item $\supp~\tilde\eta_i \cap \supp ~ \tilde\eta_j=\emptyset$ if $i\neq j$.
    \item  For all $t\in[0,T]$, we have
    $ \frac{1}{4} \le \sum_{i=0} \int_{\mathbb T^2}  \eta_i^2(x,t) \dd x \le 1.$
\end{enumerate}
Now we define the energy corrector $\de_{q+1}$ by
\begin{align*}
    \de_{q+1}=\sum_i{\tilde \eta_i(t,x)}e_{q}^{1/2} h_{\Lambda}(\lambda_{q+1}\bar{\Phi}_i),
\end{align*}
here $\bar{\Phi}_i$ is given by \eqref{barphi}.
From  \eqref{e:glu-B} and \eqref{time},
Using $\Div \vv_q=\Div\wtq=0$, we have
\begin{align*}
    { \Big|\frac{\dd}{\dd t}\int_{\mathbb T^2}  {E(t,x)} \dd x\Big|= \Big|\int_{\TTT^2}\matd{\vv_q } E (t,x)\dd x\Big|\lesssim \tau^{-1}_q\delta_{q+1}\ell^{\frac{\alpha}{2}}_q.}
\end{align*}
{This inequality combined with the fact that $(\wt_{q+1,i,n+1}, \dt_{q+1,i,n+1})|_{t=t_i} = 0 $ yields $|E(t)|\le\delta_{q+1}$}. Hence, we conclude that
$5M^2\delta_{q+1}\le e_q(t)\le 50M^2\delta_{q+1}$. %and ${\color{red}\|\de_{q+1}\|_0\le M\delta^{1/2}_{q+1}}$.

We define
\[w_{q+1}=\wpq+\wcq+\wtq,\quad d_{q+1}=\dpq+\dtq+\de_{q+1}.\]
By the definition of $\de_{q+1}$, we have
\begin{align*}
&\Big|e(t)-\int_{\TTT^2}(|v_{q+1}|^2+2|\theta_{q+1}|^2)\dd x-20\delta_{q+2}\Big|\\
=&\Big|\int_{\TTT^2}(|\wcq|^2+2\vv_q\cdot(\wpq+\wcq)+2(\wpq+\wcq)\cdot\wtq+2\wpq\cdot\wcq)\dd x\\
&+2\int_{\TTT^2}(2\bar{\theta}_q(\dpq+\de_{q+1})+2\dpq\dtq)\dd x+2\int_{\TTT^2}\sum_i{\tilde \eta^2_i(t,x)}e_{q}\mathbb{P}_{\neq 0}(h^2_{\Lambda}(\lambda_{q+1}\bar{\Phi}_i))\dd x\Big|\\
\le& 10\delta_{q+2}.
\end{align*}
Hence we show \eqref{v-theta-energy2} at $q+1$ level. Noting that $\de_q\wpq=\de_q\wcq=0$, we easily obtain \eqref{e:RR_q-C01} with $q$ replaced by $q+1$.

Now, Proposition \ref{Prope} can help us to obtain a weak solution $(v,\theta) \in C^{\beta}(\mathbb{T}^2)$ to the system~\eqref{e:B}, where the energy profile is defined as
\[
e(t) := \int_{\mathbb{T}^2} |v|^2 + 2|\theta|^2 \, \dd x.
\]
The  fact $w_{q+1}(0) = w_{q+1}(T) = 0$ implies that
\[
v(0) = u_1(0) = \phi(\bar{k}\cdot x)k \quad \text{and} \quad v(T) = u_1(T) = 0.
\]
Choosing ${e}(t)$ such that ${e}(0) - 1 \neq {e}(T)$, we derive
\begin{align*}
\int_{\mathbb{T}^2} |\theta(\cdot,0)|^2 \, \dd x
&= \frac{1}{2}\left( e(0) - \int_{\mathbb{T}^2} |v(\cdot,0)|^2 \, \dd x  \right) \\
&= \frac{1}{2}\left( e(0) - \int_{\mathbb{T}^2} |\phi|^2 \, \dd x \right) \\
&\neq \frac{1}{2}e(T) = \int_{\mathbb{T}^2} |\theta(\cdot,T)|^2 \, \dd x.
\end{align*}
This fact shows that the pair $u = (v,\theta)$ together with $b = (0,0,\theta)$ is a weak solution to the 3D ideal magnetohydrodynamics system \eqref{E} such that
 the total energy $\mathcal{E}(t)=e(t)$ is strictly decreasing and the cross helicity $\mathcal{H}_c(t) := \|\theta(t)\|_{L^2(\mathbb{T}^2)}^2$ fails to be conserved. Hence, we prove Theorem~\ref{t:main-energy1}.

\end{proof}

\section{Proofs of Theorem \ref{main00} and Theorem \ref{app}}\label{sec3}
\subsection{}We present a quantitative version of Theorem \ref{main00}.
\begin{prop}\label{t:main00}
Let $T>0$, $0\le \beta<\frac{1}{5}$ and $0<e(t), \widetilde{e}(t)\in C^{\infty}([0,T]; \R^{+}) $ with $e(t)=\widetilde{e}(t)$ for $t\in[0, \tfrac{T}{2}]$, then there exist two weak solutions $ (u, b) , (\widetilde{u}, \widetilde{b}) \in C^{\beta}([0,T]\times \R^2)\times {C^2}([0,T]\times \R^2) $ to the system  \eqref{E} such that $(v, b)|_{t=0}=(\widetilde{u}, \widetilde{b})|_{t=0}$,
\begin{align*}
  \int_{\mathbb R^3} |u|^2+|b|^2 \dd x=e(t) \quad \text{and}\quad\int_{\mathbb R^3} (|\widetilde{u }|^2+|\widetilde{b}|^2) \dd x=\widetilde{e}(t) ,~~\forall t\in[0,T].
\end{align*}
\end{prop}

By choosing different energy profile $e(t) $ in Proposition \ref{t:main00}, one easily shows Theorem~\ref{main00}. Hence, it suffices to prove Proposition  \ref{t:main00}. We shall prove it by Propositions \ref{p:main-prop00} and \ref{p:main-prop2}.

 \subsubsection{Iteration propositions}
 Before stating the iteration proposition,  we give all parameters needed  in the inductive procedure, which are different from the parameters in section 2.  Given $0<\beta<\tfrac{1}{5}$ and $T<\infty$, we define
\begin{align}\label{b-alpha}
1<b\leq {\min\big\{1+\tfrac{1-5\beta}{12\beta},  \tfrac{3}{2}\big\}},\quad N_0=\tfrac{50(b+1)}{b-1} ,\quad 0<\alpha<\min\Big\{\tfrac{2\beta b^2-(\tfrac12+\tfrac32 \beta)b+\tfrac12-\tfrac{\beta}{2}}{10},\tfrac{1}{100}\Big\}.
\end{align}
Let $a\in\mathbb{N}^+$ be a large number depending on $b,\beta,\alpha,T$. We define
\begin{align}\label{delta}
    \lambda_q  \coloneq  a^{b^q} , \quad   \delta_q &\coloneq c\lambda_q^{-2\beta},~~q\in\mathbb{N}^+,
\end{align}
and
\begin{align}\label{def-tauq}
\ell_q:={\lambda^{-\frac{3}{4}-8\alpha}_q\lambda^{-\frac{1}{4}}_{q+1} }
\big(\tfrac{\delta_{q+1}}{\delta_{q}}\big)^{\frac{3}{8}},\,\,
\tau_q:=(\lambda^{\frac{1}{2}}_q\lambda^{\frac{1}{2}}_{q+1}
\delta^{\frac{1}{4}}_q\delta^{\frac{1}{4}}_{q+1})^{-1}.
%\mu_q:=(\lambda^{\frac{1}{2}}_q\lambda^{\frac{1}{2}}_{q+1}\delta^{\frac{1}{4}}_q\delta^{-\frac{1}{4}}_{q+1} )^{-1}.
\end{align}
Without loss of generality, we set $c=1$,
 $e(t)$ and $\widetilde{e}(t)$ given in Theorem \ref{t:main00} satisfy
\begin{align}\label{e1}
 \frac{1}{3}\delta_2 \le e(t), \widetilde{e}(t)\le 3\delta_2 , \quad \forall t\in [0, T].
\end{align}
The above conditions ensure that
\begin{align}\label{beta} \tau_q\lambda_q\delta^{1/2}_{q} \delta_{q+1}=\lambda^{1/2}_q\lambda^{-1/2}_{q+1}\delta^{1/4}_{q}\delta^{3/4}_{q+1}\leq \delta_{q+2}\lambda_{q+1}^{-4 \alpha}\quad,\quad
\ell^{-N_0 }_q\lambda^{ -N_0 +10}_q   \leq \lambda^{ -10}_{q+1}.\end{align}
Let $\epsilon_0=\lambda^{-\frac{1}{20}}_1$, we define the periodic domains $\Omega_{q}$ and $\mathcal{O}$ by
\begin{align*}
\Omega_{q}&=\bigcup_{k\in\ZZ^2}\Big(\big([0,1]^2\setminus [\tfrac{1}{2}-\tfrac{3}{4}\epsilon_0-\lambda^{-\frac{1}{30}}_{q-1} ,  \tfrac{1}{2}+\tfrac{3}{4}\epsilon_0+\lambda^{-\frac{1}{30}}_{q-1} ]^2\big)+k\Big)\\
\mathcal{O}&=\bigcup_{k\in\ZZ^2} \Big(  [\tfrac{1}{2}-\tfrac{2}{4}\epsilon_0  ,  \tfrac{1}{2}+\tfrac{2}{4}\epsilon_0  ]^2 +k\Big),
\end{align*}
This fact implies $\Omega_{q}\cap\mathcal{O}=\emptyset$ for any $q\geq1$.
  %Now, let $C_0$ be a fixed large positive constant.
We consider  the so-called MHD-Reynolds system.
\begin{equation}
\left\{ \begin{alignedat}{-1}
&\partial_t  u_q+\Div ( u_q \otimes  u_q)+\Div (b_q \otimes b_q)+\nabla P_q = \Div \RR_q,\\
&\partial_t b_q  +\Div (   u_q \otimes b_q)    -\Div (  b_q \otimes  u_q)      = 0,
\end{alignedat}\right.\label{MHD-R}
\end{equation}
where $q\geq1$ and the \emph{Reynolds stress} $\RR_q$  is a symmetric trace-free $2\times2$  matrix.  Let
 $$ u_q= \vloc_q+\vnon_q \quad \text{with} \quad  \text{spt}_x \vloc_q\subset \Omega_q,\quad\text{spt}_x b_q\subset \mathcal{O}.$$
This equality shows that $\vloc_q\otimes b_q=0$ and the  system \eqref{MHD-R} is reduced to
\begin{equation}
\left\{ \begin{alignedat}{-1}
&\partial_t  u_q+\Div \big((\vloc_q+\vnon_q)\otimes (\vloc_q+\vnon_q)\big)+\Div (b_q \otimes b_q)+\nabla P_q = \Div \RR_q,\\
&\partial_t b_q  +\Div (  \vnon_q\otimes b_q)    -\Div (  b_q \otimes \vnon_q)      = 0,
\end{alignedat}\right.  \label{e:subsol-euler}
\end{equation}
Define $ D_{t,q}:=\partial_t+ u_q\cdot \nabla $.

Now we introduce the \textit{spatial-separation-driven} iterative scheme as follows.
\begin{prop}
\label{p:main-prop00}
Let $b$ and $\alpha$ satisfy \eqref{b-alpha} and $T>0 ~ $. Assume that $(u_1, b_1,P_1,\RR_1)$ is a smooth solution of the system \eqref{e:subsol-euler} on $[0,T]$ with $$\text{spt}_x b_1\subset \bigcup_{k\in\ZZ^2}\Big(  [\tfrac{1}{2}-\tfrac{1}{16}\epsilon_0  ,  \tfrac{1}{2}+\tfrac{1}{16}  ]^2 +k\Big)\subsetneqq  \mathcal{O}.$$
%There exist  $a_0$ such that for $a>a_0(\beta, b,\alpha)$, the following holds for $t\in [0,T]$.
For $q\geq1$, if $( u_q, b_q, P_q,\RR_q)$ obeys the equations \eqref{e:subsol-euler} with
\begin{align}
    &\| u_q\|_{0}+\|b_q\|_0\leq 1+\sum_{l=1}^q \delta^{1/2}_l,\label{e:vq-C0}
    \\
    &\|\vloc_q\|_{1}+\|\vnon_q\|_{1}+\|b_q\|_{1} \le   \lambda_q \delta^{1/2}_q  ,
    \label{e:vq-C1}
    \\
    & \|\vnon_q\|_{H^4}+\|b_q\|_{H^4} \le
    \epsilon^{\frac{3}{2}}_0+\sum_{l=1}^{q }\lambda^{-3}_l,
    \label{e:vq-H3}\\
&\|\vloc_q\|_{H^N}+\|\vnon_q\|_{H^N}+\|b_q\|_{H^N} \le   \lambda^{N }_q \delta^{1/2}_q,   \qquad \qquad\qquad\quad\quad\quad\,\, {N\ge 5 }\label{e:vq-CN}
    \\
    & \|  D_{t,q}\vloc_q\|_{N}+\|  D_{t,q}\vnon_q\|_{N}+\|  D_{t,q}b_q\|_{N} \le \lambda^{1+N}_q\delta_{q},\qquad\qquad \qquad   N=0,1,
    \label{e:Dv} \\
    &\| \RR_q \|_{N} \le \delta_{q+1}\lambda_q^{N-3 \alpha},  \qquad\qquad \qquad  \qquad \qquad\qquad\qquad\qquad\quad\,   N=0,1,2,
    \label{e:RR_q-C0}
    \\
    &\big\| D_{t,q}\RR_q \big\|_{N} \le \lambda^{1+N}_q\delta^{1/2}_{q}\delta_{q+1} ,\qquad\qquad\qquad \qquad\qquad\qquad\qquad\,\,  N=0,1,
    \label{e:DRR_q-C0} \\
    & \text{spt}_x (\vloc_q,\RR_q)\subset \Omega_q, ~~~~
      {\text{spt}_x b_q\subset   \mathcal{O}},\label{spt}\\
    &\frac{1}{3}\delta_{q+1}\le e(t) -\int_{\mathbb T^2} | u_q|^2+|b_q|^2 \dd x \le 3\delta_{q+1},\label{e:energy-q-estimate}
\end{align}
then there exists a solution $ ( u_{q+1} ,P_{q+1} ,b_{q+1} , \RR_{q+1})$ to the system \eqref{e:subsol-euler}  satisfying  \eqref{e:vq-C0}--\eqref{e:energy-q-estimate} with $q$ replaced by $q+1$, and
\begin{align}
       &\|\vloc_{q+1} - \vloc_q\|_{0}+ \|\vnon_{q+1} - \vnon_q\|_{0}\leq  \delta_{q+1}^{1/2},\label{e:velocity-diff}\\
       & \|b_{q+1} - b_q\|_{0}\leq  \delta_{q+1}^{1/2}, \quad\quad b_{q+1}(0,x)=b_q(0,x)*\psi_{\ell_q}.
\label{e:b-spt}
\end{align}
\end{prop}
The following proposition shows that if $e(t)=\tilde{e}(t)$ for $[0, T/4]$, then the limit functions of the approximate solutions constructed in Proposition \ref{p:main-prop00} share the same initial value.
\begin{prop}\label{p:main-prop2}Let $T>0$ and $e(t)=\widetilde{e}(t)$ for $t\in[0, \tfrac{T}{4}+8\tau_1]$. Suppose that  $(v_{q } ,P_{q } ,b_{q } , \RR_{q })$ solves
\eqref{e:subsol-euler} and satisfies  \eqref{e:vq-C0}--\eqref{e:energy-q-estimate} ,    $(\widetilde{v}_{q } ,\widetilde{P}_{q } ,\widetilde{b}_{q } , \tRR_{q })$ solves
\eqref{e:subsol-euler} and satisfies  \eqref{e:vq-C0}--\eqref{e:energy-q-estimate}  with $e(t)$ replaced by $\widetilde{e}(t)$. Then   if
\begin{align*}
 v_{q }=\widetilde{v}_{q },\quad b_{q }=\widetilde{b}_{q }  \quad \RR_{q } ={\tRR}_{q },\quad\text{on}\quad [0, \tfrac{T}{4}+5\tau_q],
\end{align*}
 we have
\begin{align*}
 v_{q +1}=\widetilde{v}_{q+1 },\quad b_{q +1}=\widetilde{b}_{q+1 }  \quad \RR_{q+1 } ={\tRR}_{q+1},\quad\text{on}\quad [0, \tfrac{T}{4}+5\tau_{q+1}].
\end{align*}
\end{prop}
%\textcolor{blue}{The rest of this section is devoted to the proof of  Proposition \ref{p:main-prop} and Proposition \ref{p:main-prop2}, which will be enough to yield Theorem  \label{t:main}. }

%\subsubsection{Proof of Proposition \ref{t:main00}}
\noindent \textbf{Propositions \ref{p:main-prop00} and \ref{p:main-prop2} imply Proposition \ref{t:main00}.}
We start the iteration with
 $$(u_1,b_1,P_1,\RR_1)=(\widetilde{u}_1,\widetilde{b}_1,\widetilde{P}_1,\tRR_1)
=( 0+g(x) ,g(x),0,0),$$
and
$$\vloc_1=\widetilde{u^{\mathcal{C}}_1}=0,\qquad\vnon_1=\widetilde{u^{\mathcal{P}}_1}=g(x),$$
where $g(x)\in C^{\infty}(\mathbb{T}^2)$ satisfies
$$\text{spt}_xg(x)\subset \bigcup_{k\in\ZZ^2}\Big(\big([0,1]^2\setminus [\tfrac{1}{2}-\tfrac{1}{16}\epsilon_0  ,  \tfrac{1}{2}+\tfrac{1}{16}\epsilon_0  ]^2\big)+k\Big),\quad \|g\|_{H^{N_0}}\leq \lambda^{-5}_{ 1}.$$

Let $e(t), \widetilde{e}(t)$ be two distinct  smooth functions satisfying \eqref{e1} on $[0,T]$, and $e (t)=\widetilde{e}(t),~t\in[0,\frac{T}{2}] $. It is easy to see that $(u_1,P_1,b_1,\RR_1)$ and $(\widetilde{u}_1,\widetilde{P}_1,\widetilde{b}_1,\tRR_1)$ satisfy \eqref{e:vq-C0}--\eqref{spt} for $q=1$, and $e(t),~\widetilde{e}(t)$ satisfy \eqref{e:energy-q-estimate} for $q=1$.

Then, making use of Proposition \ref{p:main-prop}, we obtain  $C^0_{t,x}$  convergent sequence of functions such that
$$ (u^{\mathcal{C}}_q , u^{\mathcal{P}}_q ,P_q,b_q,\RR_q) \to (\vloc,\vnon , P ,b ,0 ) ,~(\widetilde{u^{\mathcal{C}}_q} ,\widetilde{u^{\mathcal{P}}_q} , \widetilde{P}_q,\widetilde{b}_q,\widetilde{\RR}_q) \to (\widetilde{{\vloc}} ,\widetilde{{\vnon}} , \widetilde{P} ,\widetilde{b} ,0 ).  $$
Moreover, we have $\text{spt}_x\vloc  \cap \text{spt}_x b=\text{spt}_x\widetilde{\vloc}  \cap \text{spt}_x \widetilde{b}=\emptyset$ and
\begin{alignat}{8}
     \int_{\mathbb{T}^2}|u|^2+|b|^2 = e (t),\quad \int_{\mathbb{T}^2}|\widetilde{u}|^2+|\widetilde{b}|^2 = \widetilde{e} (t),\quad t\in[0,T].\end{alignat}
     Letting $u=\vloc+\vnon$ and $\widetilde{u}=\widetilde{{\vloc}}+\widetilde{{\vnon}}$, one can deduce  that $(u, b),~(\widetilde{u},\widetilde{b})\in C^{\beta'}_{t,x}$ for all $\beta'<\beta<\frac{1}{5}$ as in \cite{BDSV} and solve \eqref{E}.

Finally,  Proposition \ref{p:main-prop2} ensures  that $(u ,b)\big|_{t=0}=(\widetilde{u}, \widetilde{b})\big|_{t=0}$, and so we prove that the non-uniqueness for weak solutions  by choosing two  smooth  non-increasing energy profiles $e(t),\widetilde{e}(t)$ such that
$$e (t)\neq \widetilde{e}(t),~t\geq\tfrac{T}{2}. $$ Thus, we complete the proof of Proposition \ref{t:main00}, which immediately implies Theorem~\ref{main00}. We now direct our focus to the proofs of Proposition  \ref{p:main-prop00} and  Proposition \ref{p:main-prop2}.

\subsubsection{Proof of Proposition \ref{p:main-prop00}}We prove Proposition \ref{p:main-prop00} in three steps: mollification, constructing perturbations, and estimating errors.\\
\noindent{\textbf{Mollification}}\,\, Let  mollifiers $\varphi_{\epsilon}(t)$ and $\psi_{\epsilon}(x)$ be given by \eqref{Moll},
 and $( u_{\ell_q},P_{\ell_q},b_{\ell_q}, \RR_{\ell_q} )$ be defined by
\begin{align*}
   &  u_{\ell_q} \coloneq  u_q * \psi_{\ell_q} ,~~\vloc_{\ell_q} \coloneq \vloc_q * \psi_{\ell_q},~~\vnon_{\ell_q} \coloneq \vnon_q * \psi_{\ell_q},\\
   &P_{\ell_q} \coloneq P_q * \psi_{\ell_q},~~b_{\ell_q} \coloneq b_q * \psi_{\ell_q} ,~~\RR_{\ell_q} \coloneq \RR_q * \psi_{\ell_q},
\end{align*}
where $\ell_q$ is introduced in \eqref{def-tauq}. We denote
\[D_{t, \ell_q}=\partial_t+ u_{\ell_q}\cdot\nabla, \quad  \text{and}\quad  t_i:=i\tau_q,~~0\leq i\leq  \lceil\tfrac{T}{\tau_q}\rceil.\]
%Without loss of generality, we set $T=1$.

Let $\Phi(s,x;t)$ be the backward flow map with the drift velocity $ u_{\ell_q}$ defined on  $[t ,t+2\tau_q]$ starting at the initial time $s=t$:
\begin{equation}
\left\{ \begin{alignedat}{-1}
&D_{s, \ell_q}\Phi = 0,
\\
  &  \Phi |_{s=t} =  x .
\end{alignedat}\right.
\end{equation}
Particularly, we denote $\Phi (s,x;t_i)$ by $\Phi_i(s,x)$. Then $\Phi_i(s,x)$ satisfies that
 \begin{equation}
\left\{ \begin{alignedat}{-1}
&\partial_s \Phi_i + u_{\ell_q}\cdot\nabla\Phi_i= 0,
\\
  &  \Phi_i (t_i ,x ) =  x.
\end{alignedat}\right.
 \label{flow}
\end{equation}
From the definition of $\Phi$, one deduces that the forward flow  $\Phi^{-1}$ satisfies
\begin{equation}
\left\{ \begin{alignedat}{-1}
&\partial_s \Phi^{-1} (s,x; t)   =  u_{\ell_q}(s,\Phi^{-1} (s,x;t)),
\\
  &  \Phi^{-1}|_{s=t} =  x .
\end{alignedat}\right.
\end{equation}
We give the mollification  as introduced in \cite{GR,Ise17},
\begin{align*}
    \RR_{\ell_q,\tau_q}(t,x):=(\varphi_{\tau_q} *_{\Phi^{-1} } \RR_{\ell_q})(t,x)=
\int_{\mathbb{R}}\RR_{\ell_q}(t+s,\Phi^{-1}(t+s,x;t))
\varphi_{\tau_q}(s)\dd s.
\end{align*}
Then we mollify $(u_q,b_q)$, the resulting pair solves the following system:
\begin{equation}
\left\{ \begin{alignedat}{-1}
&\partial_t  u_{\ell_q}+\Div \big( u_{\ell_q}\otimes  u_{\ell_q}\big)+\Div (b_{\ell_q} \otimes b_{\ell_q})+\nabla P_q = \Div \RR_{\ell_q}+\Div\RR_{\textup{mol,v}},\\
&\partial_t b_{\ell_q}  +\Div (  \vnon_{\ell_q}\otimes b_{\ell_q})    -\Div (  b_{\ell_q} \otimes \vnon_{\ell_q})      = \Div\RR_{\textup{mol,b}},
\end{alignedat}\right.  \label{modify-e:subsol-euler}
\end{equation}
where  $$\RR_{\textup{mol,v}}=\big( u_{\ell_q}\otimes  u_{\ell_q}-( u_q \otimes u_q)_{\ell_q}\big)-\big(b_{\ell_q}\otimes b_{\ell_q}-(b_q \otimes b_q)_{\ell_q}\big),$$ $$\RR_{\textup{mol,b}}=\big(\vnon_{\ell_q}\otimes b_{\ell_q}-(\vnon_{ q}\otimes b_{ q})_{\ell_q}\big)-\big(b_{\ell_q} \otimes \vnon_{\ell_q}-(b_{ q} \otimes \vnon_{ q})_{\ell_q}\big) .$$ By the standard method as in \cite{GR}  ,  we have the following two propositions.
\begin{prop}[ Estimates for mollified functions]\label{p:estimates-for-mollified}
Let $q\geq1$, then
\begin{align}
&\|\vloc_{\ell_q }\|_{N+1}+\|\vnon_{\ell_q }\|_{N+1}+\|b_{\ell_q }\|_{N+1}  \lesssim \delta_{q}^{1 / 2} \lambda_{q} \ell_q^{-N}, && \forall N \geq 0, \label{e:v_ell-CN+1}
\\
&\|  D_{t,\ell_q}\vloc_{\ell_q}\|_{N}+\|  D_{t,q}\vnon_{\ell_q}\|_{N}+\|  D_{t,q}b_{\ell_q}\|_{N} \le \lambda^{1+N}_q\delta_{q} , && N=0,1\\
&\|\RR_{\ell_q }\|_{N+\alpha} \lesssim  \delta_{q+1} \ell_q^{-N+\frac{5}{4}\alpha}, && \forall N \geq 0,\label{e:R_ell}\\
&\|D_{t, \ell_q}\RR_{\ell_q }\|_{N } \lesssim  \lambda^{1}_q\delta^{1/2}_{q}\delta_{q+1} \ell_q^{-N+\frac{5}{4}\alpha}, && \forall N=0,1,\label{e:R_ell}\\
&\|D^n_{t, \ell_q}\RR_{\ell_q,\tau_q}\|_{N}\leq \ell_q^{-N}\tau^{-n}_q\delta_{q+1}, && \forall N,n \geq 0, \label{e:v_ell-CN+1}
\end{align}
Moreover, ~we~have :%for $\tau\in [t_i ,t_{i+2}]$:
\begin{align}
&\| \vloc_{\ell_q}-\vloc_{q} \|_{H^5}+\| \vnon_{\ell_q}-\vnon_{q} \|_{H^5}+\| b_{\ell_q}-b_{q} \|_{H^5}
 \lesssim   \ell^{N_0-5}_q\lambda^{ N_0+1}_q \delta_{q}^{1 / 2} \leq \lambda^{ -10}_{q+1},    \label{e:v_ell-vq}
\\
%&\| D_{t, \ell_q }(m_{\ell_q}-m_{q}) \|_{N-1} \lesssim \ell^{2-N}\lambda^{2 }_q\delta_{q}   ,&& \forall N =1,2 \label{e:Dv_ell-vq}
%\\
&\| \RR_{\ell_q}-\RR_{\ell_q,\tau_q} \|_{N} \lesssim \lambda^{1/2}_q\lambda^{-1/2+N}_{q+1}\delta^{1/4}_{q}\delta^{3/4}_{q+1},~~~~~ \forall N =0,1,2 .\label{e:R_ell-vq}%\\
%&\frac{1}{5}\delta_{q+1}\le e(t) -\int_{\mathbb T^3} | u_q|^2+|b_q|^2 \dd x \le 5\delta_{q+1}.\label{energy:v_ell}
%\\
%&\|D_{t, \ell_q }(\RR_{\ell_q}-\RR_{\ell_q,\tau_q} )\|_{N-1} \lesssim
%\lambda^{1/2}_q\lambda^{-1/2+N}_{q+1}\delta^{1/4}_{q}\delta^{5/4}_{q+1},&& \forall N =1,2.\label{e:DR_ell-vq}
\end{align}
\end{prop}
\begin{prop}[Estimates for $\Phi_i$]\label{p:estimates-for-inverse-flow-map} For  $n=0,1,2,~  N\in \mathbb{N}$ and $\forall s\in  [t_i,t_{i+2}]$, we have
\begin{align}
& |\Phi_i-x|\lesssim \tau_q, \label{Phi-x}
 \\
&\|(\nabla\Phi_i)^{\pm1}-{\rm Id}\|_N \lesssim \tau_q\lambda_q\delta_q^{1/2} \ell_q^{-N}.\\
 &\| (\nabla \Phi_i)^{-1}\|_N+ \|  \nabla \Phi_i\|_N \le \ell_q^{-N}, \label{e:nabla-phi-i-CN}
\\
&\|D^n_{t,\ell_q} (\nabla \Phi_i)^{-1}\|_N+\|D^n_{t,\ell_q} \nabla \Phi_i\|_N \lesssim (\lambda_q\delta_q^{1/2})^n \ell_q^{-N}, \label{e:nabla-phi-i-matd}\\
&{\|D^n_{t,\ell_q} ({\rm {det}}(\nabla \Phi_i))^{-1}\|_N} \lesssim (\lambda_q\delta_q^{1/2})^n \ell_q^{-N}. \label{e:nabla-phi-i-matd2}
\end{align}
\end{prop}
With the aid of \eqref{e:v_ell-vq}, one easily deduces that
\begin{align}
\|\RR_{\textup{mol,v}}\|_{H^5}+\|\RR_{\textup{mol,b}}\|_{H^5}&\leq   \lambda^{ -10}_{q+1},\label{est-Rell} \\
\|\RR_{\textup{mol,v}}\|_{H^N}+\|\RR_{\textup{mol,b}}\|_{H^N}&\leq   \ell^{- N  }_q\lambda^{ -10}_{q+1},\,\,\forall N\ge 6.\label{est-RellN}
\end{align}
Moreover, we deduce by $ \ell_q\ll\tau_q\ll\lambda^{-\frac{1}{30}}_{q}$ that $$\text{spt}_x \RR_{\ell_q},~~\text{spt}_x \RR_{\ell_q,\tau_q}\subset
 \bigcup_{k\in\ZZ^2}\Big(\big([0,1]^2\setminus [\tfrac{1}{2}-\tfrac{3}{4}\epsilon_0-10\lambda^{-\frac{1}{30}}_{q} ,  \tfrac{1}{2}+\tfrac{3}{4}\epsilon_0+10\lambda^{-\frac{1}{30}}_{q} ]^2\big)+k\Big)
.$$
%$$ \text{spt}_x b_{\ell_q}\subset \mathbb{T}^2/[-1-5\lambda^{-\alpha}_{q} , 1+ 5\lambda^{-\alpha}_{q}]~~~and~~~\text{spt}_x \RR_{\textup{mol,b}}\subset \mathbb{T}^2/[-1-5\lambda^{-\alpha}_{q} , 1+ 5\lambda^{-\alpha}_{q}],$$
%, where .

\noindent{\textbf{Perturbation}}\,\,
We  give nonnegative smooth  temporal cutoff $\{\eta_i(t)\}_{i\in\ZZ}$ such that
\begin{align*}
& \supp \,\,\eta_i(t)\subset[i\tau_q-\tfrac{\tau_q}{6}, (i+1)\tau_q ], \quad \eta_i(t)|_{[i\tau_q,(i+1)\tau_q-\tfrac{\tau_q}{6}]}\equiv 1,\quad\|\tfrac{\dd^N}{\dd t^N}\eta_i\|_0\lesssim \tau^{-N}_q;
\end{align*}
and the   smooth  spatial cutoff
$ \chi_q(x) $ satisfying $\|\tfrac{\dd^N}{\dd x^N}\chi_q\|_0\lesssim \lambda^{\frac{N}{30}}_q$
\begin{align*}
&\supp \,\,\chi_q(x)\subset\bigcup_{k\in\ZZ^2}\Big(\big([0,1]^2\setminus [\tfrac{1}{2}-\tfrac{3}{4}\epsilon_0-6\lambda^{-\frac{1}{30}}_{q } ,  \tfrac{1}{2}+\tfrac{3}{4}\epsilon_0+6\lambda^{-\frac{1}{30}}_{q } ]^2\big)+k\Big),\\
&\chi_q(x)=1, \quad \text{if} \,\,x\in {\bigcup_{k\in\ZZ^2}\Big(\big([0,1]^2\setminus [\tfrac{1}{2}-\tfrac{3}{4}\epsilon_0-8\lambda^{-\frac{1}{30}}_{q } ,  \tfrac{1}{2}+\tfrac{3}{4}\epsilon_0+8\lambda^{-\frac{1}{30}}_{q } ]^2\big)+k\Big)}.
\end{align*}
This fact implies that $\chi_q\RR_{\ell_q,\tau_q}=\RR_{\ell_q,\tau_q}.$
%where $\mu_q$ is defined by \eqref{def-tauq} and $\chi_0$ is a non-negative smooth function satisfying
%\[{\rm spt} \,\,\chi_0\subset B_{1}(0) , \quad \chi_0|_{B_{1/4}(0)}\equiv 1.\]

The energy gap $e_q$ is defined as follows:
\begin{align}
e_q(t) \coloneq   \frac{1}{3\int_{\mathbb{T}^2}\chi^2_q\dd x}\Big(e(t) - \int_{\mathbb T^2}
   | u_q|^2+|b_q|^2\dd x   - \frac{\delta_{q+2}}{2} \Big). \label{e0:energy-gap}
\end{align}
It follows from \eqref{e:energy-q-estimate} that $e_q(t)$ is strictly positive in $[0,T]$ and satisfy
\begin{align}\label{eq}
\frac{1}{8}\delta_{q+1}\leq 3 e_q(t) \leq 8\delta_{q+1}.
\end{align}
Let $\Lambda=\Lambda_1 \cup \Lambda_2$ be given in Lemma \ref{first S}. We set
 $$\Lambda_i=\Lambda_1,\quad i ~\textup{is ~odd};\qquad\Lambda_i=\Lambda_2,\quad i ~~\textup{is ~even}.$$
Now, we construct \textit{the principle perturbation} $ \wpq$ as follows~$(\bar{k}:=k^{\perp})$:
 \begin{align}
 \wpq&:=\sum_{i;k\in\Lambda_i} \ii \eta_i\chi_qe^{1/2}_q a_k
 \big({\rm{Id}}-\tfrac{\RR_{\ell_q,\tau_q}}{e_q}\big)
 e^{2\pi\ii \lambda_{q+1}k\cdot\Phi_i}\bar{k}\notag =:\sum_{i;k\in\Lambda_i} \ii a_{k,i}e^{2\pi\ii \lambda_{q+1}k\cdot\Phi_i}\bar{k}.\label{def-a}
 \end{align}
We rewrite it as
 \begin{align}
 \wpq&=\sum_{i;k\in\Lambda_i}(\ii a_{k,i}e^{2\pi \ii \lambda_{q+1}k\cdot(\Phi_i-x)})e^{2\pi i\lambda_{q+1}k\cdot x}\bar{k}\\
  &=\frac{1}{ \lambda_{q+1}  }\sum_{i;k\in\Lambda_i}(   a_{k,i}e^{2\pi \ii \lambda_{q+1}k\cdot (\Phi_i-x)}) \nabla^{\perp} e^{ 2\pi \ii\lambda_{q+1}k\cdot x}\notag\\
 &=\frac{1}{ \lambda_{q+1}  }\sum_{i;k\in\Lambda_i}  \bar{a}_{k,i}   \nabla^{\perp} e^{ 2\pi \ii\lambda_{q+1}kx}= \sum_{i;k\in\Lambda_i} \ii \bar{a}_{k,i}     e^{ 2\pi \ii\lambda_{q+1}k\cdot x}\bar{k}.
 \end{align}
Now we introduce \textit{the incompressibility corrector} $\wcq$ by
\begin{equation}\label{wc}
 \begin{aligned}
 \wcq= \sum_{i;k\in\Lambda_i} \frac{1}{ \lambda_{q+1}  } \nabla^{\perp} \bar{a}_{k,i} e^{ 2\pi \ii\lambda_{q+1}k\cdot x} .
 \end{aligned}
 \end{equation}
One can rewrite $\wcq$ as
  \begin{align}\label{rewrite-wcq}
 \wcq=& \sum_{i;k\in\Lambda_i} \frac{1}{ \lambda_{q+1}  } \nabla^{\perp} \big( - a_{k,i}e^{2\pi \ii \lambda_{q+1}k\cdot(\Phi_i-x)}\big) e^{ 2\pi \ii\lambda_{q+1}k\cdot x} \notag\\
  =&\sum_{i;k\in\Lambda_i} \Big[-\frac{1}{ \lambda_{q+1} }\nabla^{\perp} ( {a}_{k,i}) -\frac{1}{ \lambda_{q+1}  } {a}_{k,i} \nabla^{\perp} ( \ii\lambda_{q+1}k\cdot(\Phi_i -x)) \Big]e^{ 2\pi \ii\lambda_{q+1}k\cdot \Phi_i }.
 \end{align}
 We define $\wloc_{q+1}$ by
\[\wloc_{q+1}:=\wpq+\wcq=\sum_{i;k\in\Lambda_i} \frac{1}{ \lambda_{q+1} } \nabla^{\perp}\big( \bar{a}_{k,i} e^{ 2\pi \ii\lambda_{q+1}k\cdot x}\big) .\]
Finally, we define the MHD perturbation flow  $(\theta_{q+1},d_{q+1})$  which solves
 \begin{equation}
\left\{ \begin{alignedat}{-1}
&\partial_t \theta_{q+1} + \theta_{q+1}\cdot\nabla \theta_{q+1} + \theta_{q+1}\cdot\nabla \vnon_{\ell_q}+\vnon_{\ell_q}\cdot\nabla \theta_{q+1}+\nabla P^{(\textup{per})}_{q+1} \\
   &\qquad\qquad\qquad=   d_{q+1}\cdot\nabla d_{q+1}+b_{\ell_q}\cdot\nabla d_{q+1} + d_{q+1}\cdot\nabla b_{\ell_q}+ \Div\RR_{\textup{mol,v}}+F_{q+1},\\
&\partial_t d_{q+1} + \theta_{q+1}\cdot\nabla d_{q+1} + \vnon_{\ell_q}\nabla d_{q+1}+\theta_{q+1}\cdot\nabla b_{\ell_q})\\
      &\qquad\qquad\qquad=d_{q+1}\cdot\nabla \theta_{q+1}+b_{\ell_q}\cdot\nabla \theta_{q+1}+d_{q+1}\cdot\nabla \vnon_{\ell_q} + \Div\RR_{\textup{mol,b}} ,
\\
 & \Div \theta_{q+1} =\Div d_{q+1} = 0,
  \\
  & (\theta_{q+1},~d_{q+1}) |_{t=0}=  0 ,
\end{alignedat}\right.
 \label{e:wt}
\end{equation}
 where the forcing term $F_{q+1}$ derives from from the terms $\Div(\wpq\otimes \wpq)$, $D_{t,\ell_q} (\wpq+\wcq)$ and $(\wpq+\wcq)\nabla \vloc_{\ell_q}$ , which will be given in \eqref{est-Fq+1}.

Now, we set $ \wnon_{q+1}:=\theta_{q+1}$ and define $w_{q+1}$ by
\[w_{q+1}=\wloc_{q+1}+\wnon_{q+1}.\]
%Thanks to $\curl(fW)=\nabla f\times W+f\curl W$, one immediately infers that $\Div w_{q+1}=0$.

In order to estimate the perturbation $\wloc_{q+1}$, we firstly establish the derivative estimates of the coefficient functions $a_{k,i}$ as below.
\begin{prop}\label{est-a}
Let $0\leq M\leq 2N_0$ and $N=0,1,2$ and $a_{k,i}$ be defined in \eqref{def-a}, we have
\begin{align}\label{key0}
\|D^N_{t,\ell_q}a_{k,i}\|_M\lesssim
\ell_q^{-M}\tau^{-N}_q\delta^{1/2}_{q+1}.
\end{align}
\end{prop}
\begin{proof}
Using \eqref{eq} and the definition of $a_{k,i}$ in \eqref{def-a}, one obtains
\begin{align*}
\|a_{k,i}\|_M
\lesssim&
\delta^{1/2}_{q+1}\| \chi_{q} \|_M   \|a_k (  {\rm {Id}}-\tfrac{\RR_{\ell_q,\tau_q}}{e_q})   \|_0 +
\delta^{1/2}_{q+1}\|\chi_{q}\|_0   \|a_k (  {\rm {Id}}-\tfrac{\RR_{\ell_q,\tau_q}}{e_q})\|_M\\
\lesssim&\delta^{1/2}_{q+1}\lambda^{ M}_q+\delta^{1/2}_{q+1}\ell^{-M}_q\lesssim \delta^{1/2}_{q+1}\ell^{-M}_q,
\end{align*}
where we have used Propositions \ref{p:estimates-for-mollified} and  \ref{p:estimates-for-inverse-flow-map}.

With the aid of the equations \eqref{e:subsol-euler} and \eqref{e0:energy-gap}, one gets
\begin{align*}
e'_q(t)=\frac{1}{3}e'(t)
-\frac{1}{3}\int_{\mathbb{T}^2}\Div \RR_q\cdot  u_q\dd x=\frac{1}{3}e'(t)
+\frac{1}{3}\int_{\mathbb{T}^2} \RR_q:\nabla  u_q\dd x.
\end{align*}
This equality yields that
\begin{align*}
e''_q(t)=&\frac{1}{3}e''(t)
+\frac{1}{3}\frac{\dd}{\dd t}\int_{\mathbb{T}^2}  \RR_q:\nabla  u_q\dd x\\
=&\frac{1}{3}e''(t)
+\frac{1}{3}\int_{\mathbb{T}^2} D_{t,q}\RR_q:\nabla  u_q-( u_q\cdot\nabla\RR_q):\nabla  u_q\dd x\\
&+\frac{1}{3}\int_{\mathbb{T}^2} \RR_q:D_{t,q} (\nabla  u_q )-\RR_q:\big( u_q\cdot\nabla(\nabla  u_q)\big)\dd x\\
=&\frac{1}{3}e''(t)
+\frac{1}{3}\int_{\mathbb{T}^2} D_{t,q}\RR_q:\nabla  u_q -( u_q\cdot\nabla\RR_q):\nabla  u_q\dd x\\
&+\frac{1}{3}\int_{\mathbb{T}^2} \RR_q:D_{t,q} \nabla  u_q
+\big( u_q\cdot\nabla\RR_q\big):\nabla  u_q
\\
=&\frac{1}{3}e''(t)
+\frac{1}{3}\int_{\mathbb{T}^2} D_{t,q}\RR_q:\nabla  u_q +\RR_q:D_{t,q} \nabla  u_q \dd x.
 \end{align*}
 Note that $$\|D_{t,q}(\nabla  u_q)\|_0\leq\|\nabla D_{t,q}   u_q \|_0+\|\nabla  u_q\cdot\nabla   u_q \|_0\leq  2\lambda^2_q\delta_q,
 $$
one easily infers that
 \begin{align}
 \| e'_q(t)\|_0
\lesssim&   1+ \|  \RR_q :\nabla  u_q\|_0
\lesssim \lambda_q\delta_{q+1}\delta^{1/2}_q ,\label{rho1}\\
\|e''_q(t)\|_0
\lesssim&   1+ \|D_{t,q}\RR_q: \nabla  u_q \|_0
+ \|\RR_q:D_{t,q}(\nabla  u_q) \|_0
\lesssim    \lambda^2_q\delta_q\delta_{q+1} \label{rho2}.
  \end{align}
For simplicity, let $a_k:=a_k ({\rm Id}-\frac{\RR_{\ell_{t,x}}}{e_q})   $. Using the fact  $D_{t,\ell_q}\Phi_i=0$, we infer from \eqref{rho1}, \eqref{rho2}, Propositions \ref{p:estimates-for-mollified} and \ref{p:estimates-for-inverse-flow-map}  that
\begin{align*}
\|D_{t,\ell_q}a_{k,i}\|_M \lesssim&\delta^{1/2}_{q+1}\|a_kD_{t,\ell_q}( \eta_i\chi_{q} )   \|_M
+\|(\tfrac{\dd}{\dd t}e^{1/2}_q){\rho_0}^{1/2}\eta_i\chi_{u}a_k  \|_M+\delta^{1/2}_{q+1} \|(D_{t,\ell_q}a_k) \eta_i\chi_{q}\|_M \\
\lesssim&
\ell^{-M}_q\delta^{1/2}_{q+1}\tau^{-1}_q,
\end{align*}
and
\begin{align*}
\|D^2_{t,\ell_q}a_{k,i}\|_M \lesssim& \delta^{1/2}_{q+1}\|a_kD^2_{t,\ell_q}( \eta_i\chi_{q} )  \|_M
+\delta^{1/2}_{q+1}\| D_{t,\ell_q}( \eta_i\chi_{q} )\cdot   D_{t,\ell_q}a_k \|_M\\
&+\|(\tfrac{\dd^2}{\dd t^2}e^{1/2}_q )\cdot  \eta_i\chi_{q}a_k  \|_M+\|(\tfrac{\dd}{\dd t}e^{1/2}_q ) D_{t,\ell_q}( \eta_i\chi_{q}a_k ) \|_M\\
&+\delta^{1/2}_{q+1} \|(D^2_{t,\ell_q}a_k) \eta_i\chi_{q}   \|_M +\delta^{1/2}_{q+1} \|(D_{t,\ell_q}a_k) D_{t,\ell_q}( \eta_i\chi_{q} )  \|_M\\
\lesssim&
\ell^{-M}_q\delta^{1/2}_{q+1}\tau^{-2}_q.
\end{align*}
Therefore we complete the proof of Proposition \ref{est-a}.
\end{proof}
Based on the estimates for the  coefficients $a_{k,i}$ and $\text{spt}_xa_{k,i}\subset \Omega_q$, one can bound the perturbation $w_{q+1}$ as follows.
\begin{prop}[Estimates for $w_{q+1}$]\label{estimate-wq+1}We have
\begin{align}
&\|\wpq\|_0+\tfrac{1}{\lambda^N_{q+1}}\|\wpq\|_N\le \delta^{1/2}_{q+1},\label{estimate-wp} \\
&\|\wcq\|_0+\tfrac{1}{\lambda^N_{q+1}}\|\wcq\|_N\le\lambda^{-1}_{q+1}\delta^{1/2}_{q+1}\ell^{-1}_q, \label{estimate-wc}\\
%&\|\wloc_{q+1}\|_0+\tfrac{1}{\lambda^N_{q+1}}\|\wloc_{q+1}\|_N\le \delta^{1/2}_{q+1}, \\
&\|D_{t,\ell_q}\wpq\|_0+\tfrac{1}{\lambda_{q+1}}\|D_{t,\ell_q}\wpq\|_1\le \tau^{-1}_q\delta^{1/2}_{q+1}, \\
&\|D_{t,\ell_q}\wcq\|_0+\tfrac{1}{\lambda_{q+1}}\|D_{t,\ell_q}\wcq\|_1\le (\ell_q\lambda_{q+1} )^{-1}\tau^{-1}_q\delta^{1/2}_{q+1}.\label{estimate-Dt-wc}
\end{align}
%Moreover, by interpolation, we obtain from the above estimates that
%\begin{equation}\label{estimate-alpha}
%\|w^{(p)}_{q+1}\|_{\alpha}\le \lambda^{\alpha}_{q+1}\delta^{1/2}_{q+1},\quad\|w^{(c)}_{q+1}\|_{\alpha}\le \lambda^{-1+\alpha}_{q+1}\ell^{-1}_q\delta^{1/2}_{q+1}.
%\end{equation}
\end{prop}
\begin{proof}
By Proposition \ref{est-a}, one can obtain \eqref{estimate-wp} and \eqref{estimate-wc} directly. Moreover, we have
\begin{align*}
 \|D_{t,\ell_q}\wpq\|_0
\le \| D_{t,\ell_q}{a}_{k,i}  e^{2\pi\ii\lambda_{q+1}k\cdot \Phi_i } \|_0
    \le  \tau^{-1}_q\delta^{1/2}_{q+1},
\end{align*}
and
\begin{align*}
  \|\partial_jD_{t,\ell_q}\wpq\|_0
\le \|\partial_j(D_{t,\ell_q}{a}_{k,i})   e^{ 2\pi\ii\lambda_{q+1}k\cdot{\Phi}_i}\|_0
+\|(D_{t,\ell_q}{a}_{k,i}) \partial_je^{ 2\pi\ii\lambda_{q+1}k\cdot{\Phi}_i}\|_0
\le   \lambda_{q+1}\tau^{-1}_q \delta^{1/2}_{q+1}.
\end{align*}
%We also have
%$$\|D_{t,\ell_q}\partial_jw^{(p)}_{q+1}\|_0\le
%\|\partial_jD_{t,\ell_q}w^{(p)}_{q+1}\|_0+\|\partial_j  u_{\ell_q}\nabla w^{(p)}_{q+1}\|_0\le\lambda_{q+1}\tau^{-1}_q\delta^{1/2}_{q+1}.$$
A similar calculation yields \eqref{estimate-Dt-wc} for $\wcq$.
\end{proof}
\noindent{\textbf{Reynolds stress}}\,\,
  Let $ u_{q+1}= u_{\ell_q}+w_{q+1},~b_{q+1}=b_{\ell_q}+d_{q+1} $, and
 $$ {T^{\text{rem}}_{q+1}}:=\Div \Big(\vloc_{\ell_q}\otimes \wnon_{q+1}+\wnon_{q+1}\otimes \vloc_{\ell_q}+\wloc_{q+1}\otimes \wnon_{q+1}+\wnon_{q+1}\otimes\wloc_{q+1}+\RR_{\ell_q} -\RR_{\ell_q,\tau_q}  \Big),$$
then we obtain that
\begin{align*}
&\partial_t  u_{q+1} +\Div \big( u_{q+1}\otimes  u_{q+1}-b_{q+1}\otimes b_{q+1}\big)  +\nabla P_{\ell_q}\\
=&\underbrace{ \partial_t \wloc_{q+1} +  u_{\ell_q}\cdot\nabla \wloc_{q+1}}_{T^{\textup{tran}}_{q+1}}+ \underbrace{\wloc_{q+1} \cdot \nabla u_{\ell_q}}_{T^{\textup{Nash}}_{q+1}}+
\underbrace{\Div \Big(\wloc_{q+1}\otimes \wloc_{q+1}+ \RR_{\ell_q,\tau_q}\Big)}_{T^{\textup{osc}}_{q+1}}+\Trem+\Div\RR_{\textup{mol,v}}\\
 & +  \partial_t \wnon_{q+1} + (\wnon_{q+1} +\vnon_{\ell_q})\cdot\nabla \wnon_{q+1}+ \wnon_{q+1}\cdot\nabla\vnon_{\ell_q}- (d_{q+1} +b_{\ell_q})\cdot\nabla d_{q+1}- d_{q+1}\cdot\nabla b_{\ell_q}  ,
\end{align*}
and
\begin{align*}
 &  \partial_t b_{q+1} +\Div \Big( \vnon_{q+1}\otimes b_{q+1}- b_{q+1}\otimes \vnon_{q+1}  \Big)   \\    =&\partial_t d_{q+1}     + (\wnon_{q+1} +\vnon_{\ell_q})\cdot\nabla d_{q+1} + \wnon_{q+1}\cdot\nabla b_{\ell_q}\\
 &- (d_{q+1} +b_{\ell_q})\cdot\nabla \wnon_{q+1}- d_{q+1}\cdot\nabla \vnon_{\ell_q}+\Div\RR_{\textup{mol,b}}  .
\end{align*}
%Let $\mathcal{R}$ be the inverse divergence operator in Appendix \ref{lemR}. We rewrite it as follows.

\begin{prop}[Transport error]\label{proptrans}
$T^{\textup{tran}}_{q+1}$ has the following decomposition
 \begin{align*}
 T^{\textup{tran}}_{q+1}
 =& \textup{div}\sum_{m,i;k\in\Lambda_i} \lambda^{-1}_{q+1}A^{(m)}_{\textup{tran}}
      (\mathcal{L}_{\alpha_m}e^{2\pi ik\cdot })\circ(\lambda_{q+1}\Phi_i   )\\
& +\nabla \sum_{m,i;k\in\Lambda_i} \lambda^{-1}_{q+1}B^{(m)}_{\textup{tran}}    (\mathcal{L}_{\alpha_m}e^{2\pi ik\cdot })\circ(\lambda_{q+1}\Phi_i   )
  \\
  &+\sum_{i;k\in\Lambda_i} \lambda^{-N_0}_{q+1}C^{(N_0)}_{\textup{tran}}    (\mathcal{L}_{\alpha_m}e^{2\pi ik\cdot })\circ(\lambda_{q+1}\Phi_i   )\\
 :=& \textup{div} \RR^{\textup{tran}}_{q+1}+\nabla P^{\textup{tran}}_{q+1}+F^{\textup{tran}}_{q+1},
 \end{align*}
 where $\mathcal{L}_{\alpha_n} f = \partial^{\alpha_n} (\Delta^{-1})^n f$.  Moreover, we have $\text{spt}_x(\RR^{\textup{tran}}_{q+1},P^{\textup{tran}}_{q+1},F^{\textup{tran}}_{q+1})\subset \Omega_{q+1}$ and
\begin{align*}
&\|D_{t,\ell_q} \RR^{\textup{tran}}_{q+1}\|_{N-1}
\lesssim  (\lambda^N_{q+1}\delta^{1/2}_{q+1}) (\frac{\tau^{-1}_q}{\lambda_{q+1}}  \delta^{1/2}_{q+1}), ~ &&N=1,2\\
&\|  \RR^{\textup{tran}}_{q+1} \|_{N}\lesssim  \lambda^N_{q+1}(\frac{\tau^{-1}_q}{\lambda_{q+1}}  \delta^{1/2}_{q+1}), ~~~~&&N=0,1,2\\
&\| P^{\textup{tran}}_{q+1}\|_{0}\lesssim (\frac{\tau^{-1}_q}{\lambda_{q+1}} \delta^{1/2}_{q+1}),~~~
 \|F^{\textup{tran}}_{q+1}\|_N\lesssim    \lambda^{N-10}_{q+1},~~~&&0\leq  N\leq N_0.
\end{align*}
\end{prop}
\begin{proof}
 Since $D_{t,\ell_q} \wpq= \sum_{i;k\in\Lambda_i}
    \ii D_{t,\ell_q}{a}_{k,i}   e^{ 2\pi \ii\lambda_{q+1}k\cdot {\Phi_i}}\bar{k}$, employing   Proposition \ref{tracefree} with $\Phi=\Phi_i,~G= \ii D_{t,\ell_q}{a}_{k,i} $ and $\rho^{(0)}=  e^{2\pi \ii k\cdot x}   $, we obtain
 \begin{align*}
 &\partial_t\wpq+ u_{\ell_q}\nabla\wpq \\
 =& \Div \sum_{m,i;k\in\Lambda_i} \lambda^{-1}_{q+1}A^{(m)}_{\text{tran,p}}
      (\mathcal{L}_{\alpha_m}e^{2\pi ik\cdot })\circ(\lambda_{q+1}\Phi_i   )
 +\nabla \sum_{m,i;k\in\Lambda_i} \lambda^{-1}_{q+1}B^{(m)}_{\text{tran,p}}    (\mathcal{L}_{\alpha_m}e^{2\pi ik\cdot })\circ(\lambda_{q+1}\Phi_i   )
  \\
  &+\sum_{i;k\in\Lambda_i} \lambda^{-N_0}_{q+1}C^{(N_0)}_{\text{tran,p}}    (\mathcal{L}_{\alpha_m}e^{2\pi ik\cdot })\circ(\lambda_{q+1}\Phi_i   )\\
 :=& \Div  \RR_{\text{tran,p}}+\nabla P_{\text{tran,p}}+F_{\text{tran,p}}.
 \end{align*}
 Using Propositions \ref{p:estimates-for-mollified}-\ref{est-a}, we have $\|G\|_m \lesssim {\ell^{-m}_q}\tau^{-1}_q\delta^{1/2}_{q+1} $, and
 \begin{align*}
&\|(A^{(m)}_{\text{tran,p}}(G), B^{(m)}_{\text{tran,p}}(G))\|_0 \lesssim \|G \|_{m-1}\lesssim {\ell^{-(m-1)}_q}\tau^{-1}_q\delta^{1/2}_{q+1},\\
&\|C^{(N_0)}_{\text{tran,p}}(G)\|_0 \lesssim \|G \|_{N_0}\lesssim {\ell^{-{N_0}}_q}\tau^{-1}_q\delta^{1/2}_{q+1}.
\end{align*}
Then we deduce that for $N=0,1,2$,
\begin{align}
 \|\RR_{\text{tran,p}}\|_{N}
\lesssim &  \sum_{m=1}^{ N_0}  \lambda^{-m+N}_{q+1}\|A^{(m)}_{\textup{tran}} \|_0+\sum_{n=1}^{ N_0}  \lambda^{-m}_{q+1}\|A^{(m)}_{\textup{tran}} \|_N\label{tran-p}\\
\lesssim & \lambda^{ N}_{q+1} \sum_{m=1}^{ N_0}(\frac{{\ell^{-1}_q}}{\lambda_{q+1}})^{m-1}(\frac{\tau^{-1}_q}{\lambda_{q+1}}) \delta^{1/2}_{q+1}
 \lesssim  \lambda^{ N}_{q+1} \frac{\tau^{-1}_q}{\lambda_{q+1}}  \delta^{1/2}_{q+1},\notag\\
 \| P_{\text{tran,p}}\|_{0}
\lesssim & \sum_{m=1}^{ N_0}  \lambda^{-m}_{q+1}\|B^{(m)}_{\textup{tran}} \|_0
\lesssim  \sum_{m=1}^{ N_0}(\frac{{\ell^{-1}_q}}{\lambda_{q+1}})^{m-1}(\frac{\tau^{-1}_q}{\lambda_{q+1}}) \delta^{1/2}_{q+1}
\lesssim
  \frac{\tau^{-1}_q}{\lambda_{q+1}}  \delta^{1/2}_{q+1},
\end{align}
and for $N\le N_0$,
\begin{align}
\|F^{\textup{tran}}_{q+1}\|_N
\lesssim&  \lambda^{-{ N_0} }_{q+1}\|C^{(  N_0)}_{\textup{tran}} \|_N     +\lambda^{-{ N_0}+N}_{q+1}\|C^{(  N_0)}_{\textup{tran}} \|_0\notag\\
\lesssim &  (\frac{ {\ell^{-1}_q}}{\lambda_{q+1}})^{{ N_0}}\lambda^{4+N}_{q+1}  \delta^{1/2}_{q+1}
\lesssim
\lambda^{N-10}_{q+1},
\end{align}
where we let $N_0=\frac{50(b+1)}{b-1}$  such that  $\ell^{-N_0 }_q\lambda^{ -N_0 +10}_q   \leq \lambda^{ -10}_{q+1}$.
  %The uniformity constant $C\leq \lambda^{c_0}_1\ll\lambda^{c_0}_q$ and $~\lambda^5_{q+1}\leq\lambda^{c_0M}_{q+1}$ for sufficiently large $M$.% here $C$ is a universal constant in Proposition \ref{tracefree}.
%{2comu}

To estimate $\|D_{t,\ell_q} \RR_{\text{tran,p}}\|_{N-1}$ for $N=1, 2$, it is sufficient to estimate $\|D_{t,\ell_q}(A^{(m)}_{\textup{tran,p}})  \|_{N-1}$. Using
\begin{align}
 D_{t,\ell_q}A^{(m)}_{\textup{tran,p}}
  \sim &
D_{t,\ell_q}\partial^{(\gamma)} (D_{t,\ell_q} {a}_{k,i})
\notag\\
\sim &\sum_{\substack{|\zeta|= 1\\ \gamma_1+\gamma_2+\zeta=\gamma \\|\gamma|=m-1}}   \partial^{(\gamma_1)} (\partial^{(\zeta)} u_{\ell_q}\cdot \nabla \partial^{(\gamma_2)}D_{t,\ell_q}{a}_{k,i})+\partial^{(\gamma)}D_{t,\ell_q}(D_{t,\ell_q} {a}_{k,i}),\notag
\end{align}
one can easily deduce that, for $N=1,2$,
\begin{align}
 &\|D_{t,\ell_q}A^{(m)}_{\textup{tran,p}}\|_{N-1}\notag\\
  \lesssim&
\sum_{\substack{|\zeta|= 1\\ \gamma_1+\gamma_2+\zeta=\gamma }}     \|\partial^{(\gamma_1)} \big(\partial^{(\zeta)}  u_{\ell_q}\cdot \nabla \partial^{(\gamma_2)} (D_{t,\ell_q}{a}_{k,i})\big)\|_{N-1}
+\|\partial^{(\gamma)}D_{t,\ell_q}(D_{t,\ell_q} {a}_{k,i})\|_{N-1}\notag\\
\lesssim& \ell^{-(m-1)-(N-1)}_q\lambda_q\delta^{1/2}_{q}\tau^{-1}_q\delta^{1/2}_{q+1}
+ \ell^{-(m-1)-(N-1)}_q\tau^{-2}_q\delta^{1/2}_{q+1}
\lesssim\ell^{-(m-1)-(N-1)}_q\tau^{-2}_q\delta^{1/2}_{q+1}.
\end{align}
Hence, we have
\begin{align*}
\|D_{t,\ell_q} \RR^{\textup{tran}}_{q+1}\|_{N-1}\leq& \sum_{m=1}^{ N_0}  \lambda^{-m+N-1}_{q+1}\|D_{t,\ell_q}A^{(m)}_{\textup{tran}} \|_{N-1}+\sum_{n=1}^{ N_0}  \lambda^{-m}_{q+1}\|A^{(m)}_{\textup{tran}} \|_{N-1}\\
\lesssim&
\lambda^{ N-1}_{q+1}\sum_{m=1}^{ N_0}(\ell_q\lambda_{q+1})^{-(m-1) } \tau^{-2}_q\lambda^{-1}_{q+1}\delta^{1/2}_{q+1}\\
\lesssim&
(
\lambda^{ N}_{q+1}\delta^{1/2}_{q+2}) \tau^{-1}_q\lambda^{-1}_{q+1}\delta^{1/2}_{q+1},
\end{align*}
where we use the fact that $\tau^{-1}_q\ll\lambda_{q+1}\delta^{1/2}_{q+2}$.

 Using Proposition \ref{tracefree} with $\Phi=\Phi_i,~G= \Big[\frac{1}{ \lambda_{q+1} } \nabla ( {a}_{k,i}) + \frac{1}{ \lambda_{q+1} } {a}_{k,i}\nabla ( \ii\lambda_{q+1}k(\Phi_i -x)) \Big]$ and $\rho^{(0)}=  e^{2\pi \ii k\cdot x}   $, we deduce that
 \begin{align*}
 &\partial_t\wcq+ u_{\ell_q}\nabla\wcq \\
 =& \Div \sum_{m;i;k\in\Lambda_i} \lambda^{-m}_{q+1}A^{(m)}_{\text{tran,c}}    (\mathcal{L}_{\alpha_m}e^{2\pi ik\cdot })\circ(\lambda_{q+1}\Phi_i   )\\
 &+\nabla \sum_{m;i;k\in\Lambda_i}\lambda^{-m}_{q+1} B^{(m)}_{\text{tran,c}}    (\mathcal{L}_{\alpha_m}e^{2\pi ik\cdot })\circ(\lambda_{q+1}\Phi_i   ) \\
 &+\sum_{i;k\in\Lambda_i}\lambda^{-N_0}_{q+1} C^{(N_0)}_{\text{tran,c}}  (\mathcal{L}_{\alpha_m}e^{2\pi ik\cdot })\circ(\lambda_{q+1}\Phi_i   ) \\
 :=& \Div \RR_{\text{tran,c}}+\nabla P_{\text{tran,c}}+F_{\text{tran,c}}.
 \end{align*}
Following the methods on estimating $\RR_{\text{tran,p}}$, we obtain
\begin{align}
&\|D_{t,\ell_q}\RR_{\text{tran,c}}\|_{N-1}
\lesssim (
\lambda^{ N}_{q+1}\delta^{1/2}_{q+2}) \tau^{-1}_q\lambda^{-1}_{q+1}\delta^{1/2}_{q+1},\qquad \qquad N=1,2,\\
 &\|\RR_{\text{tran,c}}\|_{N}
\lesssim   \lambda^{ N}_{q+1} \frac{\tau^{-1}_q}{\lambda_{q+1}}  \delta^{1/2}_{q+1},\qquad \qquad \qquad \qquad \qquad \,\,\, N=0,1,2, \\
 &\| P_{\text{tran,c}}\|_{0}
\lesssim
  \frac{\tau^{-1}_q}{\lambda_{q+1}}  \delta^{1/2}_{q+1}, ~~~~
  \|F_{\text{tran,c}}\|_N
\lesssim \lambda^{N-10}_{q+1} ,\qquad\qquad\,\,\,N\leq N_0 .\label{tran-c}
\end{align}
 Finally, let
\[\RR^{\textup{tran}}_{q+1}=\RR_{\text{tran,p}}+\RR_{\text{tran,c}}, \,\,P^{\textup{tran}}_{q+1}=P_{\text{tran,p}}+P_{\text{tran,c}},\,\,F^{\textup{tran}}_{q+1}=F_{\text{tran,p}}+F_{\text{tran,c}}.\]
Collecting \eqref{tran-p}--\eqref{tran-c} together, we finish the proof of Proposition \ref{proptrans}.
\end{proof}
\begin{prop}[Estimates for $\Rnash$]\label{nash}
$T^{\textup{Nash}}_{q+1}$ can be decomposed into the three parts:
 \begin{align*}
 T^{\textup{Nash}}_{q+1}
 =& \textup{div} \sum_{m,i;k\in\Lambda_i} \lambda^{-m}_{q+1}A^{(m)}_{\textup{Nash }}
      (\mathcal{L}_{\alpha_m}e^{2\pi \ii k\cdot })\circ(\lambda_{q+1}\Phi_i   )\\
 &+\nabla \sum_{m,i;k\in\Lambda_i} \lambda^{-m}_{q+1}B^{(m)}_{\textup{Nash }}    (\mathcal{L}_{\alpha_m}e^{2\pi \ii k\cdot })\circ(\lambda_{q+1}\Phi_i   )
  \\
  &+\sum_{i;k\in\Lambda_i} \lambda^{-N_0}_{q+1}C^{(N_0)}_{\textup{Nash }}   (\mathcal{L}_{\alpha_m}e^{2\pi \ii k\cdot })\circ(\lambda_{q+1}\Phi_i   )\\
 =:& \textup{div}  \RR^{\textup{Nash}}_{q+1}+\nabla P^{\textup{Nash}}_{q+1}+F^{\textup{Nash}}_{q+1},
 \end{align*}
 where $\mathcal{L}_{\alpha_n} f = \partial^{\alpha_n} (\Delta^{-1})^n f$. Moreover, we have $\text{spt}_x(\RR^{\textup{Nash}}_{q+1},P^{\textup{Nash}}_{q+1},F^{\textup{Nash}}_{q+1})\subset \Omega_{q+1}$ and
\begin{align*}
&\|D_{t,\ell_q} \Rnash\|_{N-1}
\lesssim  (\lambda^N_{q+1}\delta^{1/2}_{q+1}) \Big(\frac{\lambda_q\delta^{1/2}_q}{\lambda_{q+1}}  \delta^{1/2}_{q+1}\Big), ~~~~&&N=1,2\\
&\|  \Rnash \|_{N}\lesssim  \lambda^N_{q+1}\Big(\frac{\lambda_q\delta^{1/2}_q}{\lambda_{q+1}}  \delta^{1/2}_{q+1}\Big), ~~~~&&N=0,1,2\\
&\| P^{\textup{Nash}}_{q+1}\|_{0}\lesssim \Big(\frac{\lambda_q\delta^{1/2}_q}{\lambda_{q+1}}  \delta^{1/2}_{q+1}\Big),~~~
 \|\RR^{\textup{Nash}}_{q+1}\|_N\lesssim    \lambda^{N-10}_{q+1},~~~~&&N\leq N_0.
\end{align*}
\end{prop}
\begin{proof}
 Note that
 $$w_{q+1}\cdot\nabla  u_{\ell_q}=(\wpq+\wcq)\cdot \nabla  u_{\ell_q},$$
 it is sufficient to estimate $\wpq \cdot\nabla  u_{\ell_q}$. Since
\begin{align*}
 \wpq\cdot\nabla  u_{\ell_q}  =  \sum_{i;k\in\Lambda_i} (\ii a_{k,i}\bar{k}\cdot\nabla  u_{\ell_q})
 e^{ 2\pi\ii \lambda_{q+1}k\cdot\Phi_i},
\end{align*}
and using Proposition \ref{tracefree} with $\Phi=\Phi_i,~G= (\ii a_{k,i}\bar{k}\cdot\nabla  u_{\ell_q}) $ and $\rho^{(0)}=  e^{2\pi \ii k\cdot x}   $, we obtain
 \begin{align*}
 \wpq\cdot\nabla  u_{\ell_q}
= &\Div  \sum_{m,i;k\in\Lambda_i} \lambda^{-1}_{q+1}A^{(m)}_{\text{Nash,p}}
      (\mathcal{L}_{\alpha_m}e^{2\pi ik\cdot })\circ(\lambda_{q+1}\Phi_i   )\\
& +\nabla \sum_{m,i;k\in\Lambda_i} \lambda^{-1}_{q+1}B^{(m)}_{\text{Nash,p}}    (\mathcal{L}_{\alpha_m}e^{2\pi ik\cdot })\circ(\lambda_{q+1}\Phi_i   )
  \\
  &+\sum_{i;k\in\Lambda_i} \lambda^{-N_0}_{q+1}C^{(N_0)}_{\text{Nash,p}}   (\mathcal{L}_{\alpha_m}e^{2\pi ik\cdot })\circ(\lambda_{q+1}\Phi_i   )\\
 :=& \Div \RR_{\text{Nash,p}}+\nabla P_{\text{Nash,p}}+F_{\text{Nash,p}}.
 \end{align*}
 It follows from
 \begin{align*}
&\|(A^{(m)}_{\text{Nash,p}}(G), B^{(m)}_{\text{Nash,p}}(G))\|_0 \lesssim \|G \|_{m-1}\lesssim {\ell^{-(m-1)}_q}\lambda_q\delta^{1/2}_{q }\delta^{1/2}_{q+1}\\
&\|C^{(N_0)}_{\text{Nash,p}}(G)\|_0 \lesssim \|G \|_{N_0}\lesssim {\ell^{-{N_0}}_q}\lambda_q\delta^{1/2}_{q }\delta^{1/2}_{q+1},
\end{align*}
that, for $N=0,1,2$,
\begin{align}
 \|\RR_{\text{Nash,p}}\|_{N}
\lesssim &  \sum_{m=1}^{ N_0}  \lambda^{-m+N}_{q+1}\|A^{(m)}_{\textup{Nash,p}} \|_0+\sum_{n=1}^{ N_0}  \lambda^{-m}_{q+1}\|A^{(m)}_{\textup{Nash,p}} \|_N\notag\\
\lesssim & \lambda^{ N}_{q+1} \sum_{m=1}^{N_0}(\frac{{\ell^{-1}_q}}{\lambda_{q+1}})^{m-1}(\frac{\lambda_q\delta^{1/2}_{q }}{\lambda_{q+1}}) \delta^{1/2}_{q+1}
 \lesssim \lambda^{ N}_{q+1} \frac{\lambda^{1/2}\delta^{1/2}_{q }}{\lambda_{q+1}}  \delta^{1/2}_{q+1},~~~~ \notag\\
 \| P_{\text{Nash,p}}\|_{0}
\lesssim & \sum_{m=1}^{ N_0}  \lambda^{-m}_{q+1}\|B^{(m)}_{\textup{Nash,p}} \|_0
\lesssim
  \frac{\lambda^{1/2}\delta^{1/2}_{q }}{\lambda_{q+1}}  \delta^{1/2}_{q+1},\notag
\end{align}
and for $N\leq N_0$,
\begin{align*}
\|F_{\text{Nash,p}}\|_N
\lesssim&  \lambda^{-{N_0} }_{q+1}\|C^{( N_0)}_\text{Nash,p} \|_N     +\lambda^{-{N_0}+N}_{q+1}\|C^{(N_0)}_{\textup{tran}} \|_0
\lesssim   (\frac{ {\ell^{-1}_q}}{\lambda_{q+1}})^{{ N_0}}\lambda^{4+N}_{q+1}  \delta^{1/2}_{q+1}
\lesssim
\lambda^{N-10}_{q+1}  ,
\end{align*}
where $N_0$ is large enough such that $ (\frac{\ell^{-1}_q}{\lambda_{q+1}})^{ N_0}\leq\lambda^{-10}_{q+1}$.

Now we estimate $\|D_{t,\ell_q} \RR_{\text{Nash,p}}\|_{N-1}$ for $N=1,2$. Note that
\begin{align*}
 D_{t,\ell_q}A^{(m)}_{\text{Nash,p}}
   \sim&
D_{t,\ell_q}\partial^{(\gamma)}  (a_{k,i}\bar{k} \nabla  u_{\ell_q})
\notag\\
\sim  &\sum_{ \substack{  |\zeta|= 1\\ \gamma_1+\gamma_2+\zeta=\gamma\\|\gamma|=m-1}}   \partial^{(\gamma_1)} (\partial^{(\xi)} u_{\ell_q}\cdot \nabla \partial^{(\gamma_2)}(a_{k,i}\bar{k} \nabla  u_{\ell_q}) )+\partial^{(\gamma)}D_{t,\ell_q}(a_{k,i}\bar{k} \nabla  u_{\ell_q}),
\end{align*}
we deduce by $m\le N_0$ that, for $ N=1,2$,
\begin{align*}
 &\|D_{t,\ell_q}A^{(m)}_{\text{Nash,p}}\|_{N-1}\\
  \lesssim&
\sum_{ \substack{  |\zeta|= 1\\ \gamma_1+\gamma_2+\zeta=\gamma\\|\gamma|=m-1}}     \|\partial^{(\gamma_1)} \big(\partial^{(\xi)}  u_{\ell_q}\cdot \nabla \partial^{(\gamma_2)} (a_{k,i}\bar{k} \nabla  u_{\ell_q}) \big)\|_{N-1}
+\|\partial^{(\gamma)}D_{t,\ell_q}(a_{k,i}\bar{k} \nabla  u_{\ell_q}) \|_{N-1}\\
\lesssim& \ell^{-(m-1)-(N-1)}_q\lambda^2_q\delta_{q}\delta^{1/2}_{q+1}
+ \ell^{-(m-1)-(N-1)}_q\lambda_q\delta^{1/2}_{q}\tau^{-1}_q\delta^{1/2}_{q+1}\\
\lesssim&\ell^{-(m-1)-(N-1)}_q\lambda_q\delta^{1/2}_{q}\tau^{-1}_q\delta^{1/2}_{q+1},
\end{align*}
and for $N=1,2$,
\begin{align*}
 \|D_{t,\ell_q} \RR_{\text{Nash,p}}\|_{N-1}
\leq& \sum_{m=1}^{N_0}  \lambda^{-m+N-1}_{q+1}\|D_{t,\ell_q}A^{(m)}_{\text{Nash,p}} \|_{0}+\sum_{n=1}^{N_0}  \lambda^{-m}_{q+1}\|D_{t,\ell_q}A^{(m)}_{\text{Nash,p}} \|_{N-1}\\
\lesssim&
\lambda^{ N-1}_{q+1}\sum_{m=1}^{N_0}(\ell_q\lambda_{q+1})^{-(m-1) } \lambda^{-1}_{q+1}\lambda_q\delta^{1/2}_{q}\tau^{-1}_q \delta^{1/2}_{q+1}\\
\lesssim&
(\lambda^{ N}_{q+1}\delta^{1/2}_{q+2}) \lambda_q\delta^{1/2}_{q} \lambda^{-1}_{q+1}\delta^{1/2}_{q+1}.
\end{align*}
\end{proof}

\begin{prop}[Estimates for $\Rosc $]\label{propRosc}$T^{\textup{osc}}_{q+1}$ has the following decomposition:
 \begin{align*}
 T^{\textup{osc}}_{q+1}=& \textup{div}  \sum_{\substack{i;k\in\Lambda_{i}\\ i';k'\in\Lambda_{i'}}}\sum_{m=1}^{N_0}\lambda^{-m}_{q+1}
 A^{(m)}_{ \textup{osc }}
  (\mathcal{L}_{\alpha_n}e^{2\pi i(k'+k)\cdot })\circ(\lambda_{q+1}\cdot   )\\
 &+\nabla \sum_{\substack{i;k\in\Lambda_{i}\\ i';k'\in\Lambda_{i'}}}\sum_{m=1}^{N_0}\lambda^{-m}_{q+1}B^{(m)}_{ \textup{osc }}     (\mathcal{L}_{\alpha_n}e^{2\pi i(k'+k)\cdot })\circ(\lambda_{q+1}\cdot   ) \\
 &+\sum_{\substack{i;k\in\Lambda_{i}\\ i';k'\in\Lambda_{i'}}} \lambda^{-N_0}_{q+1}C^{(N_0)}_{ \textup{osc }}      (\mathcal{L}_{\alpha_n}e^{2\pi i(k'+k)\cdot })\circ(\lambda_{q+1}\cdot   )\\
 =:& \textup{div}  \Rosc+\nabla P^{\textup{osc}}_{q+1}+F^{\textup{osc}}_{q+1},
 \end{align*}
where $\mathcal{L}_{\alpha_n} f = \partial^{\alpha_n} (\Delta^{-1})^n f$. Moreover,  $\text{spt}_x(\Rosc,P^{\textup{osc}}_{q+1},F^{\textup{osc}}_{q+1})\subset \Omega_{q+1}$ and
\begin{align*}
&\|D_{t,\ell_q} \Rosc\|_{N-1}
\lesssim  \lambda^N_{q+1}\delta^{1/2}_{q+1}   (\tau_q\lambda_{q}\delta^{1/2}_{q}) \delta_{q+1}, ~~~~&&N=1,2,\\
&\|\Rosc \|_{N}\lesssim  \lambda^N_{q+1}  (\tau_q\lambda_{q}\delta^{1/2}_{q}) \delta_{q+1}, ~~~~~~~&&N=0,1,2,\\
&\| P^{\textup{osc}}_{q+1}\|_{0}\lesssim\delta_{q+1},\,\,\,
 \|F^{\textup{osc}}_{q+1}\|_N\lesssim    \lambda^{N-10}_{q+1},~~~&&N\leq N_0.
\end{align*}
\end{prop}
\begin{proof}
By making use of the decomposition $w_{q+1}=\wpq+\wcq$, we divide $T^{\textup{osc}}_{q+1}$ as follows:
\begin{align}\label{1}
T^{\textup{osc }  }_{q+1}=& \Div\big(\wpq\otimes\wpq+\RR_{\ell_{t,x}}\big)
 + \Div\big(\wcq {\otimes}\wcq+\wpq {\otimes}\wcq+  \wpq {\otimes}\wcq\big).
\end{align}
Since $\Lambda_i\subset\nu\mathbb{S}^2\cap\mathbb{Z}^3$ and $-\Lambda_i=\Lambda_i$, we have
$$\Div_{\xi}\Big((\sum_{i ;k \in\Lambda_i }\ii c_{k,i}e^{2\pi \ii k\cdot\xi}\bar{k})\otimes (\sum_{i ;k \in\Lambda_i }\ii c_{k,i}e^{2\pi \ii k\cdot\xi}\bar{k})\Big)=\frac{1}{2}\nabla_{\xi}  \Big(\big|\sum_{i ;k \in\Lambda_i }c_{k,i}e^{2\pi \ii k\cdot\xi}\bar{k}\big|^2-\big|\sum_{i ;k \in\Lambda_i } c_{k,i}e^{2\pi\ii k\cdot\xi}  \big|^2\Big),$$
as in \cite{CDS}.  Using Lemmas \ref{Betrimi} and \ref{first S}, a directly computation yields that
\begin{align*}
& \Div\Big({\wpq\otimes\wpq }+\RR_{\ell_q,\tau_q}\Big)\\
=&\frac{1}{2}\nabla_{\xi}  \Big(\big|\sum_{i ;k \in\Lambda_i }\bar{a}_{k,i}e^{2\pi\ii \lambda_{q+1}k\cdot\xi}\bar{k} \big|^2-\big|\sum_{i ;k \in\Lambda_i } {\bar{a}_{k,i}e^{2\pi\ii \lambda_{q+1}k\cdot\xi} } \big|^2\Big)\Big|_{\xi=x}+\Div\RR_{\ell_q,\tau_q}\\
&+\Div_x\big(\sum_{\substack{k'+k\neq0\\ i;k\in\Lambda_{i}\\ i';k'\in\Lambda_{i'}}}- \bar{a}_{k,i}\bar{a}_{k',i'}e^{2\pi\ii \lambda_{q+1}(k'+k)\cdot \xi}\bar{k}'\otimes \bar{k}  +\sum_{\substack{i;k\in\Lambda_{i}\\ }}   {a}^2_{k,i}\bar{k}\otimes \bar{k}\big)\Big|_{\xi=x} \\
 =&\nabla_{x}\Big( \frac{1}{2}\sum_{\substack{k'+k\neq0\\ i;k\in\Lambda_{i}\\ i';k'\in\Lambda_{i'}}} {\bar{a}_{k,i}\bar{a}_{k',i'}  \bar{k } \cdot\bar{k'}} e^{2\pi\ii \lambda_{q+1}(k'+k)\cdot x} - \frac{1}{2}\sum_{\substack{k'+k\neq0\\ i;k\in\Lambda_{i}\\ i';k'\in\Lambda_{i'}}} {\bar{a}_{k,i}\bar{a}_{k',i'}e^{2\pi\ii \lambda_{q+1}(k'+k)\cdot x} }  +\chi^2_q e_q \Big)\\
 &-\sum_{\substack{k'+k\neq0\\ i;k\in\Lambda_{i}\\ i';k'\in\Lambda_{i'}}}\Big(\frac{1}{2}\nabla\big(  {\bar{a}_{k,i}\bar{a}_{k',i'}  \bar{k }\cdot \bar{k'}} -{\bar{a}_{k,i}\bar{a}_{k',i'} }  \big)+\Div \big(\bar{a}_{k,i}\bar{a}_{k',i'}\bar{k}'\otimes \bar{k}\big)\Big) e^{2\pi\ii \lambda_{q+1}(k'+k)\cdot x}\\
=:&\nabla_{x} P_{\textup{osc,0 }}+\sum_{\substack{k'+k\neq0\\ i;k\in\Lambda_{i}\\ i';k'\in\Lambda_{i'}}}G_{\textup{osc}}e^{2\pi\ii \lambda_{q+1}(k'+k)\cdot x}.
\end{align*}
Making use of  Proposition \ref{tracefree} with $\Phi=x,~G= G_{\textup{osc}}  $ and $\rho^{(0)}= e^{2\pi i (k'+k)\cdot x}$,  we obtain
 \begin{align*}
  \sum_{\substack{k'+k\neq0\\ i;k\in\Lambda_{i}\\ i';k'\in\Lambda_{i'}}} G_{\textup{osc}}e^{2\pi\ii \lambda_{q+1}(k'+k)x}
 =& \Div \sum_{\substack{k'+k\neq0\\ i;k\in\Lambda_{i}\\ i';k'\in\Lambda_{i'}}}\sum_{m=1}^{ N_0}\lambda^{-m}_{q+1}A^{(m)}_{\textup{osc}}     (\mathcal{L}_{\alpha_n}e^{2\pi  i(k'+k)\cdot })\circ(\lambda_{q+1}\cdot   )\\
&+\nabla \sum_{\substack{k'+k\neq0\\ i;k\in\Lambda_{i}\\ i';k'\in\Lambda_{i'}}}\sum_{m=1}^{ N_0}\lambda^{-m}_{q+1}B^{(m)}_{\textup{osc}}     (\mathcal{L}_{\alpha_n}e^{2\pi  i(k'+k)\cdot })\circ(\lambda_{q+1}\cdot   )   \\
&+\sum_{\substack{k'+k\neq0\\ i;k\in\Lambda_{i}\\ i';k'\in\Lambda_{i'}}}\lambda^{-N_0}_{q+1} C^{( N_0)}_{\textup{osc}}    (\mathcal{L}_{\alpha_n}e^{2\pi  i(k'+k)\cdot })\circ(\lambda_{q+1}\cdot   )   \\
:=& \Div \RR_{osc,1}+\nabla P_{\textup{osc,1 }}+F^{\textup{osc} }_{q+1}.
\end{align*}
Owing to
 \begin{align*}
&\|(A^{(m)}_{\textup{osc }}(G), B^{(m)}_{\textup{osc }}(G))\|_0 \lesssim \|G \|_{m-1}\lesssim \ell^{-(m-1)}_q \ell^{-1}_q  \delta_{q+1},\\
&\|C^{(N_0)}_{\textup{osc }}(G)\|_0 \lesssim \|G \|_{N_0}\lesssim {\ell^{-{N_0}}_q}\ell^{-1}_q  \delta_{q+1},
\end{align*}
We deduce that, for $N=0,1,2$,
\begin{align}\label{osc1}
 \|\RR_{osc,1}\|_{N}
\lesssim &  \sum_{m=1}^{ N_0}  \lambda^{-m+N}_{q+1}\|A^{(m)}_{\textup{osc}} \|_0+\sum_{m=1}^{ N_0}  \lambda^{-m}_{q+1}\|A^{(m)}_{\textup{osc}} \|_N \\
\lesssim&  \lambda^{ N}_{q+1} \sum_{m=1}^{ N_0}( \ell_q\lambda_{q+1})^{-(m-1)}(( \ell_q\lambda_{q+1})^{-1}\delta^{1/2}_{q+1})\notag\\
 \lesssim& \lambda^{ N}_{q+1} (( \ell_q\lambda_{q+1})^{-1}\delta^{1/2}_{q+1}) \lesssim\lambda^{N }_{q+1}(\tau_q\lambda_{q}\delta^{1/2}_{q}) \delta_{q+1} ,~~~\notag\\
 \| P_{\textup{osc,1 }}\|_{0}
\lesssim & \sum_{m=1}^{ N_0}  \lambda^{-m}_{q+1}\|B^{(m)}_{\textup{osc}} \|_0
\lesssim
  ( \ell_q\lambda_{q+1})^{-1}\delta_{q+1} \lesssim (\tau_q\lambda_{q}\delta^{1/2}_{q}) \delta_{q+1}  ,
\end{align}
and for $N\leq N_0$,
\begin{align}\label{Fosc}
\|F^{\textup{osc} }_{q+1}\|_N
\lesssim   \lambda^{-{ N_0} }_{q+1}\|C^{(  N_0)}_{\textup{osc}} \|_N     +\lambda^{-{ N_0}+N}_{q+1}\|C^{( N_0)}_{\textup{osc}} \|_0
\lesssim\lambda^{N-10}_{q+1}.
\end{align}
To estimate $\|D_{t,\ell_q} \RR_{osc,1}\|_{N-1}$ for $N=1,2$, we bound $\|D_{t,\ell_q}(A^{(m)}_{\textup{osc}})  \|_{N-1}$. Since
\begin{align*}
 D_{t,\ell_q}A^{(m)}_{\textup{osc}}
  \sim &
D_{t,\ell_q}\partial^{(\gamma)}   G_{\textup{osc}}
\notag\\
\sim  &\sum_{\substack{|\zeta|= 1\\ \gamma_1+\gamma_2+\zeta=\gamma \\ |\gamma|=m-1} }  \partial^{(\gamma_1)} (\partial^{(\zeta)} u_{\ell_q}\cdot \nabla \partial^{(\gamma_2)} G_{\textup{osc}}  )-\partial^{(\gamma)}D_{t,\ell_q} G_{\textup{osc}},
\end{align*}
we easily deduce that
\begin{align*}
 \|D_{t,\ell_q}A^{(m)}_{\textup{osc}}\|_{N-1}
  \lesssim&
\sum_{\substack{|\zeta|= 1\\ \gamma_1+\gamma_2+\zeta=\gamma \\
|\gamma|=m-1}}     \|\partial^{(\gamma_1)} \big(\partial^{(\zeta)}  u_{\ell_q}\cdot \nabla \partial^{(\gamma_2)}  G_{\textup{osc}}  \big)\|_{N-1}
+\|\partial^{(\gamma)}D_{t,\ell_q} G_{\textup{osc}}  \|_{N-1}\\
\lesssim& \ell^{-(m-1)-(N-1)}_q\lambda_q\delta^{1/2}_{q}\ell^{-1}_q\delta_{q+1}
+ \ell^{-(m-1)-(N-1)}_q\tau^{-1}_q\ell^{-1}_q\delta_{q+1}\\
\lesssim&\ell^{-(m-1)-(N-1)}_q\tau^{-1}_q\ell^{-1}_q\delta_{q+1}.
\end{align*}
Hence, we have, for $N=1,2$,
\begin{align}\label{dtosc1}
\|D_{t,\ell_q} \RR_{osc,1}\|_{N-1}
\leq& \sum_{m=1}^{  N_0}  \lambda^{-m+N-1}_{q+1}\|D_{t,\ell_q}A^{(m)}_{\textup{osc}} \|_{0}+\sum_{n=1}^{ N_0}  \lambda^{-m}_{q+1}\|A^{(m)}_{\textup{osc}} \|_{N-1}\notag\\
\lesssim&\sum_{m=1}^{ N_0} \lambda^{-m+N-1}_{q+1}\ell^{-(m-1) }_q\tau^{-1}_q\ell^{-1}_q\delta_{q+1}\notag\\
\lesssim&
(\lambda^{ N}_{q+1}\delta^{1/2}_{q+2}) \lambda^{-1}_{q+1}\ell^{-1}_q\delta_{q+1}\lesssim
(\lambda^{ N}_{q+1}\delta^{1/2}_{q+2})(\tau_q\lambda_{q}\delta^{1/2}_{q})\delta_{q+1}.
\end{align}
Finally, it follows from \ref{estimate-wq+1} that, for $N=0,1,2$,
\begin{align}
\| \wcq{\ootimes} \wcq+\wpq {\ootimes} \wcq+  \wcq {\ootimes} \wpq \|_N
\leq&  \lambda^{ N}_{q+1}(\tau_q\lambda_{q}\delta^{1/2}_{q})\delta_{q+1},\label{4}\\
\end{align}
and for $N=1,2$,
\begin{align}
\|D_{t,\ell_q}( \wcq {\ootimes} \wcq+\wpq{\ootimes} \wcq+ \wcq {\ootimes} \wpq) )\|_{N-1}
\leq& (\lambda^{ N}_{q+1}\delta^{1/2}_{q+2})(\tau_q\lambda_{q}\delta^{1/2}_{q})\delta_{q+1}.\label{4.5}
\end{align}
Let
 \begin{align*}
 P_{\textup{osc,2}}:=&\frac{1}{3}\tr (\wcq{\otimes} \wcq+\wpq {\otimes} \wcq+  \wcq {\otimes} \wpq),\\
 \RR_\textup{osc,2}:=&\wcq{\ootimes} \wcq+\wpq {\ootimes} \wcq+  \wcq {\ootimes} \wpq , \end{align*}
 we define $$P^{\textup{osc}}_{q+1}:=P_{\textup{osc,0 }}+P_{\textup{osc,1 }}+P_{\textup{osc,2 }},\quad
 \Rosc:=\RR_{osc,1}+\RR_{osc,2}.$$
 Plugging \eqref{osc1}-- \eqref{4.5} into \eqref{1}, we complete the proof of  Proposition \ref{propRosc}.
\end{proof}

Let $F_{q+1}:=F^{\textup{tran} }_{q+1}+F^{\textup{Nash}}_{q+1}+F^{\textup{osc}}_{q+1}$, we obtain that
\begin{align}\label{est-Fq+1}
\text{spt}_xF_{q+1}\subset \Omega_{q+1},~~~~~~~~~~\|F_{q+1}\|_N\leq \lambda^{N-10}_{q+1}  ,\quad \forall  {N\leq N_0}.\end{align}
Next, we show the local well-posedness of \eqref{e:wt} in $[0,T]$ and establish  higher-order derivative estimates of $(\theta_{q+1},d_{q+1})$. More precisely, we have
\begin{prop}\label{est-GMHD}
 \eqref{e:wt} has a unique solution $(\theta_{q+1},d_{q+1})$ on $[0,T]$ such that
\begin{align}
 &\|\theta_{q+1}\|_{L^2}+\|d_{q+1}\|_{L^2}\leq \lambda^{-10}_{q+1},\label{L2}  \\
&\|\theta_{q+1}\|_{H^4}+\|d_{q+1}\|_{H^4}\leq\lambda^{-3}_{q+1} , \label{H4} \\
 &\|\theta_{q+1}\|_{H^N}+\|d_{q+1}\|_{H^N}\leq \lambda^{N }_{q+1} \delta^{1/2}_{q+1} ,~~~&&4\leq N \leq N_0.\label{HN}\\
  &\|  D_{t,\ell_q}\theta_{q+1}\|_{N}+\|  D_{t,\ell_q}d_{q+1}\|_{N} \le \lambda^{-3}_{q+1}  ,\,\quad \quad\quad &&N=0,1.\label{Dtnon}
 \end{align}
\end{prop}
\begin{proof}
It suffices to give the  priori estimates to $(\theta_{q+1},d_{q+1})$, since the classical fixed point theorem could guarantee the existence and uniqueness of \eqref{e:wt}. Let $(\theta_{q+1},d_{q+1})\in C_TH^N$ with $4\leq N\leq N_0$  be a solution of system \eqref{e:wt}. Taking $L^2_x$ inner product of  the first equation of \eqref{e:wt} with $\theta_{q+1}$ and of the
second one with $d_{q+1}$, we get
\begin{align*}
&\frac{1}{2}\frac{\dd }{\dd t}\int_{\mathbb{T}^2}|\theta_{q+1}|^2\dd x + \int_{\mathbb{T}^2}(\theta_{q+1}\cdot\nabla \theta_{q+1}+\vnon_{\ell_q}\cdot\nabla \theta_{q+1}+\theta_{q+1} + \theta_{q+1}\cdot\nabla \vnon_{\ell_q}+\nabla P^{(\textup{per})}_{q+1})\cdot\theta_{q+1}\dd x \\
   =&  \int_{\mathbb{T}^2}  (d_{q+1}\cdot\nabla d_{q+1}+b_{\ell_q}\cdot\nabla d_{q+1}+d_{q+1}\cdot\nabla b_{\ell_q}+ \Div\RR_{\textup{mol,v}}+F_{q+1})\cdot\theta_{q+1}\dd x,
   \end{align*}
   and
   \begin{align*}
&\frac{1}{2}\frac{\dd }{\dd t}\int_{\mathbb{T}^2}|d_{q+1}|^2\dd x +  \int_{\mathbb{T}^2}(\theta_{q+1}\cdot\nabla d_{q+1} + \vnon_{\ell_q}\cdot\nabla d_{q+1}+\theta_{q+1}\cdot\nabla b_{\ell_q})\cdot d_{q+1}\dd x\\
      =&\int_{\mathbb{T}^2}(d_{q+1}\cdot\nabla \theta_{q+1}+b_{\ell_q}\cdot\nabla \theta_{q+1}+d_{q+1}\cdot\nabla \vnon_{\ell_q} + \Div\RR_{\textup{mol,b}})\cdot d_{q+1}\dd x
\end{align*}
Owning to
 $$\int_{\mathbb{T}^2}(d_{q+1}\cdot\nabla d_{q+1}+b_{\ell_q} \cdot\nabla d_{q+1}) \cdot d_{q+1} +(d_{q+1}\cdot\nabla \theta_{q+1}+b_{\ell_q}\cdot\nabla \theta_{q+1})\cdot\theta_{q+1}=0,$$
 and  $$\int_0^T\|\vnon_{\ell_q}\|_{H^3}+\|b_{\ell_q}\|_{H^3}\dd t\lesssim T\epsilon_0 \lesssim T\lambda^{-\frac{1}{20}}_1 \ll 1,$$ By  Gr\"{o}nwall's inequality, we deduce that
\begin{align*}
&\|\theta_{q+1}\|^2_{L^\infty_TL^2}+\|d_{q+1}\|^2_{L^\infty_TL^2}\\
\lesssim&
 \int_{ 0}^{T}  (\|b_{\ell_q},\vnon_{\ell_q})\|_{H^3}) (\|\theta_{q+1}\|^2_{L^2}+\|d_{q+1}\|^2_{L^2}) +\|(\Div\RR_{\textup{mol,v}}, \Div\RR_{\textup{mol,b}}, F_{q+1})\|_{L^2}\dd t\\
  \lesssim& \exp\Big(\int_{ 0}^{1} \|(b_{\ell_q},\vnon_{\ell_q})\|_{H^3}\dd t\Big)\int_{ 0}^{1}\|(\Div\RR_{\textup{mol,v}},\Div\RR_{\textup{mol,b}}, F_{q+1}\|_{L^2}\dd t\\
  \lesssim& \lambda^{-10}_{q+1}.
\end{align*}
This inequality implies \eqref{L2} by \eqref{est-Rell} and \eqref{est-Fq+1} .

Furthermore, for $\forall j=1,2$, one obtains from \eqref{e:wt}  that
\begin{align*}
&\frac{1}{2}\frac{\dd}{\dd t}\int_{\mathbb{T}^2}|\partial^4_j\theta_{q+1}|^2\dd x + \int_{\mathbb{T}^2}\partial^4_j(\theta_{q+1}\cdot\nabla \theta_{q+1}+\vnon_{\ell_q}\cdot\nabla \theta_{q+1}+\theta_{q+1}\cdot\nabla \vnon_{\ell_q})\cdot\partial^4_j\theta_{q+1}\dd x \\
&   =  \int_{\mathbb{T}^2} \partial^4_j(d_{q+1}\cdot\nabla d_{q+1}+b_{\ell_q}\cdot\nabla d_{q+1}+d_{q+1}\cdot\nabla b_{\ell_q}+ \Div\RR_{\textup{mol,v}}+F_{q+1})\cdot\partial^4_j\theta_{q+1}\dd x,
\end{align*}
and
\begin{align*}
&\frac{1}{2}
\frac{\dd}{\dd t}\int_{\mathbb{T}^2}|\partial^4_jd_{q+1}|^2\dd x +  \int_{\mathbb{T}^2}\partial^4_j(\theta_{q+1}\cdot\nabla d_{q+1} + \vnon_{\ell_q}\cdot\nabla d_{q+1}+\theta_{q+1}\cdot\nabla b_{\ell_q})\cdot\partial^4_jd_{q+1}\dd x\\
&      =\int_{\mathbb{T}^2}\partial^4_j(d_{q+1}\cdot\nabla \theta_{q+1}+b_{\ell_q}\cdot\nabla \theta_{q+1}+d_{q+1}\cdot\nabla \vnon_{\ell_q} + \Div\RR_{\textup{mol,b}})\cdot\partial^4_jd_{q+1}\dd x.
\end{align*}
Thanks to
\begin{align*}
&\Big|\int_{\mathbb{T}^2}\partial^4_j(d_{q+1}\cdot\nabla d_{q+1}+b_{\ell_q}\cdot\nabla d_{q+1}+d_{q+1}\cdot\nabla \theta_{q+1}+b_{\ell_q}\cdot\nabla \theta_{q+1})\cdot\partial^4_jd_{q+1}\dd x\Big|\\ \leq
  & \|(\vnon_{\ell_q},b_{\ell_q},\theta_{q+1}, d_{q+1})\|_{H^4}\|(\theta_{q+1},d_{q+1})\|_{H^4},
 \end{align*}
we deduce that
\begin{align}\label{est-H4}
\|\theta_{q+1}\|^2_{L^\infty_T\dot H^4}+\|d_{q+1}\|^2_{L^\infty_T\dot H^4}
\lesssim&
 \int_{ 0}^{T}  (\|(b_{\ell_q},\vnon_{\ell_q},\theta_{q+1},d_{q+1})\|_{H^4}  (\|\theta_{q+1}\|^2_{L^2}+\|d_{q+1}\|^2_{L^2})\notag\\
  &
 +\|(d_{q+1},\theta_{q+1})\|_{L^2}\|(\vnon_{\ell_q},b_{\ell_q})\|_{W^{5,\infty}}
 \|(d_{q+1},\theta_{q+1})\|_{H^4} \notag\\
  &+\|(\Div\RR_{\textup{mol,v}}, \Div\RR_{\textup{mol,b}},F_{q+1})\|_{H^4}\dd t.
\end{align}
Let $T_1< T$ be the maximum lifespan. We conclude by  a continuity argument and \eqref{L2} that $$\| \theta_{q+1}\|^2_{H^4}+\| d_{q+1}\|^2_{H^4}\leq \lambda^{-6}_{q+1},\quad t\in[0,T_1].$$
Using \eqref{L2} and the fact
  $$\|(\vnon_{\ell_q}, b_{\ell_q})\|_{W^{5,\infty}}
\leq \|\vnon_{\ell_q}\|^{1/2}_{H^{5 }}\|\vnon_{\ell_q}\|^{1/2}_{H^{7 }}
 +\|b_{\ell_q}\|^{1/2}_{H^5 }\|b_{\ell_q}\|^{1/2}_{H^7 }\lesssim \lambda^7_{q}\delta^{1/2}_{q},$$   we deduce that
\begin{align}
&\|\theta_{q+1}\|^2_{L^\infty_T\dot H^4}+\|d_{q+1}\|^2_{L^\infty_T\dot H^4}\notag\\
\lesssim&
 \int_{ 0}^{T_1}  \|(b_{\ell_q},\vnon_{\ell_q},\theta_{q+1},d_{q+1})\|_{H^4}(\|\theta_{q+1}\|^2_{L^2}+\|d_{q+1}\|^2_{L^2})
 \notag\\
 &+ \|(d_{q+1},\theta_{q+1})\|_{L^2}
  \|(\vnon_{\ell_q},b_{\ell_q})\|_{W^{4,\infty}}
 \|(d_{q+1},\theta_{q+1})\|_{H^4}\notag\\
 &+\|\Div\RR_{\textup{mol,v}}\|^2_{H^4}  + \|\Div\RR_{\textup{mol,b}}\|^2_{H^4}+\|F_{q+1}\|^2_{H^4}\dd t\notag\\
  \lesssim&\int_{ 0}^{T_1}(\sum_{l=1}^{q }\lambda^{-1}_l+\lambda^{-4}_{q+1})(\|\theta_{q+1}\|^2_{L^2}+\|d_{q+1}\|^2_{L^2})
  +\lambda^{-10}_{q+1}(\lambda^7_{q}\delta^{1/2}_{q})(\|(d_{q+1},\theta_{q+1})\|_{H^4}\notag\\
  &+\|\Div\RR_{\textup{mol,v}}\|^2_{H^4}  + \|\Div\RR_{\textup{mol,b}}\|^2_{H^4}+\|F_{q+1}\|^2_{L^4}\dd t\notag\\
  \lesssim&  \frac{1}{4} \lambda^{-6}_{q+1} <\lambda^{-6}_{q+1}.
\end{align}
This estimate implies that $T_1=T$ and thus \eqref{H4} is valid  by $\|f\|_{H^s}\approx\|\nabla^sf\|_{L^2 }+\| f\|_{L^2 },~~s>0$.

For each $\gamma$ with $ |{\gamma}|=N $, one deduces that
\begin{align*}
&\frac{1}{2}\frac{\dd}{\dd t}\int_{\mathbb{T}^2}\|\partial^{\gamma}\theta_{q+1}\|^2_{L^2}\dd x + \int_{\mathbb{T}^2}\partial^{\gamma}(\theta_{q+1}\cdot\nabla \theta_{q+1}+\vnon_{\ell_q}\cdot\nabla \theta_{q+1}+\theta_{q+1}\cdot\nabla \vnon_{\ell_q})\cdot\partial^{\gamma}\theta_{q+1} \dd x \\
  =&   \int_{\mathbb{T}^2}  \partial^{\gamma}(d_{q+1}\cdot\nabla d_{q+1}+b_{\ell_q}\cdot\nabla d_{q+1}+d_{q+1}\cdot\nabla b_{\ell_q}+ \Div\RR_{\textup{mol,v}}+F_{q+1})\cdot\partial^{\gamma}\theta_{q+1}\dd x,
  \end{align*}
  and
  \begin{align*}
&\frac{1}{2}\frac{\dd}{\dd t}\int_{\mathbb{T}^2}|\partial^{\gamma}d_{q+1}|^2\dd x +  \int_{\mathbb{T}^2}(\theta_{q+1}\cdot\nabla d_{q+1} + \vnon_{\ell_q}\cdot\nabla d_{q+1}+\theta_{q+1}\cdot\nabla b_{\ell_q})\cdot\partial^{\gamma}d_{q+1}\dd x\\
      =& \int_{\mathbb{T}^2}\partial^{\gamma}(d_{q+1}\cdot\nabla \theta_{q+1}+b_{\ell_q}\cdot\nabla \theta_{q+1}+d_{q+1}\cdot\nabla \vnon_{\ell_q} + \Div\RR_{\textup{mol,b}})\cdot\partial^{\gamma}d_{q+1}\dd x.
\end{align*}
Following the method as in estimating \eqref{est-H4}, we have
\begin{align}
\frac{1}{2} (\|\theta_{q+1}\|^2_{\dot H^N}+\|d_{q+1}\|^2_{\dot H^N})
\lesssim&
 \int_{ 0}^{T}  \|(b_{\ell_q}, \vnon_{\ell_q}, \theta_{q+1},d_{q+1})\|_{H^3}(\|\theta_{q+1}\|^2_{H^N}+\|d_{q+1}\|^2_{H^N})\notag\\
  &
 +\|(d_{q+1},\theta_{q+1})\|_{L^2}\|(\vnon_{\ell_q},b_{\ell_q})\|_{W^{N+1,\infty}}\|(d_{q+1},\theta_{q+1})\|_{H^N} \notag\\
  &+\|\Div\RR_{\textup{mol,v}}\|^2_{H^N}  + \|\Div\RR_{\textup{mol,b}}\|^2_{H^N}+\|F_{q+1}\|^2_{H^N}\dd t.\notag
\end{align}
Assume that $T_1< T $ is the maximum lifespan such that $$\|  \theta_{q+1}\|^2_{\dot H^N}+\| d_{q+1}\|^2_{\dot H^N}\leq \lambda^{2(N+1)}_{q+1}\delta_{q+1},~~~~4\leq N\leq N_0,~~t\in[0,T_1].$$
Since $\|(\Div\RR_{\textup{mol,v}},\Div\RR_{\textup{mol,b}},F_{q+1})\|_{H^N}\leq \lambda^{N+1-10}_{q+1}$, one can easily deduce that
\begin{align*}
\frac{1}{2} (\|\theta_{q+1}\|^2_{\dot H^N}+\|d_{q+1}\|^2_{\dot H^N})
\le & \frac{1}{16}\lambda^{2(N+1)}_{q+1}\delta_{q+1}
+C\lambda^{-10}_{q+1}(\lambda^{N+3}_{q}\delta^{1/2}_{q})
\lambda^{N+1}_{q+1}\delta^{1/2}_{q+1}+ C\lambda^{2N-18}_{q+1}\\
\lesssim&\frac{1}{4}\lambda^{2(N+1)}_{q+1}\delta_{q+1}<\lambda^{2(N+1)}_{q+1}\delta_{q+1},
\end{align*}
and so we  shows \eqref{HN}.

Finally, we can deduce  by \eqref{e:wt} that
\begin{align}\label{dtwnon}
\| P^{(\textup{per})}_{q+1}\|_{N}+\|  \partial_t\theta_{q+1}\|_{N}+\|  \partial_td_{q+1}\|_{N} \le \lambda^{-3}_{q+1},\quad~N=0,1.
\end{align}
This inequality implies that, for $N=0,1$,
 \begin{align*}
 &\|(D_{t,\ell_q}\theta_{q+1},D_{t,\ell_q}d_{q+1})\|_{N}
 \le \| (\partial_t\theta_{q+1},\partial_td_{q+1},u_{\ell_q}\cdot\nabla\theta_{q+1},u_{\ell_q}\cdot\nabla b_{q+1})\|_{N}\le \lambda^{-3}_{q+1},
 \end{align*}
 which yields \eqref{Dtnon}. Thus, we complete the proof of Proposition \ref{est-GMHD}.
\end{proof}

\begin{prop}[Estimates for $\Rrem$]\label{propRem}${\Trem}$ can be decomposed by
$$ {\Trem}=\Div \Rrem+\nabla {\Prem}$$ with $\text{spt}_x (\Rrem ,{\Prem})\subset \Omega_{q+1}$, and
 \begin{align*}
 &\|{\Prem}\|_0\leq\delta_{q+1}\\
&\|\Rrem\|_N\leq\lambda^{N}_{q+1} \tau_q\lambda_q \delta^{1/2}_{q}\delta_{q+1},~~&&N=0,1,2,\\
&\|D_{t,q+1}\Rrem\|_{N-1}\leq \lambda^{ N-1}_{q+1}\tau^{-1}_q\delta_{q+1},~~&&N=1,2.
\end{align*}
\end{prop}
\begin{proof}
We rewrite $\Trem$ to be
\begin{align*}
{\Trem}=& \Div \Big(\vloc_{\ell_q}\ootimes \wnon_{q+1}+\wnon_{q+1}\ootimes \vloc_{\ell_q}+\wloc_{q+1}\ootimes \wnon_{q+1}+\wloc_{q+1}\ootimes \wnon_{q+1} +\RR_{\ell_q}-\RR_{\ell_q,\tau_q} \Big)\\
&+\frac13\nabla {Tr}\Big(\vloc_{\ell_q}\otimes \wnon_{q+1}+\wnon_{q+1}\otimes \vloc_{\ell_q}+\wloc_{q+1}\otimes \wnon_{q+1}+\wloc_{q+1}\otimes \wnon_{q+1} \Big)\\
=:& \Div\Rrem+\nabla {\Prem}.
\end{align*}
With the aid of Proposition \ref{p:estimates-for-mollified}, we deduce that
\begin{align*}
&\|{\Prem}\|_0\lesssim\delta_{q+1}, \,\,\|\RR_{\ell_q}-\RR_{\ell_q,\tau_q}\|_0\lesssim \tau_q\lambda_q \delta^{1/2}_{q}\delta_{q+1},\\
&\|\RR_{\ell_q}-\RR_{\ell_q,\tau_q}\|_N\leq
\|\RR_{\ell_q}\|_N+\|\RR_{\ell_q,\tau_q}\|_N
\lesssim  \ell^{-N}_q\delta_{q+1}\lesssim \lambda^{N}_{q+1} \tau_q\lambda_q \delta^{1/2}_{q}\delta_{q+1},~~N= 1,2,
\end{align*}
 and, for $~N=1,2$,
\begin{align*}
&\|D_{t,q+1}(\RR_{\ell_q}-\RR_{\ell_q,\tau_q})\|_{N-1}  \\
\leq
&\|D_{t,q+1} \RR_{\ell_q}\|_{N-1}+\|D_{t,q+1}\RR_{\ell_q,\tau_q}\|_{N-1}  \\
\leq
&\big\|D_{t,\ell_q}\RR_{\ell_q }\big\|_{N-1}+\Big\| ( u_q- u_{\ell_q}) \cdot\nabla\RR_{\ell_q }\Big\|_{N-1}+\|w_{q+1} \cdot\nabla\RR_{\ell_q }\|_{N-1}\\
& +\big\|D_{t,\ell_q}\RR_{\ell_q,\tau_q}\big\|_{N-1}+\Big\| ( u_q- u_{\ell_q}) \cdot\nabla\RR_{\ell_q,\tau_q}\Big\|_{N-1}+\|w_{q+1}\cdot \nabla\RR_{\ell_q,\tau_q}\|_{N-1}\\
\lesssim&
 \lambda^{ N-1}_{q+1}\tau^{-1}_q\delta_{q+1}
+\lambda^{ N-1-10}_{q+1}\ell^{-1}_q \delta_{q+1}
+\lambda^{ N-1}_{q+1}\lambda_q\delta^{3/2}_{q+1}\\
\lesssim  &  \lambda^{ N-1}_{q+1}\tau^{-1}_q\delta_{q+1}.
\end{align*}
Combining Proposition \ref{p:estimates-for-mollified} with \eqref{H4} and \eqref{dtwnon}, we have for $N=0,1,2$,
\begin{align*}
&\|\vloc_{\ell_q}\ootimes \wnon_{q+1}+\wnon_{q+1}\ootimes \vloc_{\ell_q}+\wloc_{q+1}\ootimes \wnon_{q+1}+\wloc_{q+1}\ootimes \wnon_{q+1} \|_{N }\\
\leq & \lambda^{N-3}_{q+1}\leq  \lambda^{N}_{q+1}\tau_q\lambda_q \delta^{1/2}_{q}\delta_{q+1}
\end{align*}
and, for $N=1,2$,
\begin{align*}
&\|D_{t,q+1}(\vloc_{\ell_q}\ootimes \wnon_{q+1}+\wnon_{q+1}\ootimes \vloc_{\ell_q}+\wloc_{q+1}\ootimes \wnon_{q+1}+\wloc_{q+1}\ootimes \wnon_{q+1})\|_{N-1}  \\
\leq
&\|(D_{t,\ell_q}+(v_{q}- u_{\ell_q}+w_{q+1})\cdot\nabla)(\vloc_{\ell_q}\ootimes \wnon_{q+1}+\wnon_{q+1}\ootimes \vloc_{\ell_q}+\wloc_{q+1}\ootimes \wnon_{q+1}+\wnon_{q+1}\ootimes\wloc_{q+1}) \|_{N-1}  \\
\leq&
  \lambda^{ N-1-3}_{q+1}\leq
  \lambda^{ N-1}_{q+1}\tau^{-1}_q\delta_{q+1}.
\end{align*}
We complete  the proof of  Proposition \ref{propRem}.
\end{proof}

\begin{prop}[Energy estimate]
\label{p:energy}For all $t\in[0,T]$, we have
    \[\bigg|e(t)-\int_{\mathbb T^2}\big(  | u_{q+1}|^{2}+|b_{q+1}|^{2}\big)\dd x -\frac{\delta_{q+2}}2\bigg| \le \frac{\delta_{q+2}}{10} .\]
\end{prop}
\begin{proof}
For $t\in[0,T]$, the total energy error can be rewritten as follows:
\begin{align}\label{energy0}
 &e(t)-\int_{\mathbb T^2}\big( |  u_{q+1}|^{2}+| b_{q+1}|^{2} \big)\dd x -\frac{\delta_{q+2}}2\notag\\
=&e(t)-\int_{\mathbb T^2} \big(| u_q|^{2}+|b_q |^{2}\big) \dd x -\frac{\delta_{q+2}}2-\int_{\mathbb T^2} |\wpq|^2 \dd x -e_{{\rm low},0} ,
\end{align}
where%{energy:v_ell}
\begin{align*}
e_{{\rm low},0}:=&\int_{\mathbb T^2}  2\Big( (\wpq+\wcq)\cdot(\wnon_{q+1}+ u_{\ell_q}) +\wpq\cdot\wcq +\wnon_{q+1}\cdot  u_{\ell_q}+d_{q+1}\cdot b_{\ell_q}\\
&+ \wcq\cdot\wcq+\wnon_{q+1}\cdot\wnon_{q+1}+ d_{q+1}\cdot d_{q+1} - (| u_q|^{2}+|b_q |^{2})+ (v^{2}_{\ell_q }+b^{2}_{\ell_q } )    \Big)\dd x.
\end{align*}
Using geometric Lemma \ref{first S}, one has
\begin{align}\label{energy2}
     \int_{\mathbb T^2}  |\wpq|^2\dd x
    =& \int_{\mathbb T^2} \tr\big(\wpq\otimes \wpq\big) \dd x \notag\\
    =&\int_{\mathbb T^2}\tr\Big(\sum_{i;k\in\Lambda_i}a^2_{k,i}
(\bar{k}\otimes\bar{k}) +  \sum_{\substack{k'+k\neq0\\ i;k\in\Lambda_{i}\\ i';k'\in\Lambda_{i'}}} \bar{a}_{k,i}\bar{a}_{k',i'}e^{\ii 2\pi\lambda_{q+1}(k+k')\cdot x}\bar{k}'\otimes\bar{k} \Big)\dd x\notag\\
    =& \int_{\mathbb T^2} \tr\Big( \chi^2_qe_q\big({\rm Id}-\tfrac{\RR_{\ell_q,\tau_q}}{e_q}\big)
+  \sum_{\substack{k'+k\neq0\\ i;k\in\Lambda_{i}\\ i';k'\in\Lambda_{i'}}}  \bar{a}_{k,i}\bar{a}_{k',i'} e^{\ii2\pi \lambda_{q+1}(k+k')\cdot x}\bar{k}'\otimes\bar{k}\Big) \dd x\notag\\
    =:& 3\int_{\mathbb{T}^2}\chi^2_q\dd x\cdot e_q  +e_{{\rm low},1}
\end{align}
where we have used the fact that  $ \tr \RR_{\ell_q,\tau_q}=0$. From the definition of $e_q(t)$ in \eqref{e0:energy-gap}, we
rewrite \eqref{energy2} as
\begin{align}\label{e:energy-wpq+dpq}
 \int_{\mathbb T^3}   |\wpq|^2\dd x
= e(t) - \int_{\mathbb T^3} \big(|v_{  q}|^2+|b_{ q}|^2\big)\dd x -\frac{\delta_{q+2}}2+e_{{\rm low},1}.
\end{align}
Employing
$$\Big|\int_{\mathbb T^3}f\mathbb{P}_{\geq c}g \dd x\Big|=\Big|\int_{\mathbb T^3}|\nabla|^Lf|\nabla|^{-L}
\mathbb{P}_{\geq c}g \dd x\Big|\leq c^{-L}\|g\|_{L^2}\|f\|_{H^L}$$
with $L$ sufficiently large and using \eqref{e:v_ell-vq} in Proposition \ref{p:estimates-for-mollified}, one can easily deduce that
\begin{align}\label{small terms-elow}
|e_{{\rm low},0}|+|e_{{\rm low},1}|\leq\frac{1}{50 }\delta_{q+2}.
 \end{align}
Putting \eqref{e:energy-wpq+dpq}  and
  \eqref{small terms-elow} into \eqref{energy0}, we have, for each $t\in[0,T]$,
  \begin{align}\label{energy4}
  \Big|e(t)-\int_{\mathbb T^3} \big(|v_{  q+1}|^2+|b_{q+1}|^2\big)\dd x -\frac{\delta_{q+2}}2\Big|
=&\Big|\Big(e(t)-\int_{\mathbb T^3}\big(|v_{  q}|^2+|b_{ q}|^2 \big)\dd x -\frac{\delta_{q+2}}2\Big) -e_{{\rm low},0}\notag\\
&-\Big(e(t) - \int_{\mathbb T^3} \big( |v_{  q}|^2+|b_{ q}|^2\big)\dd x - \frac{\delta_{q+2}}2\Big)-e_{{\rm low},1} \Big|\notag\\
\leq& |e_{{\rm low},0}|+|e_{{\rm low},1}|
\leq\frac{1}{10}\delta_{q+2}.
   \end{align}
Hence, we completethe proof of Proposition \ref{p:energy}.
\end{proof}
\noindent{\textbf{The iterative estimates at next step}} In this part, we combine the previously established estimates to prove that there exists a solution $(u_{q+1},b_{q+1},\RR_{q+1})$ of the equations \eqref{e:subsol-euler} satisfies \eqref{e:vq-C0}--\eqref{e:b-spt}. Note that
\begin{align}\label{q+1step}
 u_{q+1}=&\vnon_{q+1}+\vloc_{q+1}:=(\vnon_{\ell_q}+\theta_{q+1})+(\vloc_{\ell_q}+\wpq+\wcq),~~~~~
b_{q+1} =b_{\ell_q}+d_{q+1},
\end{align}
combining with Propositions \ref{proptrans}--\ref{p:energy}  together, we obtain
 \begin{equation}
\left\{ \begin{alignedat}{-1}
&\partial_t  u_{q+1}+\Div \big( u_{q+1}\otimes  u_{q+1}\big)+\Div (b_{q+1} \otimes b_{q+1})+\nabla P_{q+1} = \Div \RR_{q+1},\\
&\partial_t b_{q+1}  +\Div (  \vnon_{q+1}\otimes b_{q+1})    -\Div (  b_{q+1} \otimes \vnon_{q+1})      = 0,
\end{alignedat}\right.
\end{equation}
and
 \eqref{e:vq-C0}-\eqref{e:Dv},~\eqref{e:energy-q-estimate}-\eqref{e:b-spt} are valid for $q+1$ step,where
 \begin{align*}
 P_{q+1}:=&P_{\ell_q}+P^{\textup{tran}}_{q+1}+ P^{\textup{Nash}}_{q+1}+P^{\textup{osc}}_{q+1}+{\Prem}+P^{(\textup{per})}_{q+1},\\
 \RR_{q+1}:=&\Rtrans+\Rnash +\Rosc+\Rrem.\end{align*}
 Then one deduces that, $\|P_{q+1}\|_0\leq 5\sum_{l=1}^{q+1}\delta_{l}$, and for $N=0,1,2,$
 \begin{align*}
\|\RR_{q+1}\|_N\lesssim\lambda^{N}_{q+1}\tau^{-1}_q\lambda_q\delta^{1/2}_{q}\delta_{q+1} +\lambda^{-1+N}_{q+1}\lambda_q\delta^{1/2}_{q}\delta^{1/2}_{q+1}+
\leq\lambda^{N-4 \alpha}_{q+1}\delta_{q+2} ,
\end{align*}
and, for $N=1,2$,
\begin{align*}
 &\|D_{t,{q+1}}\RR_{q+1}\|_{N-1}\\
 \leq& \|\big(D_{t,{\ell_q}}+ (v_{q}- u_{\ell_q}+w_{q+1})\nabla\big)(\Rtrans+\Rnash +\Rosc) \|_{N-1} +\|D_{t,{q+1}}\Rrem\|_{N-1} \\
 \lesssim& \lambda^{N-1}_{q+1}\tau^{-1}_q\delta_{q+1}\leq \lambda^{N-1}_{q+1}\lambda_{q+1}\delta^{1/2}_{q+1}\delta_{q+2}\lambda_{q+1}^{-4 \alpha},
\end{align*}
 where we use \eqref{beta} such that
$$\lambda^{-1}_{q+1}\lambda_q\delta^{1/2}_{q}\delta^{1/2}_{q+1}\leq\tau_q\lambda_q\delta^{1/2}_{q}\delta_{q+1} \leq \delta_{q+2}\lambda_{q+1}^{-4 \alpha},\quad \text{for}~~1<b\leq1+\tfrac{  1-5\beta }{8\beta},~\beta<\frac15.$$   This fact gives \eqref{e:RR_q-C0}-\eqref{e:DRR_q-C0} for $q+1$ step.

Since $\text{spt}_x a_{k,i} \subset \Omega_{q+1}$, we obtain $\text{spt}_x(\vloc_{q+1},\RR_{q+1})\subset \Omega_{q+1}$. Now we show that $\text{spt}_xb_{q+1} \subset \mathcal{O}$ as in \eqref{spt}. Firstly, we rewrite the equation of $b_{q+1}$ such that
\begin{align*}
\partial_t b_{q+1}\circ y(t,x)        =-   (b_{q+1}\cdot \nabla \vnon_{q+1})\circ y(t,x),\quad \partial_ty = \vnon_{q+1}\circ y~~~\text{with} ~~~y(0,x)=x.      \end{align*}
Taking the $L^2$ inner product of this equation with $b_{q+1}\circ y(t,x) $, integrating over the domain $\mathcal{D}$ with
\[\mathcal{D}=\bigcup_{k\in\ZZ^2}\Big(\big([0,1]^2\setminus [\tfrac{1}{2}-\tfrac{1}{4}\epsilon_0  ,  \tfrac{1}{2}+\tfrac{1}{4}\epsilon_0  ]^2\big)+k\Big),\]
 and using  Gr\"{o}nwall's inequality, we deduce from  $\|\vnon_{q+1}\|_{H^4} \sim o(\epsilon_0)$ that
 $$\int_{\mathcal{D}}|b_{q+1}\circ y(t,x)|^2\dd x\lesssim \int_{ \mathcal{D}}|b_{q+1} (0,x)|^2\dd x.$$
The relation \eqref{e:b-spt} implies that $b_{q } (0,x)=b_{q-1}*_x\psi_{\ell_{q-1}}(0,x)$. Together with $d_{q+1}(0,x)=0$, it yields that
 \begin{align*}
 b_{q+1} (0,x)=&b_{\ell_q}(0,x)=b_{ q}*_x\psi_{\ell_q}(0,x)\\
 =&b_{ q-1}*_x\psi_{\ell_{q-1}}*_x\psi_{\ell_q}(0,x)=...=b_1*_x\psi_{\ell_{1}}*_x\psi_{\ell_2}
 *_x...*_x\psi_{\ell_q}(0,x).
 \end{align*}
Recalling that$ ~\text{spt}_xb_1(0,x)\subset \bigcup_{k\in\ZZ^2}  \big(  [\tfrac{1}{2}-\tfrac{1}{16}\epsilon_0  ,  \tfrac{1}{2}+\tfrac{1}{16}  ]^2 +k\big)\subsetneqq  \mathcal{O}$ , we deduce that
 $$\text{spt}_xb_{q+1} (0,x)\subset  \bigcup_{k\in\ZZ^2}  \Big([\tfrac{1}{2}-\tfrac{1}{8}\epsilon_0  ,  \tfrac{1}{2}+\tfrac{1}{8}\epsilon_0  ]^2+k\Big)\quad \text{by}\quad \sum_{l=1}^{q}\ell_l\leq\sum_{l=1}^{+\infty}\ell_l\sim o(\epsilon_0).$$
 Therefore, one gets
 \[\int_{ \mathcal{D}}|b_{q+1}\circ y(t,x)|^2\dd x=0,\]
which implies that
$$b_{q+1}\circ y(t,x)\subset \bigcup_{k\in\ZZ^2}\Big(  [\tfrac{1}{2}-\tfrac{1}{4}\epsilon_0  ,  \tfrac{1}{2}+\tfrac{1}{4}  ]^2 +k\Big).$$
Note that $y(t,x)$ is a differential homeomorphic mapping  such that $$y(t,x)=x+\vnon_{q+1}\circ y\subset [x-o(\epsilon_0),x+o(\epsilon_0)],$$
we deduce that
 $$\text{spt}_x ~b_{q+1}\subset \bigcup_{k\in\ZZ^2}\Big( [\tfrac{1}{2}-\tfrac{2}{4}\epsilon_0  ,  \tfrac{1}{2}+\tfrac{2}{4}\epsilon_0  ]^2 +k\Big)=\mathcal{O},$$
and finish the proof of Proposition \ref{p:main-prop00}.

\subsubsection{Proof of Proposition \ref{p:main-prop2}}
Under the assumptions that
\begin{align*}
  u_{q}=\widetilde{u}_{q},\quad b_{q }=\widetilde{b}_{q},  \quad \RR_{q } ={\tRR}_{q },\quad\text{on}\quad [0, \tfrac{T}{4}+5\tau_q],
\end{align*}
one immediately has
\begin{align*}
u_{\ell_q}=\widetilde{u}_{\ell_q},~~b_{\ell_q}=\widetilde{b}_{\ell_q},  \quad \Phi_i=\widetilde{\Phi}_i,\quad \RR_{\ell_q,\tau_q} =\tRR_{\ell_q,\tau_q},\quad\text{on}\quad [0, \tfrac{T}{4}+4\tau_q].
\end{align*}
Recalling the definition of $ \wloc_{q+1}$ that
\begin{align*}\label{recallpc}
  \wloc_{q+1}
    = \wpq +\wcq
    = \sum_{i;k\in\Lambda_i} \frac{\nabla^{\perp}}{ \lambda_{q+1} }  \big(\bar{a}_{k,i} e^{ 2\pi \ii\lambda_{q+1}k\cdot x}\big),
\end{align*}
where $$\bar{a}_{k,i}=e^{1/2}_q\eta_i\chi_{q}   a_k \Big({\rm Id}-\tfrac{\RR_{\ell_{t,x},\tau_q}}{e_q}\Big) e^{2\pi \ii \lambda_{q+1}k\cdot(\Phi_i-x)},$$  we easily deduces that  $\bar{a}_{k,i}=\widetilde{\bar{a}}_{k,i}$ on $ [0, \tfrac{T}{4}+2\tau_q]$, and thus $\wloc_{q+1}=\widetilde{\wloc}_{q+1} $ on $[0, \tfrac{T}{4}+2\tau_{q }]$.

Moreover, recalling the definition of $\RR_{\textup{mol,v}},~\RR_{\textup{mol,b}}$ and $F_{q+1}$, we deduce that
$$\RR_{\textup{mol,v}}=\widetilde{\RR}_{\textup{mol,v}},~~\RR_{\textup{mol,b}}=\widetilde{\RR}_{\textup{mol,b}},
~~F_{q+1}=\widetilde{F}_{q+1},~~~\text{on}\quad [0, \tfrac{T}{4}+2\tau_q].$$
These facts imply that
$$\theta_{q+1}=\widetilde{\theta}_{q+1},\quad d_{q+1}=\widetilde{d}_{q+1}.$$
Combining with \eqref{q+1step} yields that
$$ u_{q+1}=\widetilde{u}_{q+1},\quad b_{q+1}=\widetilde{b}_{q+1},\quad t\in [0, \tfrac{T}{4}+5\tau_{q+1}].$$
Finally, we easily infer from the  definition of $\RR_{q+1}$  that
$$\RR_{q+1} =\tRR_{q+1},~~~t\in[0, \tfrac{T}{4}+5\tau_{q+1}].$$
Hence, we complete the proof of Proposition \ref{p:main-prop2}.

\subsection{Proof of Theorem \ref{app}}For each incompressible Euler weak solution $u_{E}\in C^{\gamma}_{t,x}~(\gamma>0)$ on $[0,T]$,  we could construct a sequence of approximate solutions for the ideal MHD equations via the   proof of Proposition \ref{t:main00}. The detailed proof proceeds as follows:

 Set $\epsilon=\lambda^{-\frac{1}{20}}_1$,  we introduce the spatial periodic function $\chi_{u}(x)$ satisfying
\begin{align*}
\textup{spt}_x\chi_{u}\subset
\Big(\big([0,1]^2\setminus [\tfrac{1}{2}-\tfrac{4}{5}\epsilon  ,  \tfrac{1}{2}+\tfrac{4}{5}\epsilon]^2\big)+k\Big).
\end{align*}
Let
$$u^{(1)}_{\epsilon}:=(\chi_{u} u_{E} )_{\epsilon^{ \frac{1}{20}}}=(\chi_{u}u_E)*_{x}\psi_{\epsilon^{ \frac{1}{20}}}*_{t}\varphi_{\epsilon^{ \frac{1}{20}}}\quad\textup{and}\quad
u^{(2)}_{\epsilon}:=((1-\chi_{u}) u_{E} )_{\epsilon^{ \frac{1}{20}}}.$$
One verifies that
$$\partial_tu^{(1)}_{\epsilon}+\Div(u^{(1)}_{\epsilon}\otimes u^{(1)}_{\epsilon})+\nabla P^{(1)}_{\epsilon}=
\Div R^{(1)}_{\epsilon} +F_{\epsilon},$$
  where $R^{(1)}_{\epsilon}:= u^{(1)}_{\epsilon}\ootimes u^{(1)}_{\epsilon}-(\chi_{u}   u_E\ootimes\chi_{u}   u_E)_{\epsilon^{ \frac{1}{20}}}$ and
  $$F_{\epsilon}:=\partial_tu^{(2)}_{\epsilon}+\Div\Big(( 1-\chi_{u})  u_{E} \ootimes  (1-\chi_{u}) u_{E}+ (1-\chi_{u})  u_{E} \ootimes  \chi_{u}   u_{E} +  \chi_{u} u_{E} \ootimes  (1-\chi_{u})  u_{E}\Big)_{\epsilon^{ \frac{1}{20}}}.$$
  Since $\|(1-\chi_{u})u_{E}\|_{L^1}\lesssim\epsilon^{2}\|u_{E}\|_0\lesssim\epsilon^{2}$, we obtain that
  $$\|F\|_{H^4}\leq \|F\|_{W^{5,1}}\lesssim\epsilon^{2-\frac{6}{20}}\quad\textup{and}\quad
   \|R^{(1)}\|_0\lesssim\|u^{(1)}\|^2_{\gamma}\lambda^{-\frac{\gamma}{200}}_1
  \lesssim\delta_{2}\lambda^{-3\alpha}_{1},$$
  where we have used $\delta_{q }:=c\lambda_q^{-2\beta}$ in \eqref{delta} with  $c= \lambda^{ 3\alpha-\frac{\gamma}{200}}_{1}\lambda_2^{2\beta},~~\alpha\ll\gamma$.

We consider the following Cauchy problem:
 \begin{equation}
\left\{ \begin{alignedat}{-1}
&\partial_t \theta_{\epsilon}   + \theta_{\epsilon} \cdot\nabla \theta_{\epsilon}  +\nabla P^{(\textup{per})}_{\epsilon}
   =     d_{\epsilon}\cdot\nabla d_{\epsilon} -F_{\epsilon} ,\\
&\partial_t d_{\epsilon}  + \theta_{\epsilon} \cdot\nabla d_{\epsilon}
      =  d_{\epsilon} \cdot\nabla \theta_{\epsilon}  ,
\\
 & \Div \theta_{\epsilon}  = \Div d_{\epsilon}  = 0,
  \\
  & (\theta,~d) |_{t=0}=  (0,g(x)) ,
\end{alignedat}\right.
 \label{again-e:wt}
\end{equation}
where $g(x)$ is a divergence free periodic  vector field satisfying $$\|g\|_{H^{N_0}}\leq\lambda^{-5}_1\quad \textup{and}\quad \textup{spt}_xg\subset[\tfrac{1}{2}-\tfrac{1}{32}\epsilon ,  \tfrac{1}{2}+\tfrac{1}{32}\epsilon ]^2.$$
Since $\lambda_1\gg  \epsilon^{-\frac{1}{20}} \gg T $, the system \eqref{again-e:wt} possesses a unique solution $(\theta_{\epsilon},d_{\epsilon}) $ on $[0,T]$ such that
$$\|\theta_{\epsilon}\|_{H^4}+\|d_{\epsilon}\|_{H^4} \lesssim
    \epsilon^{\frac{3}{2}} \quad \textup{and}\quad \textup{spt}_xd_{\epsilon}\subset\bigcup_{k\in\ZZ^2}\Big( [\tfrac{1}{2}-\tfrac{1}{16}\epsilon  ,  \tfrac{1}{2}+\tfrac{1}{16}\epsilon  ]^2 +k\Big).$$
We set $$(\vloc_{1,\epsilon},\vnon_{1,\epsilon},b_{1,\epsilon}, \RR_{1,\epsilon})=(u^{(1)}_{\epsilon},\theta_{\epsilon},d_{\epsilon},R^{(1)}_{\epsilon}+\vnon_{1,\epsilon}\ootimes \vloc_{1,\epsilon}+\vnon_{1,\epsilon} \ootimes \vloc_{1,\epsilon}),$$
then it is easy to verify that $(\vloc_{1,\epsilon},\vnon_{1,\epsilon},b_{1,\epsilon}, \RR_{1,\epsilon})$ satisfies  \eqref{e:vq-C0}--\eqref{spt} in Proposition \ref{p:main-prop00} with $\epsilon_0=\epsilon$  for $q=1$. In the same way as in the proof of Proposition \ref{p:main-prop00}, we obtain a weak solution $(u_{\epsilon},b_{\epsilon})\in C^{\tfrac{1}{5}-}$ with $u_{\epsilon}=\vloc_{\epsilon}+\vnon_{\epsilon} $ such that
$$ \|\vloc_{\epsilon}-\vloc_{1,\epsilon}\|_0\lesssim \sum_{q=1} \delta^{\frac{1}{2}}_{q+1} \quad\textup{and}\quad\|\vnon_{\epsilon}\|_{H^4}+\|b_{\epsilon}\|_{H^4} \le
    \epsilon^{\frac{3}{2}}+\sum_{q=1} \lambda^{-3}_q. $$
By the definition of $\delta_{q}$ and setting
$\epsilon=\lambda^{-\frac{1}{20}}_1=a^{-\frac{1}{20}b}$,  we deduce that
$$ \|\vloc_{\epsilon}-\vloc_{1,\epsilon}\|_0\lesssim \sum_{q=1} \delta_{q+1}\lesssim\delta_{2}\rightarrow 0 \quad\textup{and}\quad\|\vnon_{\epsilon} \|_{H^4}+\|b_{\epsilon} \|_{H^4} \le
    \epsilon^{\frac{3}{2}}+\sum_{q=1} \lambda^{-3}_q\rightarrow 0,\quad\text{as}\quad\epsilon\to 0.$$
Combining with
    $ \|u_E-\vloc_{1,\epsilon}\|_{L^p}
    \rightarrow 0,\,\,\text{as}\,\,\epsilon\rightarrow 0,$
we obtain that  $$\|u_{\epsilon}-u_E\|_{L^{\infty}_TL^{p}}+\|b_{\epsilon}\|_{L^{\infty}_TL^{p}}\rightarrow 0,\quad \textup{as}\quad \epsilon\rightarrow 0.$$

\quad\\
\appendix
\section{}
\begin{defn}[Mollifiers\cite{BMNV}]\label{Moll}Let  $\widetilde\phi\in C^\infty(\R;\R)$ satisfy $\supp \widetilde\phi\subset[-1,1]$,
\begin{align*}
  \int_{\R}\widetilde\phi(x)\dd x=1,\quad     \int_{\R}\widetilde\phi(x)x^n\dd x=1,\,\,\forall n=1,2,\cdots,L.
\end{align*}
Let $\psi:\R^d \to\R$ be defined by $\psi(x)=\widetilde\phi(|x|)$.
For each $\epsilon>0$, we define the two mollifiers as follows:
\begin{align}
\varphi_{\epsilon}(t)\coloneq \frac1{\epsilon} \psi\Big(\frac{t}\epsilon\Big),\quad \psi_\epsilon(x)
            \coloneq \frac1{\epsilon^d} \psi\Big(\frac{x}\epsilon\Big),\quad   \label{e:defn-mollifier-x}
\end{align}
\end{defn}
\begin{defn}[The operators $\mathcal{R}$ and $\mathcal R_{\vex}$]\label{def.R}
 For each vector field $u\in C^{\infty}(\TTT^3, \R^3)$, the operator $\mathcal{R}$ introduced in \cite{BDIS15} is defined by
\begin{align*}
    \mathcal R u = -(-\Delta)^{-1} (\nabla u + \nabla u^\TT ) - \frac12(-\Delta)^{-2} \nabla^2 \nabla\cdot u  + \frac12 (-\Delta)^{-1} (\nabla\cdot u ){\rm Id}_{3\times 3}.
\end{align*}
It is a matrix-valued right inverse of the divergence operator for the mean-free vector fields, in the sense that
\[ \Div\mathcal R u= u - \dashint_{\mathbb T^3} u. \]
In addition, $\mathcal Ru$ is traceless and symmetric.

%The operator $\mathcal R_{\vex}$ is the well-known Helmholtz solution to the equation of prescribed divergence.
For each mean-free scalar function $f\in C^\infty(\TTT^3, \R)$, we define the operator $R_{\vex}$ by
$$\mathcal R_{\vex} f:={\nabla}{\Delta}^{-1} f.$$
It is a vector-valued right inverse of the divergence operator for the mean-free scalar  functions, in the sense that
\[ \Div\mathcal R_{\vex} f= f - \dashint_{\mathbb T^3} f. \]
\end{defn}

\begin{lem}[Stationary flows in 2D\cite{CDS}]\label{Betrimi}
Given  $ \nu\geq1$ and $\Lambda\subset\nu\mathbb{S}^2\cap\mathbb{Z}^3$ such that $-\Lambda=\Lambda$. Then for any $b_k\in\mathbb{R}$ with $b_{k}=b_{-k}$, the vector flied
$$ W(\xi)=\sum_{k\in\Lambda}b_k\frac{\ii k^{\perp}}{|k|}e^{2\pi \ii k\xi} $$
is real-valued, divergence-free and satisfies
$$\Div_{\xi}(W\otimes W)= \frac{1}{2}\nabla_{\xi}(|W|^2-\nu^2\psi^2)$$
and \begin{align}\label{b-g}
W\otimes W=&\sum_{j,k\in\Lambda,~j+k\neq0}-b_j b_k e^{\ii (j+k)\xi}\frac{  j^{\perp}}{|j|}\otimes \frac{  k^{\perp}}{|k|}+\sum_{j,k\in\Lambda,~j+k=0} b^2_k\frac{  k^{\perp}}{|k|}\otimes \frac{  k^{\perp}}{|k|}\notag\\
=&\sum_{j,k\in\Lambda,~j+k\neq0}-\frac{b_k}{|k|}\frac{b_j}{|j|}  e^{\ii (j+k)\xi} {  \bar{j}} \otimes  {  \bar{k}} +\sum_{j,k\in\Lambda,~j+k=0} \Big(\frac{b_k}{|k|}\Big)^2   \bar{k}\otimes \bar{k},
%=&\sum_{j,k\in\Lambda,~j+k\neq0}-a_j a_k{j}e^{\ii (j+k)\xi}\frac{  j^{\perp}}{|j|}\otimes \frac{  k^{\perp}}{|k|}+\sum_{k\in\Lambda} a^2_k(Id- \frac{  k }{|k|}\otimes \frac{  k }{|k|}),
\end{align}
where $\bar{k}:=k^{\perp}$ and $\bar{j}:=j^{\perp}$.
%where $\frac{  k^{\perp}}{|k|}\otimes \frac{  k^{\perp}}{|k|}=Id-\frac{  k }{|k|}\otimes \frac{  k }{|k|}$.
\end{lem}

\begin{lem}[{Geometric Lemma}\cite{CDS}]\label{first S}Let $\epsilon>0$ and $B_{\sigma}(0)$ denote the ball of radius $\sigma$ centered at $\rm Id$ in the space of $2\times2$ symmetric matrices. For every $N\in\NN$, we can choose $r_0$ and $\nu\ge1$ with the following property. There exist pairwise disjoint subsets
\begin{align*}
\Lambda_i\subset \{k\in \ZZ^2| |k|=\nu\}, \quad i\in\{1,2,\cdots,N\}
\end{align*}
that consists of vectors $k$ with associated orthonormal basis $(k,\bar{k} ):=(k, {k}^{\perp} )$  and smooth function $a_{k}:B_{r_0}(\rm Id)\rightarrow\mathbb{R}$ such that for each $R\in B_{r_0}(\rm Id)$, we have the following identity:
$$R=\sum_{k\in\Lambda_i}a^2_{k}(R) \bar{k}\otimes\bar{k} ,\quad i=1,2$$
Furthermore,  we have $\Lambda_i=-\Lambda_i$ and $a_{k}=a_{-k}$. For $0\le K\le L+1$, there exists a constant $M$ such that
\begin{equation}\label{M}
\|a_{k}\|_{C^K(\bar{B}_{r_0}({\rm Id}))}\le \tfrac{M}{100}.
\end{equation}
\end{lem}

\begin{lem}[\cite{BDSV, TZ18}]\label{l:non-stationary-phase} {If $a\in C^\infty(\mathbb T^2;\mathbb R^2)$}, $b\in C^\infty(\mathbb T^2;\mathbb R)$ and $\Phi\in  C^\infty (\mathbb T^2 ; \mathbb R^2)$ satisfying  $|\nabla \Phi|  \sim 1$, then
\begin{align*}
\|\mathcal R( a \ee^{\ii k\cdot \Phi}) \|_{C^\alpha}
    \lesssim  \frac{\|a\|_{C^0}}{|k|^{1-\alpha}} + \frac{\|a\|_{C^{N+\alpha}} + \|a\|_{C^0} \|\Phi\|_{C^{N+\alpha}}} {|k|^{N-\alpha}},\\
    \|\mathcal R_{\vex}( b \ee^{\ii k\cdot \Phi}) \|_{C^\alpha}
    \lesssim  \frac{\|b\|_{C^0}}{|k|^{1-\alpha}} + \frac{\| b\|_{C^{N+\alpha}} + \|b\|_{C^0} \|\Phi\|_{C^{N+\alpha}}} {|k|^{N-\alpha}}.
\end{align*}
\end{lem}

\begin{lem}[Inverse divergence iteration step \cite{BMNV}]\label{tracefree}
Let $\lambda \in \mathbb{N}^{+}$ and $\{\rho^{(n)}\}_{0 \leq n \leq N}$ be zero-mean smooth $\mathbb{T}^2$-periodic functions such that $\rho^{(n)} = \Delta \rho^{(n+1)}$. Let $\Phi$ be a volume-preserving transformation of $\mathbb{T}^2$ satisfying $\|\nabla \Phi - \rm{Id}\|_{0} \leq 1/2$, and define the matrix $A = (\nabla \Phi)^{-1}$. Then, for any smooth vector field $G$, we have
\begin{align}\label{div-nabla}
    G^i \rho \circ (\lambda \Phi) = \partial_j \mathring{R}^{ij} + \partial_i P + E^i,
\end{align}
where the traceless symmetric stress $\mathring{R}$ is given by
$$
\mathring{R}^{ij} = \lambda^{-1} \big(G^i A^j_l + G^j A^i_l - A^j_k A^i_k G^p \partial_p \Phi^l \big) (\partial_l \rho^{(1)}) \circ (\lambda \Phi) =: \lambda^{-1} A^{ij}_l(G) (\mathcal{L}_l \rho^{(0)}) \circ (\lambda \Phi),
$$
the pressure term $P$ is given by
$$
P = \lambda^{-1} \big(2 G^j A^j_l - A^j_k A^j_k G^p \partial_p \Phi^l \big) (\partial_l \rho^{(1)}) \circ (\lambda \Phi) =: \lambda^{-1} B_l(G) (\mathcal{L}_l \rho^{(0)}) \circ (\lambda \Phi),
$$
and the error term $E^i$ is given by
$$
E^i = \big(\partial_j (G^p A^i_k A^j_k - G^j A^i_k A^p_k) \partial_p \Phi^l - \partial_j G^i A^j_l \big) (\partial_l \rho^{(1)}) \circ (\lambda \Phi) =: \lambda^{-1} C^i_l(G) (\mathcal{L}_l \rho^{(0)}) \circ (\lambda \Phi),
$$
where $\mathcal{L}_l f= \partial_l (\Delta)^{-1}f$. Moreover, if $G$ is supported on a bounded domain $\Omega \subset \mathbb{T}^2$, then $\mathring{R}$, $P$, and $E$ are also compactly supported on $\Omega$.

Applying the relation \eqref{div-nabla} to the error term $E^i$ iteratively for $N$ steps, we obtain
\begin{align*}
G^i \rho \circ (\lambda \Phi) =& \partial_j \bigg( \sum_{n=1}^N \lambda^{-n} A^{(n)}_{ij \alpha_n}(G) \big(\mathcal{L}_{\alpha_n} \rho^{(0)}) \circ (\lambda \Phi) \bigg) + \partial_i \bigg( \sum_{n=1}^N \lambda^{-n} B^{(n)}_{\alpha_n}(G) \big(\mathcal{L}_{\alpha_n} \rho^{(0)}) \circ (\lambda \Phi) \bigg)\\
&+ \lambda^{-N} C^{(N)}_{i \alpha_N}(G) \big(\mathcal{L}_{\alpha_N} \rho^{(0)}) \circ (\lambda \Phi),
\end{align*}
where $\mathcal{L}_{\alpha_n} f = \partial^{\alpha_n} (\Delta^{-1})^n f$. Here, $A^{(n)}_{ij \alpha_n}$, $B^{(n)}_{\alpha_n}$, and $C^{(N)}_{i \alpha_N}$ depend linearly on $G$, and the following estimates hold:
$$
\|(A^{(n)}_{ij \alpha_n}(G), B^{(n)}_{\alpha_n}(G))\|_0 \leq C^n \|G\|_{n-1} \quad \text{and} \quad \|C^{(N)}_{i \alpha_N}(G)\|_0 \leq C^N \|G\|_N,
$$
for some universal constant $C$. Additionally, the supports of $A^{(n)}_{ij \alpha_n}(G)$, $B^{(n)}_{\alpha_n}(G)$, and $C^{(N)}_{i \alpha_N}(G)$ are contained in $\supp G$.
\end{lem}


\begin{thebibliography}{99}
\bibitem{Alu09}
H. \textsc{Aluie}, Hydrodynamic and magnetohydrodynamic turbulence: invariants, cascades,
 and locality. Doctoral Dissertation, Johns Hopkins University, 2009.




\bibitem{BBV}
R. \textsc{Beekie}, T. \textsc{Buckmaster}, and V. \textsc{Vicol},\textit{ Weak solutions of ideal mhd which do not conserve magnetic helicity}, Annals of PDE 6 (1) (2020), Paper No. 1, 40 pp.


\bibitem{BLN}
A. C. \textsc{Bronzi}, M. C. \textsc{Lopes Filho}, and H. J. \textsc{Nussenzveig Lopes}, \textit{Wild solutions for 2d incompressible ideal
flow with passive tracer}, Comm. Math. Sci. 13
 (5) (2015),1333--1343.



\bibitem{BCD}
E. \textsc{Bru\`{e}}, M.\textsc{Colombo}, C. \textsc{DeLellis}, \emph{Positive solutions of transport equations and classical nonuniqueness of characteristic curves}, Arch. Ration. Mech. Anal.,  240 (2) (2021), 1055--1090,


\bibitem{B15}
 T. \textsc{Buckmaster},  \emph{Onsager's conjecture almost everywhere in Time}, Comm. Math.
Phys., 333 (3) (2015), 1175--1198.



\bibitem{BCV}
 T. \textsc{Buckmaster}, M. \textsc{Colombo}, V. \textsc{Vicol},  \emph{Wild solutions of the   Navier-Stokes equations whose singular
sets in time have hausdorff dimension strictly less than 1}, J. Eur. Math. Soc., 24 (9) (2021), 3333--3378.



\bibitem{BDIS15}
T. \textsc{Buckmaster}, C. \textsc{De Lellis}, P. \textsc{Isett} , L.  \textsc{ Sz\'{e}kelyhidi Jr.}, \emph{Anomalous
dissipation for 1/5-H\"{o}lder Euler flows}, Ann. of Math., \emph{(2)}, 182 (1) (2015), 127--172.





\bibitem{BDS}
T. \textsc{Buckmaster}, C. \textsc{De Lellis}, L. \textsc{Sz\'{e}kelyhidi Jr.}, \emph{Dissipative Euler flows with Onsager'critical spatial regularity}, Commun. Pur.  Appl.  Math., 69 (9) (2016), 1613--1670.

\bibitem{BDSV}
T. \textsc{Buckmaster}, C. \textsc{De Lellis}, L. \textsc{Sz\'{e}kelyhidi Jr.},
V. \textsc{Vicol}, \emph{Onsager's conjecture for admissible weak solutions},  Commun. Pure Appl. Math., 72  (2) (2019), 229--274.









\bibitem{BV21}
T. \textsc{Buckmaster}, V. \textsc{Vicol},
\textit{Convex integration constructions in hydrodynamics},
Bull. Amer. Math. Soc. (N.S.) 58 (1) (2021), 1-44.



\bibitem{BMNV}
T. \textsc{Buckmaster}, N. \textsc{Masmoudi}, M \textsc{Novack}, V. \textsc{Vicol}, \textit{Non-conservative $H^{\frac{1}{2}-}$ weak solutions of the incompressible 3D Euler equations}, arXiv:2101.09278v2.


\bibitem{CKS}
 R. E. \textsc{Caflisch}, I. \textsc{Klapper}, G. \textsc{Steele},  \textit{Remarks on singularities, dimension and energy dissipation for ideal hydrodynamics and MHD},  Commun. Math. Phys. 184 (2) (1997) 443-455.


\bibitem{CL21}
A.  \textsc{Cheskidov},  X. \textsc{Luo},  \emph{Nonuniqueness of weak solutions for the transport equation at critical space
 regularity}, Ann. PDE, 7 (1) (2021), Paper No. 2, 45 pp.


\bibitem{CL}
 A. Cheskidov, X. Luo,  \textit{Sharp nonuniqueness for the Navier-Stokes equations}, Invent. Math. 229  (2022)
 987-1054.




\bibitem{CL22}
A.   \textsc{Cheskidov}, X. \textsc{Luo},  \emph{Extreme temporal intermittency in the linear Sobolev transport: almost
 smooth nonunique solutions}, Anal. PDE 17 (6) (2024), 2161-2177.

\bibitem{CDS}
A. \textsc{Choffrut}, C. \textsc{De Lellis}, L. \textsc{Szz\'{e}kelyhidi Jr}, \textit{Dissipative continuous Euler flows in two and three dimensions}, (2012), arXiv:1205.1226 .


\bibitem{CWT}
 P. \textsc{Constantin}, E. \textsc{Weinan}, and E.S. \textsc{Titi}, \textit{Onsager's conjecture on the energy conservation for solutions of
euler's equation}, Comm. Math. Phys., 165 (1) (1994), 207-209.


\bibitem{Dai}
M. \textsc{Dai}, \textit{Non-uniqueness of Leray-Hopf weak solutions of the 3d Hall-MHD system}, SIAM J. Math. Anal., 53 (5) (2021) 5979-6016.



\bibitem{DSS}
F.  \textsc{Daniel}, L.  \textsc{Sauli}, L. \textsc{Sz\'{e}kelyhidi Jr.},
\textit{Magnetic helicity, weak solutions and relaxation of ideal MHD}, Commun. Pure Appl. Math., 77 (4) (2024), 2181-2576.



\bibitem{DS09}
C. \textsc{De Lellis}, L. \textsc{Sz\'{e}kelyhidi Jr}, \emph{The Euler equations as a
differential inclusion},  Ann. of Math., \emph{(2)} 170 (3) (2009), 1417--1436.


\bibitem{DS2013}
C. \textsc{De Lellis}, L. \textsc{Sz\'{e}kelyhidi Jr}, \emph{Dissipative continuous
Euler flows}, Invent. Math., 193 (2) (2013), 377--407.




\bibitem{DS14}
C. \textsc{De Lellis}, L. \textsc{Sz\'{e}kelyhidi Jr}, \emph{ Dissipative Euler flows and Onsager's conjecture}, J. Eur. Math. Soc., 16 (7) (2014), 1467--1505.


\bibitem{FLS}
D. \textsc{Faraco}, S. \textsc{Lindberg}, L. \textsc{Sz\'{e}kelyhidi}, \textit{Bounded solutions of ideal MHD with compact support in space-time}, Arch. Ration. Mech. Anal., 239 (1) (2021) 1, 51-93.





\bibitem{FL}
 D.  \textsc{Faraco}, S. \textsc{Lindberg}, \textit{Proof of Taylor's conjecture on magnetic helicity conservation}, Comm. Math. Phys., 373 (2020), 707-738.


\bibitem{GLL}
J.-F. \textsc{Gerbeau}, C. \textsc{Le Bris}, T. \textsc{Leli\`{e}vre}, Mathematical methods for the magnetohydrodynamics of liquid metals. Numerical Mathematics and Scientific Computation, Oxford University Press, Oxford, (2006).

\bibitem{GR}
V. \textsc{Giri} and R. O. \textsc{Radu}, \textit{The 2D Onsager conjecture: a Newton-Nash iteration}, Invent. Math., 238 (2) (2024), 691-768.



\bibitem{Ise17}
P. \textsc{Isett}, \textit{H\"{o}lder Continuous Euler Flows in Three Dimensions with Compact Support in Time},  Princeton University Press, (2017).


\bibitem{Ise18}
P. \textsc{Isett}, \emph{ A proof of Onsager's conjecture},  Ann. of Math., \emph{(2)}, 188 (3) (2018), 871--963.



\bibitem{KL07}
E. \textsc{Kang}, J. \textsc{Lee}, \textit{Remarks on the magnetic helicity and energy conservation for ideal magneto-hydrodynamics}, Nonlinearity 20 (11) (2007) 2681-2689.




\bibitem{KMY}
C. \textsc{Khor}, C. \textsc{Miao}, W. \textsc{Ye}, \emph{Infinitely many
non-conservative solutions for the three-dimensional Euler equations with arbitrary initial data in $C^{\frac{1}{3}-\epsilon}$}, 2022, doi.org/10.48550/arXiv.2204.03344.




\bibitem{LZZ}
Y. \textsc{Li}, Z. \textsc{Zeng}, D.  \textsc{Zhang}, \textit{Non-uniqueness of weak solutions to 3D magnetohydrodynamic equations},
J. Math. Pures Appl., 165  (9)  (2022), 232--285.


\bibitem{Luo}
 X. \textsc{Luo},  \emph{Stationary solutions and nonuniqueness of weak solutions for the Navier-Stokes equations in high dimensions}, Arch. Ration. Mech. Anal. 233 (2) (2019), 701--747.




\bibitem{MNY}
C. \textsc{Miao}, Y. \textsc{Nie}, W. \textsc{Ye}, \textit{Onsager's type conjecture for the inviscid Boussinesq equations}, J. Funct. Anal. 287 (7) (2024), Paper No. 110527, 52 pp.


\bibitem{MY}
C.  \textsc{Miao}, W. \textsc{Ye}, \textit{
On the weak solutions for the MHD systems with controllable total energy and cross helicity}, J. Math. Pures Appl., 181 (9) (2024), 190--227.


\bibitem{MoS}
 S. \textsc{Modena}, Jr, L. \textsc{Sz\'{e}kelyhidi}, \emph{Non-uniqueness for the transport equation with Sobolev vector fields}, Ann. PDE, 4 (2) (2018): , 18--38.

\bibitem{MS}
S. \textsc{M\"{u}ller}, V. \textsc{\v{S}ver\'{a}k}, \emph{Convex integration for Lipschitz mappings and counterexamples to regularity}, Ann. of Math., \emph{(2)} 157 (3) (2003), 715--742.



\bibitem{N20}
M. \textsc{Novack}: \textit{Nonuniqueness of weak solutions to the 3 dimensional quasi-geostrophic equations},
 SIAM J. Math. Anal. 52 (4)  (2020) 3301-3349.

\bibitem{Onsa}
L.  \textsc{Onsager}, \textit{Statistical hydrodynamics}, Nuovo Cimento (9), 6 (1949), Supplemento 2,
 279-287.

\bibitem{ST}
M. \textsc{Sermange}, R.  \textsc{Temam}, \emph{Some mathematical questions related to the MHD equations}, Comm. Pure Appl. Math.,  36 (5) (1983), 635--664.

\bibitem{Sh00}
A. \textsc{Shnirelman}, \emph{ Weak solutions with decreasing energy of incompressible Euler equations}, Commun. Math. Phys., 210 (3) (2000), 541--603.

\bibitem{Tay74}
J.B. \textsc{Taylor}, \emph{Relaxation of toroidal plasma and generation of reverse magnetic fields}, Phys. Rev. Lett., 33 (19) (1974), 1139-1141.


\bibitem{Tay86}
J.B. \textsc{Taylor}, \emph{Relaxation and magnetic reconnection in plasmas}, Rev. Mod. Phys., 58 (3) (1986), 741-763.


\bibitem{TZ17}
T. \textsc{Tao}, L. \textsc{Zhang}, \emph{H\"{o}lder continuous solutions of Boussinesq equation
with compact support}, J. Funct. Anal., 272 (10) (2017), 4334--4402.

\bibitem{TZ18}
T. \textsc{Tao}, L. \textsc{Zhang}, \emph{On the continuous periodic weak solutions of Boussinesq equations}, Siam. J. Math. Anal., 50 (1) (2018), 1120--1162.

\end{thebibliography}
\end{document}